\documentclass[11pt]{article}
\usepackage{amsmath,amssymb,amsfonts,amsthm}
\usepackage{setspace}
\newenvironment{myabstract}{\par\noindent
{\bf Abstract . } \small }
{\par\vskip8pt minus3pt\rm}
\newcounter{item}[section]
\newcounter{kirshr}
\newcounter{kirsha}
\newcounter{kirshb}
\newenvironment{enumarab}{\setcounter{kirshb}{1}
\begin{list}{(\arabic{kirshb})}{\usecounter{kirshb}} }{\end{list}}
\newtheorem{theorem}{Theorem}[section]

\newtheorem{lemma}[theorem]{Lemma}
\newtheorem{corollary}[theorem]{Corollary}
\newenvironment{demo}[1]{\noindent{\bf #1.}\upshape\mdseries}
{\nopagebreak{\hfill\rule{2mm}{2mm}\nopagebreak}\par\normalfont}
\theoremstyle{definition}

\newtheorem{example}[theorem]{Example}
\newtheorem{definition}[theorem]{Definition}
\def\R{\mathbb{R}}
\def\Q{\mathbb{Q}}
\def\C{{\mathfrak{C}}}
\def\Fm{{\mathfrak{Fm}}}

\def\At{{\bf At}}
\def\Nr{{\mathfrak{Nr}}}

\def\Sg{{\mathfrak{Sg}}}
\def\Fm{{\mathfrak{Fm}}}
\def\A{{\mathfrak{A}}}
\def\B{{\mathfrak{B}}}
\def\C{{\mathfrak{C}}}
\def\D{{\mathfrak{D}}}
\def\M{{\mathfrak{M}}}
\def\N{{\mathfrak{N}}}

\def\CA{{\bf CA}}
\def\QA{{\bf QA}}

\def\QEA{{\bf QEA}}

\def\Df{{\bf Df}}

\def\PA{{\bf PA}}
\def\PEA{{\bf PEA}}

\def\K{{\bf K}}
\def\K{{\bf K}}

\def\RCA{{\bf RCA}}

\def\Rd{{\ Rd}}
\def\(R)RA{{\bf (R)RA}}
\def\RA{{\bf RA}}

\def\R{\mathbb{R}}
\def\Q{\mathbb{Q}}

\def\Sc{{\bf Sc}}

\def\Id{{\bf Id}}
\def\c #1{{\cal #1}}
 \def\CA{{\sf CA}}
\def\B{{\sf B}}
\def\G{{\sf G}}
\def\w{{\sf w}}
\def\y{{\sf y}}
\def\g{{\sf g}}
\def\b{{\sf b}}
\def\r{{\sf r}}
\def\K{{\sf K}}

 \def\Cm{{\mathfrak{Cm}}}
\def\Nr{{\mathfrak{Nr}}}

\def\restr #1{{\restriction_{#1}}}
\def\cyl#1{{\sf c}_{#1}}
\def\diag#1#2{{\sf d}_{#1#2}}
\def\sub#1#2{{\sf s}^{#1}_{#2}}
\def\R{\sf R}
\def\Ra{{\mathfrak{Ra}}}
\def\Ca{{\mathfrak{Ca}}}
\def\set#1{\{#1\} }
\def\Ra{{\mathfrak{Ra}}}
\def\Nr{{\mathfrak{Nr}}}
\def\Tm{{\mathfrak{Tm}}}
\def\A{{\mathfrak{A}}}
\def\B{{\mathfrak{B}}}
\def\C{{\mathfrak{C}}}
\def\D{{\mathfrak{D}}}

\def\A{{\mathfrak{A}}}
\def\B{{\mathfrak{B}}}
\def\C{{\mathfrak{C}}}
\def\D{{\mathfrak{D}}}

\def\U{{\mathfrak{U}}}
\def\GG{{\mathfrak{GG}}}

\def\L{{\mathfrak{L}}}
\def\Rd{{\mathfrak{Rd}}}

\def\At{{\mathfrak{At}}}
\def\L{{\mathfrak{L}}}

\def\CA{{\bf CA}}
\def\RA{{\bf RA}}

\def\RCA{{\bf RCA}}

\def\G{{\bf G}}

\def\F{{\mathfrak{F}}}
\def\At{{\sf{At}}}
\def\N{\mathbb{N}}
\def\R{\mathfrak{R}}

\def\sub#1#2{{\sf s}^{#1}_{#2}}
\def\cyl#1{{\sf c}_{#1}}
\def\diag#1#2{{\sf d}_{#1#2}}

\def\c #1{{\cal #1}}

\def\pa{$\forall$}
\def\pe{$\exists$}

\def\ef{Ehren\-feucht--Fra\"\i ss\'e}

\def\nodes{{\sf nodes}}

\def\restr #1{{\restriction_{#1}}}

\def\Ra{{\mathfrak{Ra}}}
\def\Nr{{\mathfrak{Nr}}}
\def\Z{{\cal Z}}
\def\CA{{\bf CA}}
\def\RCA{{\bf RCA}}
\def\c#1{{\mathcal #1}}

\def\set#1{ \{#1\}}

\def\Ca{{\mathfrak Ca}}
\def\b#1{{\bar{ #1}}}
\def\pe{$\exists$}
\def\pa{$\forall$}
\def\Cm{{\mathfrak Cm}}
\def\Sg{{\mathfrak Sg}}

\def\Rl{{\mathfrak Rl}}
\def\N{{\cal N}}

\def\At{{\sf At}}

\def\rng{{\sf rng}}
\def\dom{{\sf dom}}
\def\Cm{{\sf Cm}}
\def\Mat{{\sf Mat}}
\def\w{{\sf w}}
\def\g{{\sf g}}
\def\y{{\sf y}}
\def\r{{\sf r}}

\def\cyl#1{{\sf c}_{#1}}
\def\sub#1#2{{\sf s}^{#1}_{#2}}
\def\diag#1#2{{\sf d}_{#1#2}}
\def\RQEA{{\sf RQEA}}
\def\ws{winning strategy}
\def\ef{Ehren\-feucht--Fra\"\i ss\'e}

 \def\CA{{\sf CA}}

\def\RCA{{\sf RCA}}

\def\RA{{\sf RA}}
\def\PA{{\sf PA}}

\def\PEA{\sf PEA}
\def\QEA{{\sf QEA}}

\def\y{{\sf y}}
\def\g{{\sf g}}
\def\r{{\sf r}}
\def\w{{\sf w}}
\def\Z{{\mathbb{Z}}}
\def\N{{\mathbb{N}}}

\def\U{{\mathfrak{U}}}

\def\c{{\sf c}}
\def\s{{\sf s}}

\def\d{ Dedekind--MacNeille}
\def\Id{{\sf Id}}
\def\Sc{{\sf Sc}}
\def\Df{{\sf Df}}

\def\K{{\sf K}}
\def\nodes{{\sf nodes}}

\def\G{{\bold G}}
\def\Sc{{\sf Sc}}
\def\Df{{\sf Df}}
\def\PA{{\sf PA}}
\def\Id{{\sf Id}}
\def\QEA{{\sf QEA}}
\def\s{{\sf s}}

\def\CA{{\sf CA}}
\def\K{{\sf K}}
\def\QA{{\sf QA}}
\def\RCA{{\sf RCA}}
\def\RQEA{{\sf RQEA}}
\def\A{{\mathfrak{A}}}

\def\cyl#1{{\sf c}_{#1}}
\def\sub#1#2{{\sf s}^{#1}_{#2}}
\def\diag#1#2{{\sf d}_{#1#2}}

\def\la#1{\langle#1\rangle}
\def\G{{\mathfrak{G}}}

\def\de{Dedekind-MacNeille}
\def\Cm{{\mathfrak{Cm}}}
\def\M{{\sf M}}

\title{Problems on neat embeddings solved by 
rainbow constructions and Monk algebras, a survey of results and methods}
\author{Tarek Sayed Ahmed\\
Department of Mathematics, Faculty of Science,\\
Cairo University, Giza, Egypt.
 }
\date{}
\begin{document}
\maketitle

\begin{myabstract} 
This paper is a survey of recent results and methods in (Tarskian) algebraic logic. 
We focus on cylindric algebras.  
Rainbow constructions are used to solve problems on classes consisting of algebras having a neat embedding property 
substantially generalizing seminal results of Hodkinson as well as Hirsch and 
Hodkinson on atom--canonicity and complete representations, respectively. For proving non--atom--canonicity of infinitely many 
varieties approximating the variety
of representable algebras of finite dimension $>2$, so called 
blow up and blur constructions are used. Rainbow constructions are compared to constructions using 
Monk--like algebras and cases where both constructions work are given.  
Notions of representability, like complete, weak and strong are lifted from atom structures to atomic algebras
and investigated in terms of neat embedding properties. The classical results of Monk and Maddux on non-finite axiomatizability of the classes of 
representable relation and cylindric algebras of finite dimension $>2$ are reproved using also a blow up and blur construction.
Applications to omitting types for $n$--variable fragments of 
first order logic, for $2<n<\omega$, are given. The main results of the paper are summarized in tabular form at the end of the paper.
\end{myabstract}

\section{Introduction}

\subsection{An overview}  

The purpose of this paper is to present recent developments from algebraic logic and logic
in an integrated format that is accessible to the non--specialist and informative for the practitioner.  
Using quite sophisticated (relatively recent) techniques from algebraic logic, 
(like so--called rainbow constructions and games), 
our intention is to unify, illuminate and generalize several existing results 
scattered in the literature, hopefully stimulating further research.
We focus on Tarskian algebraic logic, specifically cylindric algebras. 

We follow the notation of \cite{1} which is in conformity with the notation adopted in the monograph \cite{HMT2}.
In particular, $\CA_\alpha$ denotes the class of cylindric algebras of dimension $\alpha$, $\alpha$ an ordinal, $\sf RCA_\alpha$ denotes the class of
representable $\CA_\alpha$s and for any ordinal $\beta>\alpha$, ${\sf Nr}_\alpha\CA_\beta$ denotes the class of all $\alpha$--neat reducts of $\CA_\beta$s to be defined in a moment.
Three cornerstones in the development of the theory of cylindric algebras due to Tarski, Henkin and Monk, respectively - the last
two involving the notion of neat reducts - are
the following. 

Tarski proved that every locally finite infinite dimensional cylindric algebra of dimension $\omega$ $(\sf Lf_{\omega})$ is representable.
Henkin \cite[Theorem 3.2.10]{HMT2} proved what has become to be known as the 
{\it the neat embedding theorem}, which says that  for any ordinal $\alpha$, $\bold S{\sf Nr}_\alpha\CA_{\alpha+\omega}= \sf RCA_{\alpha}$, where $\bold S$ denotes the operation of forming subalgebras. 
This generalize's Tarski's representation theorem substantially, because $\sf Lf_{\omega}\subsetneq \bold S{\sf Nr}_{\omega}\CA_{\omega+\omega}$.
Monk \cite{Monk} proved that for any ordinal $\alpha>2$ and $k\in \omega$,  $\bold S{\sf Nr}_\alpha\CA_{\alpha+k}\neq \RCA_\alpha$. In particular,
for each finite $n>2$, and $k\in \omega$, there is an algebra $\A_k\in \bold S{\sf Nr}_n\CA_{n+k}$ that is not representable.
Any non--trivial ultraproduct of such algebras will be in $\sf RCA_n$.
Hence, the variety $\sf RCA_n$ ($2<n<\omega$) is  not finitely axiomatizable.

The notion of {\it neat reducts} and the related one of {\it neat embeddings} are both important in algebraic logic for the
very simple reason that they are very much tied
to the notion of representability, via the neat embedding theorem of Henkin's.

\begin{definition} 
Let  $\alpha<\beta$ be ordinals and $\B\in \CA_{\beta}$. Then the {\it $\alpha$--neat reduct} of $\B$, in symbols
$\Nr_{\alpha}\B$, is the
algebra obtained from $\B$, by discarding
cylindrifiers and diagonal elements whose indices are in $\beta\setminus \alpha$, and restricting the universe to
the set $Nr_{\alpha}B=\{x\in \B: \{i\in \beta: {\sf c}_ix\neq x\}\subseteq \alpha\}.$
\end{definition}
If $\A\in \CA_\alpha$ and $\A\subseteq \Nr_\alpha\B$, with $\B\in \CA_\beta$, then we say that $\A$ {\it neatly embed} in $\B$, and 
that $\B$ is a {\it $\beta$--dilation of $\A$}, or simply a {\it dilation} of $\A$ if $\beta$ is clear 
from context. We say that $\A$ has {\it a neat embedding property} and that $\A$ has {\bf the} neat embedding property if $\beta\geq \alpha+\omega.$
It is known that  $\A$ has the neat embedding property $\iff \A\in \bold S{\sf Nr}_{\alpha}\CA_{\alpha+\omega}$ \cite[Theorem 2.6.35]{HMT1} 
and by Henkin's neat embedding  theorem both are equivalent to $\A\in \RCA_{\alpha}$.

Monk's result was refined by Hirsch and Hodkinson by showing that for finite $n\geq 3$ and $k\geq 1$, the
variety $\bold S{\sf Nr}_n\CA_{n+ k+1}$ is not even finitely axiomatizable over $\bold S{\sf Nr}_n\CA_{n+k}$ answering (the famous) 
\cite[Problem 2.12]{HMT2}.
This result was lifted to infinite dimensions
by Robin Hirsch and the present author
\cite{t}, addressing other cylindric--like algebras, as well, like Pinter's substitution algebras and Halmos' quasi--polyadic algebras.
Such results will be addressed in some detail in \S6 and will be strengthened for cylindric algebras and quasipolyadic equality algebras of infinite dimensions.
It is known that for $1<n<\beta$, $n$ finite,  the class ${\sf Nr}_n\CA_\beta$ is not first order definable, least a variety \cite[Theorem 5.1.4]{Sayedneat}; this solves 
\cite[Problem 4.4]{HMT2}. We show in theorem \ref{rainbow} that it is not even closed under $\equiv_{\infty, \omega}$,
but that it is pseudo--elementary, and that its elementary theory is recursively enumerable.

From now on fix $2<n<\omega$. Analogous to the aforementioned results, in what follows we prove several results on classes of algebras having a neat embedding property
in connection to the algebraic notion of {\it atom--canonicity} and the semantical one of {\it complete represenations}. 
A variety $V\subseteq \CA_n$ is {\it atom--canonical} if whenever $\A\in V$ is atomic, 
then the complex algebra of its atom structure, in symbols  $\Cm\At\A$, is also in $V$.
In this case $\Cm\At\A$ is the {\it \de\ completion} of $\A$ which is the smallest complete algebra 
containing $\A$ as a dense subalgebra meaning that for all non--zero $b\in \Cm\At\A$, there is a non--zero $a\in A$ 
such that $a\leq b$.  

As the name suggests, complete representability is a {\it semantical} notion. A {\it representation} of $\A\in \CA_n$ is an injective homomorphism
$f:\A\to \wp(V)$, where $V\subseteq {}^nU$ for some non--empty set $U$ is a disjoint union of cartesian squares, 
that is $V=\bigcup_{i\in I}{}^nU_i$, $I$ is a non-empty indexing set, $U_i\neq \emptyset$  
and  $U_i\cap U_j=\emptyset$  for all $i\neq j$; the operations on $\wp(V)$ are the concrete operations 
defined like in cylindric set algebras of dimension $n$ relativized to $V$. A cylindric set algebra having top element such a $V$ is called a {\it generalized set algebra} of dimension $n$.
  
\begin{definition}\label{completerep} A {\it complete representation} of $\A\in \CA_n$ is a representation $f$ of $\A$ that preserves  arbitrary sums carrying
them to set--theoretic unions, that is the representation $f:\A\to \wp(V)$ is required to satisfy   
$f(\sum S)=\bigcup_{s\in S}f(s)$ for all $S\subseteq \A$ such that $\sum S$ exists. 
\end{definition}
We denote the class of algebras having a complete represenations (briefly the class of completely representable algebras) by ${\sf CRCA}_n.$
Ordinary representations are not necessarily complete. It is known that if an algebra has
a complete representation, then it has to be atomic \cite{HH}. 
It is also known that there are countable atomic $\RCA_n$s 
that have no complete representations. So atomicity is necessary but not sufficient for complete 
representability.
The class ${\sf CRCA}_n$s is not even elementary \cite[Corollary 3.7.1]{HHbook2}, this can be distilled from the proof of the first item of theorem \ref{rainbow}.
Nevertheless algebras having countably many atoms in ${\sf CRCA}_n$ {\it can be characterized} via special neat embeddings.
To specify such `special' neat embeddings we need: 

Let $\bold K$ be a class of algebras having a Boolean reduct. Then $\bold S_d\bold K$ denotes the class of dense sublgebras of algebras in $\bold K$
and $\bold S_c\bold K$ denotes the class of complete subalgebras of algebras in $\bold K$. For $\A, \B\in \bold K$, we 
write  $\A\subseteq_d \B$ if $\A$ is {\it dense} in $\B$. 
$\A$ is a {\it complete subalgebra} of $\B$, in symbols $\A\subseteq_c \B$,  if $\A\subseteq \B$ and for all $X\subseteq \A$, $\sum ^{\A}X=1\implies \sum^{\B}X=1$. 
We have $\A\subseteq_d \B\implies \A\subseteq_c \B$, but the converse implication $\Longleftarrow$ is not true in general (it is not true for Boolean algebras).

It is known that if $\A\in \CA_n$ has countably many atoms then $\A\in {\sf CRCA}_n$ $\iff \A\in \bold S_c{\sf Nr}_n\CA_{\omega}$ \cite[Theorem 5.3.6]{Sayedneat} and theorem \ref{ch}.
It is also known that ${\sf Nr}_n\CA_{\omega}\subsetneq  \bold S_d{\sf Nr}_n\CA_{\omega}\subseteq  \bold S_c{\sf Nr}_n\CA_{\omega}$ \cite{SL}.
The last inclusion is proved to be proper below for $n=3$, cf. corollary \ref{finitepa3}.
In particular, if $\A\in \CA_n$ is atomic with countably many atoms and $\A\in \bold S_d{\sf Nr}_n\CA_{\omega}$, then $\A\in {\sf CRCA}_n$.
Below we show that the countability condition cannot be omitted. There is an atomic $\A\in \CA_n$ having uncountably many atoms such that $\A\in {\sf Nr}_n\CA_{\omega}$, 
but $\A\notin {\sf CRCA}_n$. But it can be easily shown (as done below) 
that such an algebra belongs to the elementary closure of the class ${\sf CRCA}_n$ 
re--establishing that ${\sf CRCA}_n$ 
is not elementary. In fact, it will be shown that any atomic algebra in 
${\sf Nr}_n\CA_{\omega}$ satisfies the so--called {\it Lyndon conditions}, which are an infinite set of first order sentences $\rho_k$ $(k\in \omega)$; each
$\rho_k$  encodes a winning strategy in a $k$--rounded game denoted by $G_k$ to be addressed in a moment. It is known, at least implicitly, that this 
last elementary class, namely, ${\sf Mod}\{\rho_k: k\in \omega\}$  
coincides with the elementary closure of ${\sf CRCA}_n$, cf. theorem \ref{main2}.   

Such a semantical notion (of complete representability) is also closely
related to the algebraic notion of {\it atom--canonicity} of $\RCA_n$ which is an important {\it persistence property} in modal logic,
and to the metalogical property of  omitting types in finite variable fragments of first order logic
\cite[Theorems 3.1.1-2, p.211, Theorems 3.2.8, 9,10]{Sayed}, 
when non--principal types are omitted with respect to usual Tarskian semantics.
The typical question is: given a $\A\in \CA_n$ and a family $(X_i: i\in I)$ of meets ($I$ a non--empty set), 
is there a representation $f:\A\to \wp(V)$
that carries  this set of meets to set theoretic intersections, in the sense that  $f(\prod X_i)=\bigcap_{i\in I}f(x)$ for all $i\in I$? 

When the algebra $\A$ is countable, $|I|\leq \omega$ and $\prod X_i=0$ for all $\in I$, this is an algebraic version of an  omitting types theorem; 
the representation $f$ {\it omits the given set of meets (or non-principal types).}  
When it is {\it only one meet} consisting of co-atoms, in an atomic algebra, such a representation $f$ will 
a {\it complete one}, and 
this is equivalent to that $f(\prod X)=\bigcap_{x\in X}f(x)$ for all $X\subseteq \A$ 
for which $\prod X$ exists in $\A$ \cite{HH}. The last condition is an algebraic version 
of of Vaught's theorem for first order logic, namely, the unique (up to isomorphism) atomic, 
equivalently prime,  model of a countable 
atomic theory omits all 
non--principal types.

These connections will be further elaborated upon below in $\S7$. It will be shown that the (seemingly purely algebraic)
 result of non--atom canonicity of $\bold S{\sf Nr}_n\CA_{n+k}$ for $k\geq 3$, proved in theorem \ref{can} 
implies failure of the omitting types theorems for $n$--variable 
fragments of first order logic  $(2<n<\omega)$, even if we substantially broaden the class of models omitting a given family of 
non--principal types, considering so--called $n+3$--flat models, 
in place of ordinary models 
which are $\omega$--flat. In fact, we prove more. We will show that {\it Vaught's theorem} (which is a consequence of the omitting types theorem in first order logic) 
fails in the (strong) 
sense that there is an atomic 
$L_n$ theory, such that the non--principal type of co--atoms, 
cannot be omitted in an $n+3$--flat model, {\it a fortiori} it 
cannot be omitted in an ordinary 
($\omega$--flat) model. 

The chapter \cite{HHbook2} is devoted to studying various types of atom structures like
{\it completely representable} atom structures, atom structures satisfying the {\it Lyndon conditions}, the 
{\it strongly representable} atom structures, and {\it weakly representable} atom structures, 
all of dimension $n$. 
Now one can lift such notions from working on {\it atom structures (the frame level)} to working 
on the {\it (complex) algebra level} restricting his attension to atomic ones, and investigate 
such notions of represenatbility in term of neat embedding properties. 
We initiate this task in theorem \ref{main2}, which is likely to be rewarding, 
but by no means do we end it.

\subsection{On the techniques used}

We continue to fix finite $n>2$. 
We frequently use games as devised by Hirsch and Hodkinson \cite{HHbook, HHbook2} played on so--called {\it atomic networks} on a cylindric (rainbow) atom structure (to be defined below).
The $k$--rounded `usual' atomic game ($k\leq \omega$) played on an atomic $\A\in \CA_n$ between \pa\ and \pe\ is denoted by $G_k(\At\A)$ \cite[Definition 3.3.2]{HHbook2}.  
We devise `truncated versions' $F^m$, $G^m_{\omega}$
of the above games. These games have $\omega$ rounds, but the number of nodes in networks used during the play is limited to $m$ where $2<n<m$.
$F^m$ is like $G^m_{\omega}$ except that $F^m$, \pa\ has the bonus to reuse the $m$ nodes in play. When $m\geq \omega$, these games reduce to the usual
$\omega$--rounded atomic game $G_{\omega}$, definition \ref{game}. 
Then the game $F^m$ {\it is related to the existence of $m$--dilations for an algebra $\A\in \CA_n$} in the following sense.
Assume that $2<n<m$.

(*) If $\A\in \bold S_c{\sf Nr}_n\CA_m\implies$ \pe\ has a \ws\ in $F^m(\At\A)$.
If $\A$ is finite and  $\A\in \bold S{\sf Nr}_n\CA_m\implies$ \pe\ has a \ws\ in $F^m(\At\A)$, cf. lemma \ref{n}.

As a sample of the hitherto obtained results:   

(1) A finite algebra $\D\in \CA_n$ for which \pe\ can win $F^{n+3}(\At\D)$,   
so that by (the contrapositive of the second part of) (*) $\D\notin \bold S{\sf Nr}_n\CA_{n+3}$, can be embedded into the \de\ completion of an
atomic 
(infinite) countable $\A\in \RCA_n$. 
From this, we conclude that the \d\ completion of $\A$, namely, $\Cm\At\A$, 
is outside $\bold S{\sf Nr}_n\CA_{n+3},$ 
because the last class is a variety hence closed under $\bold S$ and $\D\subseteq \Cm\At\A$.  

Since $\A\in \bold S{\sf Nr}_n\CA_{n+k}$ for all $k\geq 3$ and $\Cm\At\A\notin \bold S{\sf Nr}_n\CA_{n+3}(\supseteq \bold S{\sf Nr}_n\CA_{n+k}$, $k\geq 3$), 
we conclude that $\bold S{\sf Nr}_n\CA_{n+k}$  is not {\it atom--canonical} for all $k\geq 3$. This is proved in theorem \ref{can}.

(2) We construct, for any finite $n>2$,  an atomic  algebra $\C\in {\sf RCA}_n$
with countably many atoms, such that $\C=\Cm\At\C$, for  which \pa\ can win $F^{n+3}(\At\C)$  
but \pe\ can win $G_k(\At\C)$ for all finite $k$.
It  follows from (the contrapositive of the first part of) (*) that $\C\notin \bold S_c{\sf Nr}_n\CA_{n+3}$. 
Using ultrapowers followed by an elementary chain argument (a standard procedure in such constructions), we get that $\C$ is elementary
equivalent to a countable completely representable algebra $\B$ \cite[Corollary 3.3.5]{HHbook2}, 
so that $\B\in \bold S_c{\sf Nr}_n\CA_{\omega}$, cf. \cite[Theorem 5.3.6]{Sayedneat} and theorem \ref{ch}. 
We conclude that any class $\bold K$ such that $\bold S_c{\sf Nr}_n\CA_{\omega}\subseteq \bold K\subseteq \bold S_c{\sf Nr}_n\CA_{n+3}$, 
$\bold K$ is not elementary, because $\A\notin \bold K\subseteq \bold S_c{\sf Nr}_n\CA_{n+3}$, $\A\equiv \B$ and 
$\B\in \bold S_c{\sf Nr}_n\CA_{\omega}\subseteq \bold K$. This is proved in the first item of theorem \ref{rainbow}. 
In the third item of {\it op.cit} we replace the $\bold S_c$ on the left by $\bold S_d$.

To prove (1) and (2) we use {\it rainbow constructions for cylindric algebras} \cite{HH, HHbook2}.

Throughout the paper we use fairly standard notation, which as indicated above, is in conformity with the notation in \cite{1}. 
Any deviation from such notation will be explicitly mentioned and any possibly unfamiliar notation will be explained  at its first occurrence in the text.
We assume familiarity with only the (very) basics of cylindric algebra theory. In this respect the paper is fairly self--contained.

We make the following convention which we have adopted so far and will stick to it till the end of the paper. We denote infinite ordinals by $\alpha, \beta\ldots$ 
and finite ordinals by $n, m\ldots$. Ordinals which are arbitary meaning that they could be finite or infinite 
will be denoted by  $\alpha,\beta\ldots$. 
\subsection{Layout}
\begin{itemize}

\item In $\S3$ after the preliminaries, we show that for any $2<n<\omega$ and any $k\geq 3$, the variety $\bold S{\sf Nr}_n\CA_{n+k}$ is not atom--canonical.

\item In $\S4$ we show that several classes consisting of algebras having a neat embedding property are not first order definable. 
As a sample,  we show that
for any finite $n>1$, the class ${\sf Nr}_n\CA_{\omega}$ is not closed under $\equiv_{\infty, \omega}$ and 
that for any $2<n<\omega$, any class $\bold K$, such that $\bold S_d{\sf Nr}_n\CA_{\omega}\subseteq \bold K\subseteq \bold S_c{\sf Nr}_n\CA_{n+3}$, $\bold K$
is not elementary, cf. theorem \ref{rainbow}.

\item In \S5 we lift  various notions of representability formulated for atom structures to atomic algebras, and we investigate such notions in terms of neat embeddings, cf. theorem \ref{main2}.

\item In \S6 we compare rainbow algebras to Monk--like algebras, and we reprove Monk's celebrated result on non--finite axiomatizability for both 
representable relation and cylindric algebras of finite dimension $>2$, 
cf. example \ref{Monk}.  We strengthen the result in \cite{t} for $\CA$s of infinite dimension, cf. theorem \ref{2.12a} 
and we review the  main results in \cite{t} in connection to the famous neat embedding problem \cite[Problem 2.12]{HMT1} solved by Hirsch and Hodkinson.

\item In the last section, we reep the harvest of the algebraic result proved in theorem \ref{can}.
Together with variations on the flexible construction in \cite{ANT}, omitting types theorems for the clique guarded (finite variable) fragments 
of first order logic are investigated, cf. theorem \ref{2.12}. The results of Maddux on non--finite axiomatizability 
(for representable relation and cylindric algebras)  refining Monk's results are reproved. 
  \end{itemize}

\section{Preliminaries}

Algebras will be denoted by Gothic letters, and when we write $\A$ for an algebra, then we shall be tacitly assuming that 
$A$ denotes its universe, that is 
$\A=\langle A, f_i^{\A}\rangle_{i\in I}$ where $I$ is a non--empty set and $f_i$ $(i\in I)$ are the operations in the signature of $\A$ 
interpreted via $f_i^{\A}$ in $\A$. For better readability, we omit the superscript $\A$ and we write simply 
$\A=\langle A, f_i\rangle_{i\in I}$.

\subsection{Atom structures and complex algebras}

We recall the notions of {\it atom structures} and {\it complex algebra} in the framework of Boolean algebras 
with operators, briefly $\sf BAO$s, cf. \cite[Definition 2.62, 2.65]{HHbook}. 

\begin{definition}\label{definition}(\textbf{Atom Structure})
Let $\A=\langle A, +, -, 0, 1, \Omega_{i}:i\in I\rangle$ be
an atomic $\sf BAO$ with non--Boolean operators $\Omega_{i}:i\in I$. Let
the rank of $\Omega_{i}$ be $\rho_{i}$. The \textit{atom structure}
$\At\A$ of $\A$ is a relational structure
$$\langle At\A, R_{\Omega_{i}}:i\in I\rangle$$
where $At\A$ is the set of atoms of $\A$ 
and $R_{\Omega_{i}}$ is a $(\rho(i)+1)$-ary relation over
$At\A$ defined by
$$R_{\Omega_{i}}(a_{0},
\cdots, a_{\rho(i)})\Longleftrightarrow\Omega_{i}(a_{1}, \cdots,
a_{\rho(i)})\geq a_{0}.$$
\end{definition}
\begin{definition}(\textbf{Complex algebra})
Conversely, if we are given an arbitrary first order structure
$\mathcal{S}=\langle S, r_{i}:i\in I\rangle$ where $r_{i}$ is a
$(\rho(i)+1)$-ary relation over $S$, called an {\it atom structure}, we can define its
\textit{complex
algebra}
$$\Cm(\mathcal{S})=\langle \wp(S),
\cup, \setminus, \phi, S, \Omega_{i}\rangle_{i\in
I},$$
where $\wp(S)$ is the power set of $S$, and
$\Omega_{i}$ is the $\rho(i)$-ary operator defined
by$$\Omega_{i}(X_{1}, \cdots, X_{\rho(i)})=\{s\in
S:\exists s_{1}\in X_{1}\cdots\exists s_{\rho(i)}\in X_{\rho(i)},
r_{i}(s, s_{1}, \cdots, s_{\rho(i)})\},$$ for each
$X_{1}, \cdots, X_{\rho(i)}\in\wp(S)$.
\end{definition}
An atom structure will be denoted by $\bf At$.  An atom structure $\bf At$ has the signature of $\CA_\alpha$, $\alpha$ an ordinal, 
if  $\Cm\bf At$ has the signature of $\CA_\alpha$, in which case we say that $\bf At$ is an $\alpha$--dimensional atom structure. 

\begin{definition}\label{canonical} 
Let $V$ be a variety of $\sf CA_n$s. Then $V$ is {\it atom--canonical} if whenever $\A\in V$ and $\A$ is atomic, then $\mathfrak{Cm}\At\A\in V$.
The {\it  \d\ completion} of  $\A\in \CA_n$, is the unique (up to isomorphisms that fix $\A$ pointwise)  complete  
$\B\in \CA_n$ such that $\A\subseteq \B$ and $\A$ is {\it dense} in $\B$.
\end{definition}

If $\A\in \CA_n$ is atomic,  then 
$\mathfrak{Cm}\At\A$ 
is the {\it \d\ completion of $\A$.}
If $\A\in \CA_n$, then its atom structure will be denoted by $\At\A$ with domain the set of atoms of $\A$ denoted by $At\A$.
We deal only with atom structure having the signature of $\CA_n$.
{\it Non atom--canonicity} can be proved 
by finding {\it weakly representable atom structures} that are not {\it strongly representable}.
\begin{definition} Let $\alpha$ be an ordinal. An atom structure $\bf At$ of dimension $\alpha$ is {\it weakly representable} if there is an atomic 
$\A\in \RCA_{\alpha}$ such that $\At\A=\bf At$.  The atom structure  $\bf At$ is {\it strongly representable} if for all $\A\in \CA_{\alpha}$, 
$\At\A=\bf At \implies \A\in {\sf RCA}_{\alpha}$.
\end{definition}
Fix $2<n<\omega$. Then these two notions (strong and weak representability) do not coincide for cylindric algebras as proved by Hodkinson \cite{Hodkinson}. 
This gives that $\RCA_n$ is {\it not} atom--canonical
and  that $\RCA_n$ is not closed under \d\ completions. 
In theorem \ref{can}, we generalize Hodkinson's result by showing that there are two atomic $\CA_n$s sharing the same atom structure, one is representable and the other is even outside 
$\bold S{\sf N}r_n\CA_{n+3}(\supsetneq \RCA_n$). In particular, there is a $\CA_n$ outside $\bold S{\sf Nr}_n\CA_{n+3}$ having a dense representable 
subalgebra.

\subsection{Atomic games}

We need  the notions of {\it atomic networks} and {\it atomic games} \cite{HHbook, HHbook2}:

\begin{definition}\label{game} Fix finite $n>1$. 

(1) An {\it $n$--dimensional atomic network} on an atomic algebra $\A\in \CA_n$  is a map $N: {}^n\Delta\to  \At\A$, where
$\Delta$ is a non--empty set of {\it nodes}, denoted by $\nodes(N)$, satisfying the following consistency conditions: 
\begin{itemize}
\item If $x\in {}^n\nodes(N)$, and $i<j<n$, then $N(x)\leq {\sf d}_{ij}\iff x_i=x_j$.
\item If $x, y\in {}^n\nodes(N)$, $i<n$ and $x\equiv_i y$, then  $N(x)\leq {\sf c}_iN(y)$.
\end{itemize}
Let $i<n$. For $n$--ary sequences $x$ and $y$ and $n$--dimensional atomic networks $M$ and $N$,  we write $x\equiv_ iy $ $\iff y(j)=x(j)$ for all $j\neq i$
and we write $M\equiv_i N\iff M(y)=N(y)$ for all $y\in {}^{n}(n\setminus \{i\})$.

(2)   Assume that $\A\in \CA_n$ is  atomic and that $m, k\leq \omega$. 
The {\it atomic game $G^m_k(\At\A)$, or simply $G^m_k$}, is the game played on atomic networks
of $\A$
using $m$ nodes and having $k$ rounds \cite[Definition 3.3.2]{HHbook2}, where
\pa\ is offered only one move, namely, {\it a cylindrifier move}: 
\begin{itemize}
\item Suppose that we are at round $t>0$. Then \pa\ picks a previously played network $N_t$ $(\nodes(N_t)\subseteq m$), 
$i<n,$ $a\in \At\A$, $x\in {}^n\nodes(N_t)$, such that $N_t(x)\leq {\sf c}_ia$. For her response, \pe\ has to deliver a network $M$
such that $\nodes(M)\subseteq m$,  $M\equiv _i N$, and there is $y\in {}^n\nodes(M)$
that satisfies $y\equiv _i x$, and $M(y)=a$.  
\end{itemize}
(3)  We write $G_k(\At\A)$, or simply $G_k$, for $G_k^m(\At\A)$ if $m\geq \omega$.
The {\it atomic game $F^m(\At\A)$, or simply $F^m$}, is like $G^m_{\omega}(\At\A)$ except that
\pa\ has the advantage to reuse the available
$n$ nodes during the play.
\end{definition}
Now we approach the notion of complete representations as defined in \ref{completerep}. 
It is known \cite{HH} that $f:\A\to \B$ is a complete representation of $\A$ $\iff$ $\A$ is atomic  
and $f$ is {\it atomic},  in the sense that $\bigcup_{x\in \At\A}f(x)=1^{\B}$.
For $n<\omega$ (recall that) we denote the class of completely representable $\CA_{n}$s by ${\sf CRCA}_{n}$.

\begin{theorem}\label{rep} Let $2<n<\omega$ and $\A\in \CA_n$ be atomic with countable many atoms. 
Then \pe\ has a \ws\ in $G_k(\At\A)$ for all $k\in \omega\iff \A\equiv \B$ for some 
$\B\in {\sf CRCA}_n$.  \pe\ has a \ws\ in  $G_{\omega}(\At\A)\iff \A$ is completely representable. In particular, if $\A$ is finite, then 
\pe\ has a \ws\ in $G_{\omega}(\At\A)\iff \A$ 
is representable $\iff$ \pe\ has a \ws\ in $G_k(\At\A)$ 
for all $k\in \omega$.
\end{theorem}
\begin{proof} \cite[Theorem 3.3.3]{HHbook2}. 
\end{proof} 
It is known that for any finite $n$, the class ${\sf CRCA}_n$ coincides with the class 
$\bold S_c{\sf Nr}_n\CA_{\omega}$,  
(where recall that $\bold S_c$ denotes the class of forming complete subalgebras) on algebras having countably many atoms \cite[Theorem 5.3.6]{Sayedneat}, cf. theorem \ref{ch}
and the corollary following it. 
A truncated version of
theorem \ref{rep} is the following lemma. To prove it we need a technical definition.
\begin{definition}\label{sub} Let $m$ be a finite ordinal $>0$. An $\sf s$ word is a finite string of substitutions $({\sf s}_i^j)$ $(i, j<m)$,
a $\sf c$ word is a finite string of cylindrifications $({\sf c}_i), i<m$;
an $\sf sc$ word $w$, is a finite string of both, namely, of substitutions and cylindrifications.
An $\sf sc$ word
induces a partial map $\hat{w}:m\to m$:
\begin{itemize}

\item $\hat{\epsilon}=Id,$

\item $\widehat{w_j^i}=\hat{w}\circ [i|j],$

\item $\widehat{w{\sf c}_i}= \hat{w}\upharpoonright(m\smallsetminus \{i\}).$

\end{itemize}
If $\bar a\in {}^{<m-1}m$, we write ${\sf s}_{\bar a}$, or
${\sf s}_{a_0\ldots a_{k-1}}$, where $k=|\bar a|$,
for an  arbitrary chosen $\sf sc$ word $w$
such that $\hat{w}=\bar a.$
Such a $w$  exists by \cite[Definition~5.23 ~Lemma 13.29]{HHbook}.
\end{definition}

\begin{lemma}\label{n} Assume that $2<n<m<\omega$. If $\A$ is atomic and 
$\A\in \bold S_c{\sf Nr}_n\CA_m$, then \pe\ has a \ws\ in $F^m(\At\A)$. 
In particular, if $\A\in {\sf Nr}_n\CA_{\omega}$, 
then \pe\ has a \ws\ in $F^{\omega}(\At\A)$ and $G_{\omega}(\At\A)$, and if 
$\A$ is finite and \pa\ has a \ws\ in $F^m(\At\A)$, 
then $\A\notin \bold S{\sf Nr}_n\CA_{m}$. 
\end{lemma}
\begin{proof} The proof  lifts the ideas
in \cite[Lemma 29, 26, 27]{r} formulated for relation algebras to $\CA$s.
This is tedious but not too hard. We give (more than) an  outline.
Fix $2<n<m$. Assume that $\C\in\CA_m$, $\A\subseteq_c\Nr_n\C$ is an
atomic $\CA_n$ and $N$ is an $\A$--network with $\nodes(N)\subseteq m$. Define
$N^+\in\C$ by
\[N^+ =
 \prod_{i_0,\ldots, i_{n-1}\in\nodes(N)}{\sf s}_{i_0, \ldots, i_{n-1}}{}N(i_0,\ldots, i_{n-1}).\]
Here the substitution operator is defined as in definition \ref{sub}.
For a network $N$ and  function $\theta$,  the network
$N\theta$ is the complete labelled graph with nodes
$\theta^{-1}(\nodes(N))=\set{x\in\dom(\theta):\theta(x)\in\nodes(N)}$,
and labelling defined by
$$(N\theta)(i_0,\ldots, i_{n-1}) = N(\theta(i_0), \theta(i_1), \ldots,  \theta(i_{n-1})),$$
for $i_0, \ldots, i_{n-1}\in\theta^{-1}(\nodes(N))$.
The following can be proved:

(1) for all $x\in\C\setminus\set0$ and all $i_0, \ldots, i_{n-1} < m$, there is $a\in\At\A$, such that
${\sf s}_{i_0,\ldots, i_{n-1}}a\;.\; x\neq 0$,
Then using (1) one can show:

(2) for any $x\in\C\setminus\set0$ and any
finite set $I\subseteq m$, there is a network $N$ such that
$\nodes(N)=I$ and $x\cdot N^+\neq 0$. Furthermore, for any networks $M, N$ if
$M^+\cdot N^+\neq 0$, then
$M\restr {\nodes(M)\cap\nodes(N)}=N\restr {\nodes(M)\cap\nodes(N)},$

(3) if $\theta$ is any partial, finite map $m\to m$
and if $\nodes(N)$ is a proper subset of $m$,
then $N^+\neq 0\rightarrow {(N\theta)^+}\neq 0$. If $i\not\in\nodes(N),$ then ${\sf c}_iN^+=N^+$.

Using the above proven facts,  one shows that \pe\  has a \ws\ in $F^m$; she can always
play a network $N$ with $\nodes(N)\subseteq m,$ such that
$N^+\neq 0$.
In the initial round, let \pa\ play $a\in \At\A$.
\pe\ plays a network $N$ with $N(0, \ldots, n-1)=a$. Then $N^+=a\neq 0$.
Recall that here \pa\ is offered only one (cylindrifier) move.
At a later stage, suppose \pa\ plays the cylindrifier move, which we denote by
$(N, \langle f_0, \ldots, f_{n-2}\rangle, k, b, l).$
He picks a previously played network $N$,  $f_i\in \nodes(N), \;l<n,  k\notin \{f_i: i<n-2\}$,
such that $b\leq {\sf c}_l N(f_0,\ldots,  f_{i-1}, x, f_{i+1}, \ldots, f_{n-2})$ and $N^+\neq 0$.
Let $\bar a=\langle f_0\ldots f_{i-1}, k, f_{i+1}, \ldots f_{n-2}\rangle.$
Then by  second part of  (3)  we have that ${\sf c}_lN^+\cdot {\sf s}_{\bar a}b\neq 0$
and so  by first part of (2), there is a network  $M$ such that
$M^+\cdot{\sf c}_{l}N^+\cdot {\sf s}_{\bar a}b\neq 0$.
Hence $M(f_0,\dots, f_{i-1}, k, f_{i-2}, \ldots$ $, f_{n-2})=b$,
$\nodes(M)=\nodes(N)\cup\set k$, and $M^+\neq 0$, so this property is maintained.

The last part follows by observing that for any $\C\in \CA_n$, if $\C\in \bold S{\sf Nr}_n\CA_m\implies \C^+\in \bold S_c{\sf Nr}_n\CA_m$
(where $\C^+$ is the canonical extension of $\C$) and if  $\C$ is finite,  then of course $\C=\C^+$.

\end{proof}

\subsection{Rainbow constructions}

Rainbow constructions involve `colours' as the name suggests, and 
quite sophisticated machinery from 
finite combinatorics and graph theory. For $\CA_n$s $(2<n<\omega$), the rainbow atom structure consists of certain {\it coloured graphs}. 
Here the atoms are {\it graphs}. Such coloured graphs (atoms) to be defined below are  
complete graphs of size at most $n$ whose edges are labelled by the rainbow colours. 
Some hyperedges are also labelled. While $n$--coloured graphs will be the atoms of a rainbow $\CA_n$, the board of a rainbow game will 
consist of coloured graphs.

Fix $2<n<\omega$. 

In general, {\it a coloured graph}, with no restriction on its size \cite{HH}, is a complete graph 
whose edges are labelled by the rainbow colours, $\g$ (greens), $\r$ (reds), and 
$\w$ (whites) satisfying certain consistency conditions. The greens are 
$\{\g_i: 1\leq i< n-1\}\cup \{\g_0^i: i\in \sf G\}$  and the reds are $\{\r_{ij}: i,j \in \sf R\}$ where
$\sf G$ and $\sf R$ are two relational structures. The whites are $\w_i: i\leq n-2$.
In coloured graphs certain triangles are forbidden.
For example a green triangle (a triangle whose edges are all green)  is forbidden. Not all red triangles are allowed.  
In consistent (allowed) red triangle the indices 
`must match' satisfying a certain `consistency condition'. Also, in coloured graphs some $n-1$ tuples (hyperedges) are labelled by shades of yellow \cite{HH}. 
More specifically the following are forbidden triangles in coloured graphs:
$(\g, \g^{'}, \g^{*}), (\g_i, \g_{i}, \w_i), \mbox{any }1\leq i\leq  n-2,$  
$(\g^j_0, \g^k_0, \w_0)\mbox{ any } j, k\in \sf G$, and finally 
$(\r_{ij}, \r_{j'k'}, \r_{i^*k^*})\mbox{ unless }i=i^*,\; j=j'\mbox{ and }k'=k^*,$
and no other triple of colours is forbidden \cite[4.3.3]{HH}. 
Given relational structures 
$\sf G$ and $\sf R$ the rainbow 
atom structure of dimension $n$ are equivalence classes of surjective maps $a:n\to \Delta$, where $\Delta$ is a coloured graph
in the rainbow signature, and the equivalence relation relates two such maps $\iff$  they essentially define the same graph \cite[4.3.4]{HH};
the nodes are possibly different but the graph structure is the same. We let $[a]$ denote the equivalence class containing $a$.

The accessibility binary relation corresponding 
to the $i$th  cylindrifier $(i<n)$ is defined by:  $[a] T_i [b]\iff a\upharpoonright n\setminus \{i\}=b\upharpoonright n\setminus \{i\},$ 
and the accessibility unary relation corresponding to the $ij$th diagonal element ($i<j<n$) 
is defined by: $[a]\in D_{ij}\iff a(i)=a(j)$. We refer to the atom $[a]$ ($a:n\to \Delta$) as a rainbow atom. 
We denote the complex algebra of the rainbow atom structure based on 
$\sf G$ and $\sf R$ by $\CA_{\sf G, R}$. The dimension of $\CA_{\sf G, R}$ will be clear 
from context.

Certain special finite coloured graphs play an essential role (in rainbow games). 
Such special coloured graphs are called {\it cones}:

{\it Let $i\in {\sf G}$, and let $M$ be a coloured graph consisting of $n$ nodes
$x_0,\ldots,  x_{n-2}, z$. We call $M$ {\it an $i$ - cone} if $M(x_0, z)=\g_0^i$
and for every $1\leq j\leq m-2$, $M(x_j, z)=\g_j$,
and no other edge of $M$
is coloured green.
$(x_0,\ldots, x_{n-2})$
is called  the {\bf base of the cone}, $z$ the {\bf apex of the cone}
and $i$ the {\bf tint of the cone}.}

If $\A$ is an (atomic)  rainbow $\CA_n$, 
then the atomic games $G^m_k$ and $F^m$  translate to games on coloured graphs, 
cf. \cite[p.27--29]{HH}. The typical \ws\ of \pa\ in the rainbow game played on coloured graphs played between \pe\ and \pa\  
is bombarding \pe\ with $i$--cones, $i\in \sf G$, having the 
same base 
and distinct green tints.  To respect the rules of the game \pe\ has to choose a red label for appexes of two succesive cones.  
Eventually, running out of `suitable reds',  \pe\ is forced to play an inconsistent triple of reds where indices do not match.
Thus {\it \pa\ wins on a red clique} (a graph all of whose edges are lablled by a red).  
Such a \ws\ is dictated by a simple \ef\ forth  
game played on the relational structures $\sf G$ and $\sf R$ denoted by ${\sf EF}_r^p(\sf G, R)$, here $r$ is the number of rounds and $p$ is the number of pebble pairs 
\cite[Definition 16.2]{HHbook2}. 
\section{Atom--canonicity}

\subsection{Blowing up and blurring a finite rainbow algebra}

The next theorem refines the seminal result of Hodkinson's \cite{Hodkinson} which is the limiting case when $k=\omega$.
It is fully proved in \cite{mlq}. We start with an outline. Then we get more technical giving more than the gist of the idea of the proof
which is {\it blowing up and blurring a finite rainbow algebra}.

\begin{theorem}\label{can} Let $2<n<\omega$. Then there exists a countable atomic $\A\in \RCA_n$ such that
$\Cm\At\A\notin \bold S{\sf Nr}_n\CA_{n+3}$.  In particular, the variety $\bold S{\sf Nr}_n\CA_{n+k}$ is not atom--canonical for any $k\geq 3.$ 
\end{theorem}

\subsection*{Model--theoretic outline of proof of theorem \ref{can}}

The outline of proof we present now of the above bolded statement 
focuses more on explaining the main ideas and is punctuated by some comments.
The outline is divided to three parts. In the first part we construct a certain model on which 
our constrution is based. 

{\bf (1) The model:} Fix $2<n<\omega$ and $1\leq k\leq \omega$. We have a rainbow signature $L_{ra}$ \cite[Definition 3.6.9 (i)]{HHbook2}. 
The signature $L_{ra}$ has, among other symbols determined by the other colours, $n+k$ green binary relations $\g_0^i: 1\leq i\leq n+k$ 
and  $\omega$--many red binary relations $\r_{ij}^l$, $i<j<n$, and $l\in \omega$. The other colours are irrelevant for the moment 
as far as this outline is concerned.  
The signature is like the rainbow signature in \cite{Hodkinson}, except that here we have $n+k$  many 
green binary relations. When $k=\omega$  the above rainbow signature coincides with the rainbow signature used in constructing 
the algebra denoted by $\A$ in \cite[Definition 4.1]{Hodkinson}.
(When $k=\omega$, by $n+\omega$, we mean ordinal addition so that $n+\omega=\omega$.)
We add to $L_{ra}$ an additional binary relation symbol $\rho$ forming the strict expansion $L=L_{ra}\cup \{\rho\}$.
The rainbow construction implemented here can be coded in a  theory $T$ expanding the rainbow theory 
$T_{ra}$, the latter given in \cite[Definition 3.6.9(ii)]{HHbook2}. 
The signature $L_{ra}$ of $T_{ra}$ is expanded to the language $L$ forming $T$,  and $T$ stipulates finitely many first order formulas 
spelling out `consistency conditions' for the new binary relation  $\rho$ in connection to other relation symbols (colours) in $L$ including itself (examples will be given shortly). 
In the present context, the theory $T$ is a first order
theory $\iff$ the green binary relations are finite. The presence of countably (infinitely) many green relation 
symbols makes the rainbow theory an $L_{\omega_1, \omega}$ theory \cite[Top of p.83]{HHbook2}.

A {\it coloured graph} is a model of $T_{ra}$. 
An {\it $n$--coloured graph} is a coloured graph of size at most $n$. By  {\it an extended coloured graph}, we understand a model of $T$ (in $L$).  
An {\it $n$--extended coloured graph} is a coloured graph of size at most $n$ allowing $\rho$ as a label.
If $a, b\in \Delta$, $\Delta$ an extended colured graph, and 
$(a, b)\in \g$ for some binary green  relation $\g$, say,  {\it we say that the edge $(a, b)$ is labelled by $\g$} or labelled by a green. 
This terminology will apply to other relations (colours)
in the signature. 
By a {\it red graph}, we mean an extended coloured graph that has an edge labelled by some red.  
In the signature $L_{ra}$ there are $\omega$--many (distinct) red $n$--coloured graphs.
By a {\it red clique} we undertand an extended coloured graph {\it all of whose edges} are labelled by a red, possibly $\rho$.
One constructs like in \cite{Hodkinson}, in the spirit of Fr\"aisse constructions,  a countable (infinite) $n$--homogeneous  model $M$ 
of $T$, as the limit of a play whose board consists of models of $T$, namely,  extended 
coloured graphs. The triplets $(\rho, \r,  \r')$ and $(\rho, \rho, \r)$ are consistent for any $\r$ and $\r'$  in $L_{ra}$, 
meaning that any extended coloured graph played during the game, as well as the limit, can contain the triangles whose edges are labelled by such colours as 
a subgraph; other red triples are forbidden and {\it all} green triangles are forbidden

The limit of the play $M$ is also an extended rainbow graph. Here by $n$--homogeneity, is meant that every $n$--coloured graph embeds into $M$, and that 
such coloured graphs are uniquely determined by
their isomorphism types, regardless of their
location in $M$: 
If $\triangle \subseteq \triangle'$ are extended coloured graphs,  $|\triangle'|
\leq n$, and $\theta : \triangle \rightarrow M$ is an embedding,
then $\theta$ extends to an embedding $\theta' : \triangle'
\rightarrow M$.

By the homogenuity built in $M$ such $n$--coloured graphs will constitute the atoms of the (relativized) set algebras 
based on $M$ as specified in a while; the representable algebra, and its non--representable \d\ completion.
This  game is played between \pe\ and \pa. As is the case with `rainbow games' \cite{HH, HHbook} 
\pa\ challenges \pe\  with  {\it cones} having  green {\it tints $(\g_0^i)$}, 
and \pe\ wins if she can respond to such moves. This is the only way that \pa\ can force a win.  \pe\ 
has to respond by labelling {\it appexes} of two succesive cones, having the {\it same base} played by \pa.
By the rules of the game, she has to use a red label. The \ws\ is implemented by \pe\  using the red label $\rho$ outside the rainbow signature 
that comes to her rescue  whenever she runs out of `rainbow reds', so she can respond with extended coloured graphs. 
It turns out  inevitable, that some edges in $M$ 
are  labelled by $\rho$ during the play; in fact these edges labelled by $\rho$ will form an {\it infinite red clique}
(an infinite complete extended graph whose edges are all labelled by $\rho$.)

{\bf (2) The set algebra and its \d\ completion based on the model:} Now we forget about the red label $\rho$ for a while. 
All formulas are now taken in the rainbow signature $L_{ra}$. 
Having $M$ at hand, one constructs  two atomic $n$--dimensional set algebras based on $M$, sharing the same atom structure and having 
the same top element.  
The first set algebra $\A$ is the $L_{n}$
formula set algebra having top element $W\subseteq{} ^nM$ to be specified shortly.  
The second algebra $\C$ is the $L_{\infty, \omega}^n$ formula set algebra having the same top element $W$. 
The set  $W$ is obtained from $^nM$ by discarding assignments whose
edges are labelled by $\rho$, in symbols 
$W = \{ \bar{a} \in {}^n M : M \models ( \bigwedge_{i < j <n} \neg \rho(x_i, x_j))(\bar{a}) \},$
For $\phi\in L_{\infty, \omega}^n$,  let $\phi^W=\{s\in W: M\models \phi[s]\}.$
Then $\A$ has universe $\{\phi^W: \phi \text{ an $L_n$ formula}\}$, 
and $\C$ has universe consisting of all $\phi^{W}$, but now $\phi$ is an $L_{\infty, \omega}^n$ formula.
In both cases the operations are the usual concrete 
operations of cylindric set algebras, read off from 
the semantics of the connectives, relativized to $W$. Plainly $\A\subseteq \C$.

The atoms for both algebras are determined by so--called $\sf MCA$ formulas \cite[Definition 4.3]{Hodkinson} in the rainbow signature $L_{ra}$.  
Every such $\sf MCA$ formula defines a rainbow atom, a surjective map $a:n\to \Delta$, so that $\Delta$ is an
$n$--coloured graphs (in the rainbow signature).  
Now we have $\At\A=\At\C$. The common atom structure will be
denoted by $\bf At$ in the more technical proof to follow.
The $n$--homogeneity built into $M$  now plays another crucial role. The classical semantics with respect to $^nM$ and relativized semantics obtained by restricting assignments to $W$ 
agree, with respect to  {\it first order formulas using $n$ variables}, which is 
not the case with $L_{\infty, \omega}^n$ formulas taken in the same (rainbow) signature without the red label $\rho$.
This can be proved using 
$n$ back--and--forth systems 
induced by any permutation on the set $\omega\cup \{\rho\}$, cf. \cite[Proposition 3.13]{Hodkinson}. 
Hence the set algebra $\A$ is isomorphic to a cylindric set algebra 
having top element $^nM$. 

{\bf (3) Blowing up and blurring a finite rainbow algebra:} When $k=\omega$,  Hodkinson \cite{Hodkinson} proves that $\C$ is not representable
using a semantical argument \cite[\S 5.2]{Hodkinson}.
So we are certain that $\Cm\At\A\notin \bold S{\sf Nr}_n\CA_{n+m}$ for some finite $m>0$,
because $\bigcap_{i>0}\bold S\sf Nr_n\CA_{n+i}=\sf RCA_n$ and $\Cm\At\A\notin {\sf RCA}_n.$ 
But the $\omega$--greens used, {\it do not 
give us any information on such an $m$, the dimension of this dilation}; for example what is the least such $m$?
When does the \de\ completion `stop to be representable?'
We proceed differently varying the parameter $k$. We use a so--called {\it blow up and blur construction}, a highly indicative term introduced in \cite{ANT}. 
This is a syntactical approach. 
By choosing $k=1$, one can embed a finite algebra $\D$ into $\C$ such that $\D$ is outside $\bold S{\sf Nr}_n\CA_{n+m}$, $m\geq 3$ and so will be $\C$.

Let us elaborate some more. 
Assume that $k<\omega$. The \ef\ forth game ${\sf EF}_{n+k}^{n+k}(n+k, n)$ \cite[Definition 16.2]{HHbook}, is played between \pa\ and \pe\ on the structures $n+k$ 
and $n$ viewed as complete irreflexive graphs.
It is obvious that \pa\ has a \ws\ in $n+1$ rounds. In each round $0,1,\ldots, n+1$, \pa\ places a  new pebble  on  an element of $n+k$.
The edges relation in $n+k$ is irreflexive so to avoid losing
\pe\ must respond by placing the other  pebble of the pair on an unused element of $n+k$.
After $n+1$ rounds there will be no such element,
and she loses in the next round.
This game can be lifted to the graph game on the finite rainbow algebra $\CA_{n+k, n}$, based on $n+k$ (the greens) 
and $n$ (reds) where \pa\ has a \ws\ in finitely many rounds  rounds using $n+k+2$ nodes, that is in the game 
$F^{n+k+2}_{\omega}\At(\CA_{n+k, n})$. In fact, he does not need to reuse nodes, so \pa\ actually wins the harder game $G^{n+k+2}_{\omega}\At(\CA_{n+k, n}).$
One can embed 
$\CA_{n+k, n}$ into $\C=\Cm\At\A$.  
It follows by lemma \ref{n} that  $\CA_{n+k, n}\notin \bold S{\sf Nr}_n\CA_{n+k+2}$.  The smaller $k$ is, the sharper the result we obtain. 
So take $k$ to be the least possible value, namely, $k=1$. In this case the $n+1$ 
greens tell us that $\Cm\At\A\notin \bold S{\sf Nr}_n\CA_{n+m}$ 
for any $m\geq 3$, because now $\CA_{n+1, n}\notin \bold S{\sf Nr}_n\CA_{n+3}$ and $\CA_{n+1, n}$
embeds into $\C$ by mapping every rainbow atom $[a]: n\to \Delta$,  $\Delta$ an $n$--coloured graph in the finite rainbow signature of $\CA_{n+1, n}$ to the join of its 
copies.  A copy of $[a]:n\to \Delta$ is a rainbow atom in $\C$ of the form $[b]:n\to \Delta'$ where $\Delta'$ is isomorphic to $\Delta$ modulo altering 
superscripts of the reds. In particular, the copy of any $[a]: n\to \Delta$ where $\Delta$ is not red is itself.

{\it We say that $\A$ and $\C$ are obtained by  {\bf blowing up and blurring} $\CA_{n+1, n}$. 
The algebraic structure of $\CA_{n+1, n}$  is {\bf blurred} in $\A$; $\CA_{n+1, n}$ does not embed in $\A$. 
On the other hand,  $\CA_{n+1, n}$ is  {\bf not  blurred} in $\C$, 
because $\CA_{n+1, n}$ {\bf embeds in $\C$}.}

\subsection{More (technical) details}

The argument used, as indicated above,   is a combination of the rainbow construction in \cite{Hodkinson} which is implemented model--theoretically, 
together with the blow up and blur construction used in \cite{ANT}. Here we proceed `the other way round'. We start 
where we ended in the above sketch.
We embed the finite (rainbow) algebra $\D=\CA_{n+1, n}\notin \bold S{\sf Nr}_n\CA_{n+3}$ 
in the \d\ completion of an atomic (infinite) algebra $\A\in {\sf RCA}_n$, where $\A$ 
is obtained by blowing up and blurring $\D$. The `blowing up' is done by splitting the red atoms of $\D$ each into
infinitely many atoms (of  $\A$). 

{\bf (1) Blowing up and blurring  $\CA_{n+1, n}$ forming a weakly representable atom structure $\At$:}
Take the finite rainbow cylindric algebra $R(\Gamma)$
as defined in \cite[Definition 3.6.9]{HHbook2},
where $\Gamma$ (the reds) is taken to be the complete irreflexive graph $m$, and the greens
are  $\{\g_i:1\leq i<n-1\}
\cup \{\g_0^{i}: 1\leq i\leq n+1\}$ so tht $\sf G$ is the complete irreflexive graph $n+1$.
Call this finite rainbow $n$--dimensional cylindric algebra, based on ${\sf G}=n+1$ and ${\sf R}=n$
$\CA_{n+1, n}$  and denote its finite atom structure by $\bf At_f$.
One  then replaces each  red colour
used in constructing  $\CA_{n+1, n}$  by infinitely many with superscripts from $\omega$, 
getting a weakly representable atom structure $\bf At$, that is,
the term algebra $\Tm\bf At$ is representable.
The resulting atom structure (with $\omega$--many reds),  call it $\bf At$, 
is the rainbow atom structure that is like the atom structure of the (atomic relativized set) 
algebra $\A$ defined in \cite[Definition 4.1]{Hodkinson} 
except that we have $n+1$ greens
and not infinitely many as is the case in \cite{Hodkinson}.
Everything else is the same. In particular, the rainbow signature \cite[Definition 3.6.9]{HHbook2} now consists of $\g_i: 1\leq i<n-1$, $\g_0^i: 1\leq i\leq n+1$,
$\w_i: i<n-1$,  $\r_{kl}^t: k<l< n$, $t\in \omega$,
binary relations, and $n-1$ ary relations $\y_S$, $S\subseteq n+1$.
There is a shade of red $\rho$; the latter is a binary relation that is {\it outside the rainbow signature}.

But $\rho$ is used as a  label for  coloured graphs built during a `rainbow game', and in fact, \pe\ can win the rainbow $\omega$--rounded game
and she builds an $n$--homogeneous (coloured graph) model $M$ as indicated in the above outline by using $\rho$ when
she is forced a red \cite[Proposition 2.6, Lemma 2.7]{Hodkinson}.
Then $\Tm\At$ is representable as a set algebra with unit $^nM$; this can be proved exactly as in \cite{Hodkinson}.
In fact, $\Tm\bf At\subseteq \A$, with $\A$ as described in the preceding outline.

{\bf (2) Embedding $\CA_{n+1, n}$ into the \d\ completion of $\Tm\bf At$:} 
We embed $\CA_{n+1, n}$ into  the complex algebra $\Cm\bf At$, the \d\ completion of $\Tm\bf At$.
Let ${\sf CRG}_f$ denote  the class of coloured graphs on 
$\bf At_f$ and $\sf CRG$ be the class of coloured graph on $\bf At$. We 
can assume that  ${\sf CRG}_f\subseteq \sf CRG$.
Write $M_a$ for the atom that is the (equivalence class of the) surjection $a:m\to M$, $M\in \sf CGR$.
Here we identify $a$ with $[a]$; no harm will ensue.
We define the (equivalence) relation $\sim$ on $\At$ by
$M_b\sim N_a$, $(M, N\in {\sf CGR}):$
\begin{itemize}
\item $a(i)=a(j)\Longleftrightarrow b(i)=b(j),$

\item $M_a(a(i), a(j))=\r^l\iff N_b(b(i), b(j))=\r^k,  \text { for some $l,k$}\in \omega,$

\item $M_a(a(i), a(j))=N_b(b(i), b(j))$, if they are not red,

\item $M_a(a(k_0),\dots, a(k_{n-2}))=N_b(b(k_0),\ldots, b(k_{n-2}))$, whenever
defined.
\end{itemize}
We say that $M_a$ is a {\it copy of $N_b$} if $M_a\sim N_b$. 
We say that $M_a$ is a {\it red atom} if it has at least one edge labelled by a red rainbow colour $\r_{ij}^l$ for some $i<j<n$ and $l\in \omega$. 
Clearly every red atom $M_a$ has infinitely countable many red copies, which we denote by $\{M_a^{(j)}: j\in \omega\}$.
Now we define a map $\Theta: \CA_{n+1, n}=\Cm{\bf At_f}$ to $\Cm\At$,
by  specifing  first its values on ${\sf At}_f$,
via $M_a\mapsto \sum_jM_a^{(j)}$; each atom maps to the suprema of its 
copies.  If $M_a$ is not red,   then by $\sum_jM_a^{(j)}$,  we understand $M_a$.
This map is extended to $\CA_{n+1, n}$ the obvious way by $\Theta(x)=\bigcup\{ \Theta(y):y\in \At\CA_{n+1, n}, y\leq x\}$. The map
$\Theta$ is well--defined, because $\Cm\At$ is complete. 
It is not hard to show that the map $\Theta$ 
is an injective homomorphim. 
Injectivity follows from the fact that $M_a\leq f(M_a)$, hence $\Theta(x)\neq 0$ 
for every atom $x\in \At(\CA_{n+1, n})$.  We check only preservation of cylindrifiers. Let $i<n$. By additivity (of cylindrifiers), 
we restrict our attention to atoms  $M_a\in \bf At_f$ with $a:n\to M$, and $M\in {\sf CRG}_f\subseteq \sf CRG$. Then: 
$$f({\sf c}_i^{\Cm\bf At_f}a)=f (\bigcup_{[c]\equiv_i[a]} M_c)
=\bigcup_{[c]\equiv_i [a]}f(M_c)$$
$$=\bigcup_{[c]\equiv_i [a]}\sum_j M_c^{(j)}
=\sum_j \bigcup_{[c]\equiv_i [a]}M_c^{(j)}$$
$$=\sum _j{\sf c}_i^{\Cm\bf At}M_a^{(j)}
={\sf c}_i^{\Cm\bf At}(\sum_j M_a^{(j)})
={\sf c}_i^{\Cm\bf At}f(a).$$
{\bf (3) Exactly like in above outline, one proves that \pa\ has  a 
\ws\ for \pe\ in $F^{n+3}\At(\CA_{n+1, n})$}
using the usual rainbow strategy by bombarding \pe\ with cones having the same base and distinct green tints.
He needs $n+3$ nodes to implement his \ws. 
Then by lemma \ref{n} this implies that  $\CA_{n+1,n}\notin
\bold S{\sf Nr}_n\CA_{n+3}$. Since $\CA_{n+1,n}$ embeds into $\Cm\bf At$, 
hence $\Cm\bf At$  is outside 
$\bold S{\sf Nr}_n\CA_{n+3}$, too.

\section{First order definability}

Throughout this section, unless otherwise indicated, 
$n$ is a finite ordinal $>1$.
${\sf Gs}_{n}$ is the class of {\it generalized set algebras} of dimension $n$ as defined in the introduction. $\sf Ws_\omega$ is the class of weak set algebras of dimension
$\omega$. An algebra $\A\in {\sf Ws}_{\omega}\iff \A$ has top element $V\subseteq {}^{\omega}U$ where $V$ is the set of all sequences agreeing co--finitely with a fixed in advance 
sequence $p\in {}^{\omega}U$ and the operations of $\A$ are defined like in set algebras restricted to $V$. In conformity with the notation of \cite{HMT2}, 
we denote $V$ (called an $\omega$--dimensional weak space) 
by ${}^{\omega}U^{(p)}$. Recall that ${\sf CRCA}_n$ denotes the class of completely 
representable $\CA_n$s.
\begin{theorem}\label{ch} Let $2<n<\omega$. Then ${\sf CRCA}_n\subseteq \bold S_c{\sf Nr}_n\CA_{\omega}$. Conversely, if $\A\in \bold S_c{\sf Nr}_n\CA_{\omega}$ 
has countably many atoms, then $\A\in {\sf CRCA}_n$. 
\end{theorem}
\begin{proof} The last part follows from \cite[Theorem 5.3.6]{Sayedneat} by noting that if 
$\B$ is atomic having countably many atoms and $\B\in \bold S_c{\sf Nr}_n\CA_{\omega}$, then $\Tm\At\B\subseteq_d \B$, 
so $\Tm\At\B\in \bold S_d\bold S_c{\sf Nr}_n\CA_{\omega}\subseteq \bold S_c\bold S_c{\sf Nr}_n\CA_{\omega}=
\bold S_c{\sf Nr}_n\CA_{\omega}$, and $\Tm\At\B$ is atomic and countable. 
Furthermore, $\Tm\At\B$ is completely representable $\iff$ $\B$ is completely representable, because
they share the same atom structure. The cited theorem \cite[Theorem 5.3.6]{Sayedneat} 
tells us that $\Tm\At\B$ is completely representable, so $\B$ is completely representable, too.

We prove the first inclusion.
Assume that $\C\in \sf Gs_{n}$ is a complete representable of $\A$ via $t$. 
That is $t:\A\to \C$ is a complete representation. Assume further that $\C$ 
has top element
a disjoint union of the form $\bigcup_{i\in I}{}^nU_i$ ($I$ and $U_i$ non--empty sets and $U_i\cap U_j=\emptyset$ for 
$i\ne j\in I$). For $i\in I$, let $E_i={}^nU_i$. Fix $f_i\in {}^{\omega}U_i$. Let $W_i={}^{\omega}U_i^{(f_i)}$.
Let ${\C}_i=\wp(W_i)$. Then $\C_i$ is atomic; indeed the atoms are the singletons.
Let $x\in \Nr_{n}\C_i$, that is ${\sf c}_jx=x$ for all $n\leq j<\omega$.
Now if  $f\in x$ and $g\in W_i$ satisfy $g(k)=f(k) $ for all $k<n$, then $g\in x$.
Hence $\Nr_{n}\C_i$
is atomic;  its atoms are $\{g\in W_i:  \{g(i):i<n\}\subseteq U_i\}.$

Define $h_i: \A\to \Nr_{n}\C_i$ by
$h_i(a)=\{f\in W_i: \exists a'\in \At\A, a'\leq a;  (f(i): i<n)\in t(a')\}$. 
Let $\D=\bold P _i \C_i$. Let $\pi_i:\D\to \C_i$ be the $i$th projection map.
Now clearly  $\D$ is atomic, because it is a product of atomic algebras,
and its atoms are $(\pi_i(\beta): \beta\in \At(\C_i))$.
Now  $\A$ embeds into $\Nr_{n}\D$ via $J:a\mapsto (h_i(a) :i\in I)$. If $x\in \Nr_n\D$,
then for each $i$, we have $\pi_i(x)\in \Nr_{n}\C_i$, and if $x$
is non--zero, then $\pi_i(x)\neq 0$. By atomicity of $\C_i$, there is an $n$--ary tuple $y$, such that
$\{g\in W_i: g(k)=y_k\}\subseteq \pi_i(x)$. It follows that there is an atom
$b\in \A$, such that  $x\cdot  J(b)\neq 0$, and so the embedding is atomic, hence complete.
We have shown that $\A\in \bold S_c{\sf Nr}_{n}\CA_{\omega}.$
and we are done.
\end{proof}
For a class $\sf K$ of $\sf BAO$s, we let $\K\cap \bf At$ denote the class of atomic algebras in $\K$.
The following corollary can be distilled from the above proof since the constructed $\omega$--dilation of the given completely representable
$\CA_n$ is a {\it full generalized weak set algebra} in the sense of \cite[Definition 3.1.2(iv)]{HMT2},
 so it is atomic.  The rest follows from lemma \ref{n} and the second part  of theorem \ref{ch}. 
\begin{corollary}
Assume that $2<n<\omega$. Then $\A\in \CA_{n}$ is completely representable 
$\implies \A\in \bold S_c{\sf Nr}_{n}(\CA_{\omega}\cap {\bf At})\implies  \A$ is atomic and $\A\in \bold S_c{\sf Nr}_{n}\CA_{\omega}
\implies$ \pe\ has  a \ws\ in $G_{\omega}$ and $F^{\omega}.$ 
All  reverse implications hold, if $\A$ has countably many atoms.
\end{corollary}
We note that {\it not all} of the above implications can be reversed as shown in the last item of the coming theorem \ref{rainbow}, see also theorem \ref{pa} 
addressing {\it atomic} $\omega$--dilations.
In the first item of the next theorem we generalize the main result 
in \cite{HH}. The latter result shows that the class of completely representable $\CA_n$s, for $2<n<\omega$ 
is not elementary.

To formulate 
and prove the next theorem, we need to fix some notation.
$\bold S_d$ is the operation of forming dense subalgbras. 
For $\A\in \CA_n$, $n\geq 3$, $\Ra\A$ is the relation algebra reduct of 
$\A$ as defined in \cite[Definition 5.3.7]{HMT2}. 

For relation algebras we follow the terminology of \cite{HHbook} with a single deviation. We denote the {\it identity relation} by $\sf Id$ 
rather than $1'$. In particular, $\sf RA$ denotes the class of relation algebras and $\sf (C)RRA$ denotes the class of (completely) representable 
$\sf RA$s. For $n\geq 4$ 
and $\A\in \CA_n$, $\Ra\A\in \sf RA$ \cite[Theorem 5.3.8]{HMT2}.
For $\sf K\subseteq \CA_n$, ${\sf Ra}\sf K$ denotes the class $\{\Ra\A: \A\in {\sf K}\}$.

For a class $\K$, we let ${\bf El}\K$ denote the elementary closure of 
$\sf K$, that is, the smallest
elementary class containing $\K$. For a class $\sf K$ of $\sf BAO$s, we write 
$\At\K$ for $\{\At\A: \A\in \K\cap \bf At\}$. 
Let $2<n<m$. Consider the class $\bold N_m=\{\A\in \CA_n\cap {\bf At}: \Cm{\bf At}\A\in {\sf Nr}_n\CA_m\}$. 
We will see that   
$\bold N_m\neq {\sf Nr}_n\CA_m$ by item (3)
of the forthcoming theorem \ref{rainbow}.

Two other closely related (but not identical) classes are 
$\bold C_m=\{\A\in \CA_n\cap {\bf At}: \Cm{\At}\A\in \bold S_c{\sf Nr}_n\CA_m\}$ and $\bold C_m^{\sf at}=\{\A\in \CA_n\cap {\bf At}: \Cm{\bf At}\A\in \bold S_c{\sf Nr}_n(\CA_m\cap {\bf At})\}$.
For the definitions of pseudo--elementary and pseudo--universal,  the reader is referred to \cite[Definition 9.1]{HHbook}.
It known that if $\bold K$ is pseudo--universal $\implies \bold K$ is elementary and closed under $\bold S$, cf. \cite[Chapter 10]{HHbook} for an extensive overview of such notions. 

\begin{theorem}\label{rainbow}
Let $2<n<\omega$ and let $k\geq 3$.
\begin{enumarab} 
\item For any class $\bold K$,
such that ${\sf CRCA}_n\cap {\bf S}_c{\sf Nr}_n\CA_{\omega}\subseteq \bold K\subseteq \bold S_c{\sf Nr}_n\CA_{k}$, 
$\bold K$ is not elementary. In particular, the class
${\sf CRCA}_n$ is not elementary \cite{HH}. Furthermore, the classes $\bold C_k$ and $\bold C_k^{\sf at}$ are not elementary.
\item For any class $\bold K$,
such that ${\sf CRCA}_n\cap {\bf S}_d{\sf Nr}_n\CA_{\omega}\subseteq \bold K\subseteq \bold S_c{\sf Nr}_n\CA_{k}$, 
$\bold K$ is not elementary.
Furthemore, any class $\bold L$ such that $\At({\sf Nr}_n\CA_{\omega})\subseteq \bold L\subseteq \bold \At(\bold S_c{\sf Nr}_n\CA_{n+3})$ 
and the class $\bold N_k$ are not elementary. Finally, ${\bf El}{\sf Nr}_n\CA_{\omega}\nsubseteq {\bf S}_d{\sf Nr}_n\CA_{\omega}\iff$ 
any class $\bold L$ such that  ${\sf Nr}_n\CA_{\omega}\subseteq \bold L\subseteq \bold \bold S_c{\sf Nr}_n\CA_{n+3}$, $\bold L$ is not elementary.

\item Let $\alpha$ be any ordinal $>1$.
Then for every infinite cardinal $\kappa\geq |\alpha|$, there exist completely representable algebras
$\B, \A\in \CA_{\alpha}$, that are weak set algebras, such that $\At\A=\At\B$, $|\At\B|=|\B|=\kappa$,
$\B\notin {\bf El}{\sf Nr}_{\alpha}\CA_{\alpha+1}$,
$\A\in {\sf Nr}_{\alpha}\CA_{\alpha+\omega}$,  and $\Cm\At\B=\A$,
so that $|\A|=2^{\kappa}$. In particular,
${\sf Nr}_{\alpha}\CA_{\beta}\subsetneq \bold S_d{\sf Nr}_{\alpha}\CA_{\beta}$. 

\item The classes  ${\sf CRCA}_n$ and ${\sf Nr}_n\CA_m$ for $n<m$ are pseudo--elementary but not elementary, nor pseudo--universal. 
Furthermore, their elementary 
theory is recursively enumerable. For any $n<m$, the class ${\sf Nr}_n\CA_m$ 
is not closed under $L_{\infty,\omega}$ equivalence.
\item There is an atomic $\R\in {\sf Ra}\CA_{\omega}\cap {\bf El}\sf CRRA$ that is not completely representable.
Also, there is an  atomic algebra  $\A\in {\sf Nr}_{n}\CA_{\omega}\cap {\bf El}{\sf  CRCA_n}$, 
that is not completely representable. 
In particular, both $\sf CRRA$ and  ${\sf CRCA}_n$ are not elementary \cite{HH}.
\end{enumarab}
\end{theorem}
\begin{proof}
(1) \cite{mlq} Fix finite $n>2$. One takes an algebra $\A_{\Z, \N}$ based on the ordered structure $\Z$ and $\N$, that is similar to the rainbow algebra 
$\CA_{\Z, \N}$ but not identical. The rainbow colours (signatures) are the same.
In particular, the reds ${\sf R}$ constitute the set $\{\r_{ij}: i<j<\omega(=\N)\}$ and the green colours used 
constitute the set $\{\g_i:1\leq i <n-1\}\cup \{\g_0^i: i\in \Z\}$. 
In complete coloured graphs the forbidden triples are like 
in $\CA_{\Z, \N}$   
but now  the additional triple  $(\g^i_0, \g^j_0, \r_{kl})$ is also forbidden if $\{(i, k), (j, l)\}$ is {\it not an order preserving partial function} from
$\Z\to\N$.
For the sake of brevity, we  write $\C$ for $\A_{\Z, \N}$ throughout the whole 
proof. 
Then \pe\ has a \ws\ $\rho_k$ in the $k$--rounded game $G_k(\At\C)$ for all $k\in \omega$ \cite{mlq}.
Hence, using ultrapowers and an elementary chain argument  \cite[Corollary 3.3.5]{HHbook2}, one gets a countable algebra $\B$ 
such that $\B\equiv \A$, and  \pe\ has a \ws\ in $G_{\omega}(\At\B)$. 

The reasoning is as follows: We can assume that $\rho_k$ is deterministic. 
Let $\D$ be a non--principal ultrapower of $\C$.  Then \pe\ has a \ws\ $\sigma$ in $G_{\omega}(\D)$ --- essentially she uses
$\rho_k$ in the $k$'th component of the ultraproduct so that at each
round of $G_{\omega}(\D)$,  \pe\ is still winning in co--finitely many
components, this suffices to show she has still not lost. We can assume that $\C$ is countable by replacing it, without loss, by $\Tm\At\C$.
Winning strategies are preserved. Now one can use an
elementary chain argument to construct countable elementary
subalgebras $\C=\A_0\preceq\A_1\preceq\ldots\preceq\ldots \D$ in this manner.
One defines  $\A_{i+1}$ to be a countable elementary subalgebra of $\D$
containing $\A_i$ and all elements of $\D$ that $\sigma$ selects
in a play of $G_{\omega}(\D)$ in which \pa\ only chooses elements from
$\A_i$. Now let $\B=\bigcup_{i<\omega}\A_i$.  This is a
countable elementary subalgebra of $\D$, hence $\B\equiv \C$, because $\C\equiv \D$,  
and clearly \pe\ has a \ws\ in
$G_{\omega}(\B)$. Then $\B$ 
is completely representable by \cite[Theorem 3.3.3]{HHbook2}. 

On the other hand, one can show that \pa\ has a \ws\ in $F^{n+3}(\At\C)$.
The idea here, is that, as is the case with \ws's of \pa\ in rainbow constructions, 
\pa\ bombards \pe\ with cones having distinct green tints demanding a red label from \pe\ to appexes of succesive cones.
The number of nodes are limited but \pa\ has the option to re-use them, so this process will not end after finitely many rounds.
The added order preserving condition relating two greens and a red, forces \pe\ to choose red labels, one of whose indices form a decreasing 
sequence in $\N$.  In $\omega$ many rounds \pa\ 
forces a win, 
so $\C\notin \bold S_c{\sf Nr}_n\CA_{n+3}$.
He plays as follows: In the initial round \pa\ plays a graph $M$ with nodes $0,1,\ldots, n-1$ such that $M(i,j)=\w_0$
for $i<j<n-1$
and $M(i, n-1)=\g_i$
$(i=1, \ldots, n-2)$, $M(0, n-1)=\g_0^0$ and $M(0,1,\ldots, n-2)=\y_{\Z}$. This is a $0$ cone.
In the following move \pa\ chooses the base  of the cone $(0,\ldots, n-2)$ and demands a node $n$
with $M_2(i,n)=\g_i$ $(i=1,\ldots, n-2)$, and $M_2(0,n)=\g_0^{-1}.$
\pe\ must choose a label for the edge $(n+1,n)$ of $M_2$. It must be a red atom $r_{mk}$, $m, k\in \N$. Since $-1<0$, then by the `order preserving' condition
we have $m<k$.
In the next move \pa\ plays the face $(0, \ldots, n-2)$ and demands a node $n+1$, with $M_3(i,n)=\g_i$ $(i=1,\ldots, n-2)$,
such that  $M_3(0, n+2)=\g_0^{-2}$.
Then $M_3(n+1,n)$ and $M_3(n+1, n-1)$ both being red, the indices must match.
$M_3(n+1,n)=r_{lk}$ and $M_3(n+1, r-1)=r_{km}$ with $l<m\in \N$.
In the next round \pa\ plays $(0,1,\ldots n-2)$ and re-uses the node $2$ such that $M_4(0,2)=\g_0^{-3}$.
This time we have $M_4(n,n-1)=\r_{jl}$ for some $j<l<m\in \N$.
Continuing in this manner leads to a decreasing
sequence in $\N$.
Let $k\geq 3$ and let $\bold K$ be as in the statement. Then $\C\notin \bold K$, $\B\in \bold K\cap {\sf CRCA}_n$ and 
$\C\equiv \B$, we are done. 
$\bold C_k^{\sf at}=\bold S_c{\sf Nr}_n(\CA_k\cap {\bf At})$, hence by the above it is not elementary.

For non--elementarity of $\bold C_k$, we have $\C\equiv \B$, $\C\notin \bold S_c{\sf Nr}_n\CA_k$ and $\B$ is completely representable, 
hence it is $\bold S_c{\sf Nr}_n\CA_{\omega}$.

(2) We first give the general idea. Let $\C=\A_{\Z, \N}$ be as defined in the previous item.
One can (and will) define a $k$--rounded atomic game stronger than $G_k$ call it $H_k$, for $k\leq \omega$,  
so that if $\B\in \CA_n$ is countable and atomic and \pe\ has a \ws\ in $H_{\omega}(\At\B)$, 
then 
(*) $\At\B\in \At{\sf Nr}_n\CA_{\omega}$ and $\Cm\At\B\in {\sf Nr}_n\CA_{\omega}$. 
One shows that \pe\ has a \ws\ in $H_k(\At\C)$ for all 
$k\in \omega$, hence using ultrapowers and an elementary chain argument, we get that $\C\equiv \B$, 
for some countable completely representable $\B$ that satisfies the two conditions in (*).
Since $\B\subseteq_d \Cm\At\B$, we get the required result, because $\B\in \bold S_d{\sf Nr}_n\CA_{\omega}$
and as before $\C\notin \bold S_c{\sf Nr}_n\CA_{n+3}$ and $\C\equiv \B$. 
Now we prove the second part. Let $\bold L$ be as specified and $\B$ and $\C(=\A_{\Z, \N}$) be the algebras constructed above. Since an 
atom structure of an algebra is first order interpretable in the algebra, then we have $\B\equiv \C\implies \At\B\equiv \At\C$.
Furthermore $\At\B\in \At({\sf Nr}_n\CA_{\omega})\subseteq \bold L$ (though $\B$ might not be in ${\sf Nr}_n\CA_{\omega}$, cf. the next item) 
and $\At\C\notin  \At(\bold S_c{\sf Nr}_n\CA_{n+3})\supseteq \bold L$. 
The last part follows from the fact that
if $\D\in \CA_n$ is atomic, 
then $\At\D\in   \At(\bold S_c{\sf Nr}_n\CA_{n+3})\iff \D\in \bold S_c{\sf Nr}_n\CA_{n+3}$. 
We conclude that $\bold L$ is not elementary.

We define the game $H$.
But first some preparation.
Fix $2<n<\omega$.

For an $n$--dimensional atomic network on an atomic $\CA_n$ and for  $x,y\in \nodes(N)$, we set  $x\sim y$ if
there exists $\bar{z}$ such that $N(x,y,\bar{z})\leq {\sf d}_{01}$.
Define the  equivalence relation $\sim$ over the set of all finite sequences over $\nodes(N)$ by
$\bar x\sim\bar y$ iff $|\bar x|=|\bar y|$ and $x_i\sim y_i$ for all
$i<|\bar x|$. (It can be easily checked that this indeed an equivalence relation).

A \emph{ hypernetwork} $N=(N^a, N^h)$ over an atomic $\CA_n$
consists of an $n$--dimensional  network $N^a$
together with a labelling function for hyperlabels $N^h:\;\;^{<
\omega}\!\nodes(N)\to\Lambda$ (some arbitrary set of hyperlabels $\Lambda$)
such that for $\bar x, \bar y\in\; ^{< \omega}\!\nodes(N)$
if $\bar x\sim\bar y \Rightarrow N^h(\bar x)=N^h(\bar y).$
If $|\bar x|=k\in \N$ and $N^h(\bar x)=\lambda$, then we say that $\lambda$ is
a $k$-ary hyperlabel. $\bar x$ is referred to as a $k$--ary hyperedge, or simply a hyperedge.
We may remove the superscripts $a$ and $h$ if no confusion is likely to ensue.

A hyperedge $\bar{x}\in {}^{<\omega}\nodes(N)$ is {\it short}, if there are $y_0,\ldots, y_{n-1}$
that are nodes in $N$, such that
$N(x_i, y_0, \bar{z})\leq {\sf d}_{01}$
or $\ldots N(x_i, y_{n-1},\bar{z})\leq {\sf d}_{01}$
for all $i<|x|$, for some (equivalently for all)
$\bar{z}.$
Otherwise, it is called {\it long.}
A hypernetwork $N$
is called {\it $\lambda$--neat} if $N(\bar{x})=\lambda$ for all short hyperedges.

Concerning \pa's  moves, $H_m$ has  $m$ rounds, $m\leq \omega$.  
He can play a cylindrifier move, like before but now played on $\lambda$---neat hypernetworks
with $\lambda$ a constant label on short hyperedges.
Also \pa\ can play a \emph{transformation move} by picking a
previously played $\lambda$--neat hypernetwork $N$ and a partial, finite surjection
$\theta:\omega\to\nodes(N)$, this move is denoted $(N, \theta)$.  \pe's
response is mandatory. She must respond with $N\theta$.
Finally, \pa\ can play an
\emph{amalgamation move} by picking previously played $\lambda$--neat hypernetworks
$M, N$ such that
$M\restr {\nodes(M)\cap\nodes(N)}=N\restr {\nodes(M)\cap\nodes(N)},$
and $\nodes(M)\cap\nodes(N)\neq \emptyset$.
This move is denoted $(M,
N).$
To make a legal response, \pe\ must play a $\lambda$--neat
hypernetwork $L$ extending $M$ and $N$, where
$\nodes(L)=\nodes(M)\cup\nodes(N)$.
We claim that  \pe\ has a \ws\ in $H_{m}(\At\CA_{\Z, \N})$ for each finite $m$.
The analogous proof for relation algebras is rather long
\cite[p.25--31]{r}. We assume that the 
claim is true and  take it from there.

Using the usual technique of forming ultrapowers followed by 
an elementary  chain argument, we get that there exists a countable (completely representable) algebra, which we continue to denote by a slight abuse of notation
also $\B$, such that $\A_{\Z, \N}\equiv \B$, and  \pe\ has a \ws\ on $H(\At\B)$. For brevity, let $\alpha=\At\B$.
Using \pe's \ws\ in $H$, one builds an $\omega$--dilation $\D_a$ of $\B$
for every $a\in \At\B$, based on a structure $\M_a$ in some signature to be specified shortly. 
Strictly speaking, $\M_a$ will be a {\it weak model}, 
where assignments are {\it relativized}, they are required to agree 
co--finitely with a fixed sequence in $^{\omega}\M_a$.
This weak model $\M_a$ is taken in a signature $L$ consisting of one $n$--ary relation for each $b\in \At\B$ and 
a $k$--ary relation symbol  for each hyperedge of length $k$ labelled by $\lambda$ the constant neat hyperlabel. 

For $a\in \At\B$, the weak model $\M_a$ is the limit of the play $H_{\omega}$; in the sense that $\M_a$ is the union of the $\lambda$--neat 
hypernetworks on $\B$ played during the game $H_{\omega}$, with starting point the initial atom $a$ that \pa\ chose in the first move.  
Labels for the edges and hyperedges in $\M_a$ 
are defined the obvious way,  inherited from the $\lambda$--neat hypernetworks played during the $\omega$--rounderd game $H_{\omega}(\At\B)$. 
These are nested, so this labelling is well 
defined giving an 
interpretation of {\it only} the atomic formulas of $L$ in $\M_a$.  
There is some space here in `completing' the interpretation. One uses 
an extension $\L$, not necessarily a proper one, of  $L_{\omega, \omega}$ as a vehicle for constructing 
$\D_a$.  The algebra $\D_a$ will  be a  {\it weak set algebra} based on $\M_a$ of $\L$--formulas taken in the signature $L$. That is the 
base in the sense of \cite[Definition 3.1.1]{HMT2} of $\D_a$ is $\M_a$, 
and the set--theoretic operations of $\D_a$ are read off the connectives in $\L$. 
 In all cases, as long as $\L$ contains $L_{\omega, \omega}$ as a fragment, we get that $\B$ neatly embeds into $\D$, that is $\B\subseteq \Nr_n\D$, where 
$\D={\bf P}_{a\in \At\B}\D_a$.  
There  are three possibilites measuring `how close' $\B$ is to $\Nr_n\D$. We go from the closest to the less close. 
Either (a) $\B=\Nr_n\D$  or (b) $\B\subseteq_d \Nr_n\D$ or  (c) $\B\subseteq_c \Nr_n\D$. From the first part, building 
$\D$ using the weaker game $G$ used in the proof of the previous item, we can get the last possibility. It is reasonable to expect that the stronger $\L$ is, the 
`more  control' $\At\B$ has over the hitherto obtained $\omega$--dilation $\D$; the closer $\B$ 
is  to the neat $n$--reduct of $\D$ based on $\L$-formulas. 
If (a) is true than any $\bold K$ between ${\sf Nr}_n\CA_{\omega}\cap \sf CRCA_n$ and $\bold S_c{\sf Nr}_n\CA_{n+3}$ would be non--elementary.
{\it We could not prove (a)}.
So let us approach the two remaining possibilities (b) and (c).
Suppose we take  $\L=L_{\infty, \omega}$. Then using the fact that in the $\lambda$--neat hypernetworks played during the game $H_{\omega}$ 
short hyperedges are constantly
labelled by $\lambda$, one can show that $\B$ and $\Nr_n\D$ have {\it isomorphic 
atom structures}, in symbols $\At\B\cong \At\Nr_n\D$ as follows.
For brevity, denote the hitherto obtained $\At\B$ by $\alpha$.

Fix some $a\in\alpha$. Using \pe\ s \ws\ in the game $H(\alpha)$ played on $\lambda$--neat hypernetworks $\lambda$ a constant label kept on short hyperedges,
one defines a
nested sequence $M_0\subseteq M_1,\ldots$ of $\lambda$--neat hypernetworks
where $M_0$ is \pe's response to the initial \pa-move $a$, such that:
If $M_r$ is in the sequence and $M_r(\bar{x})\leq {\sf c}_ia$ for an atom $a$ and some $i<n$,
then there is $s\geq r$ and $d\in\nodes(M_s)$
such that  $M_s(\bar{y})=a$,  $\bar{y}_i=d$ and $\bar{y}\equiv_i \bar{x}$.
In addition, if $M_r$ is in the sequence and $\theta$ is any partial
isomorphism of $M_r$, then there is $s\geq r$ and a
partial isomorphism $\theta^+$ of $M_s$ extending $\theta$ such that
$\rng(\theta^+)\supseteq\nodes(M_r)$ (This can be done using \pe's responses to amalgamation moves).

Now let $\M_a$ be the limit of this sequence, that is $\M_a=\bigcup M_i$, the labelling of $n-1$ tuples of nodes
by atoms, and hyperedges by hyperlabels done in the obvious way.
Let $L$ be the signature with one $n$-ary relation for
each $b\in\alpha=\At\B$, and one $k$--ary predicate symbol for
each $k$--ary hyperlabel $\lambda$.
{\it Now we work in $L_{\infty, \omega}.$}
For fixed $f_a\in\;^\omega\!\nodes(\M_a)$, let
$\U_a=\set{f\in\;^\omega\!\nodes(\M_a):\set{i<\omega:g(i)\neq
f_a(i)}\mbox{ is finite}}$.
Now we  make $\U_a$ into the base of an $L$ relativized structure 
${\cal M}_a$ like in \cite[Theorem 29]{r} except that we allow a clause for infinitary disjunctions.
In more detail,  for $b\in\alpha,\; l_0, \ldots, l_{n-1}, i_0 \ldots, i_{k-1}<\omega$, \/ $k$--ary hyperlabels $\lambda$,
and all $L$-formulas $\phi, \phi_i, \psi$, and $f\in U_a$:
\begin{eqnarray*}
{\cal M}_a, f\models b(x_{l_0}\ldots,  x_{l_{n-1}})&\iff&{\cal M}_a(f(l_0),\ldots,  f(l_{n-1}))=b,\\
{\cal M}_a, f\models\lambda(x_{i_0}, \ldots,x_{i_{k-1}})&\iff&  {\cal M}_a(f(i_0), \ldots,f(i_{k-1}))=\lambda,\\
{\cal M}_a, f\models\neg\phi&\iff&{\cal M}_a, f\not\models\phi,\\
{\cal M}_a, f\models (\bigvee_{i\in I} \phi_i)&\iff&(\exists i\in I)({\cal M}_a,  f\models\phi_i),\\
{\cal M}_a, f\models\exists x_i\phi&\iff& {\cal M}_a, f[i/m]\models\phi, \mbox{ some }m\in\nodes({\cal M}_a).
\end{eqnarray*}
We are now
working with (weak) set algebras  whose semantics is induced by $L_{\infty, \omega}$ formulas in the signature $L$,
instead of first order ones.
For any such $L$-formula $\phi$, write $\phi^{{\cal M}_a}$ for
$\set{f\in\U_a: {\cal M}_a, f\models\phi}.$
Let
$D_a= \set{\phi^{{\cal M}_a}:\phi\mbox{ is an $L$-formula}}$ and
$\D_a$ be the weak set algebra with universe $D_a$. 
Let $\D=\bold P_{a\in \alpha} \D_a$. Then $\D$ is a {\it generalized weak set algebra} \cite[Definition 3.1.2 (iv)]{HMT2}.
Let $x\in \D$. Then $x=(x_a:a\in\alpha)$, where $x_a\in\D_a$.  For $b\in\alpha$ let
$\pi_b:\D\to \D_b$ be the projection map defined by
$\pi_b(x_a:a\in\alpha) = x_b$.  Conversely, let $\iota_a:\D_a\to \D$
be the embedding defined by $\iota_a(y)=(x_b:b\in\alpha)$, where
$x_a=y$ and $x_b=0$ for $b\neq a$.  

We show that $\alpha\cong \At\Nr_n\D$ and that $\Cm\alpha\cong \Nr_n\D$. 
The argument used is like the argument used in \cite[Theorem 39]{r} adapted to $\CA$s.
Suppose $x\in\Nr_n\D\setminus\set0$.  Since $x\neq 0$,
then it has a non-zero component  $\pi_a(x)\in\D_a$, for some $a\in \alpha$.
Assume that $\emptyset\neq\phi(x_{i_0}, \ldots, x_{i_{k-1}})^{\D_a}= \pi_a(x)$, for some $L$-formula $\phi(x_{i_0},\ldots, x_{i_{k-1}})$.  We
have $\phi(x_{i_0},\ldots, x_{i_{k-1}})^{\D_a}\in\Nr_{n}\D_a$.
Pick
$f\in \phi(x_{i_0},\ldots, x_{i_{k-1}})^{\D_a}$  
and assume that ${\cal M}_a, f\models b(x_0,\ldots x_{n-1})$ for some $b\in \alpha$.
We show that
$b(x_0, x_1, \ldots, x_{n-1})^{\D_a}\subseteq
 \phi(x_{i_0},\ldots, x_{i_{k-1}})^{\D_a}$.  
Take any $g\in
b(x_0, x_1\ldots, x_{n-1})^{\D_a}$,
so that ${\cal M}_a, g\models b(x_0, \ldots x_{n-1})$.  
The map $\{(f(i), g(i)): i<n\}$
is a partial isomorphism of ${\cal M}_a.$ Here that short hyperedges are constantly labelled by $\lambda$ 
is used.
This map extends to a finite partial isomorphism
$\theta$ of $M_a$ whose domain includes $f(i_0), \ldots, f(i_{k-1})$.
Let $g'\in {\cal M}_a$ be defined by
\[ g'(i) =\left\{\begin{array}{ll}\theta(i)&\mbox{if }i\in\dom(\theta)\\
g(i)&\mbox{otherwise}\end{array}\right.\] 
We have ${\cal M}_a,
g'\models\phi(x_{i_0}, \ldots, x_{i_{k-1}})$. But 
$g'(0)=\theta(0)=g(0)$ and similarly $g'(n-1)=g(n-1)$, so $g$ is identical
to $g'$ over $n$ and it differs from $g'$ on only a finite
set.  Since $\phi(x_{i_0}, \ldots, x_{i_{k-1}})^{\D_a}\in\Nr_{n}\D_a$, we get that
${\cal M}_a, g \models \phi(x_{i_0}, \ldots,
x_{i_k})$, so $g\in\phi(x_{i_0}, \ldots, x_{i_{k-1}})^{\D_a}$ (this can be proved by induction on quantifier depth of formulas).  
This
proves that 
$$b(x_0, x_1\ldots x_{n-1})^{\D_a}\subseteq\phi(x_{i_0},\ldots,
x_{i_k})^{\D_a}=\pi_a(x),$$ and so
$$\iota_a(b(x_0, x_1,\ldots x_{n-1})^{\D_a})\leq
\iota_a(\phi(x_{i_0},\ldots, x_{i_{k-1}})^{\D_a})\leq x\in\D_a\setminus\set0.$$
Now every non--zero element 
$x$ of $\Nr_{n}\D_a$ is above a non--zero element of the following form 
$\iota_a(b(x_0, x_1,\ldots, x_{n-1})^{\D_a})$
(some $a, b\in \alpha$) and these are the atoms of $\Nr_{n}\D_a$.  
The map defined  via $b \mapsto (b(x_0, x_1,\dots, x_{n-1})^{\D_a}:a\in \alpha)$ 
is an isomorphism of atom structures, 
so that $\alpha=\At\B\in \At{\sf Nr}_n\CA_{\omega}$.
Because we are working in $L_{\infty, \omega},$ infinite disjuncts exist in $\D_a$ $(a\in \alpha)$,
hence, they exist too in the dilation $\D=\bold P_{a\in\alpha}\D_a$.
Therefore $\D$ is complete, so $\Nr_n\D$ is complete, too.
Indeed, let  $X\subseteq \Nr_n\D$. Then by completeness of $\D$, we get that
$d=\sum^{\D}X$ exists.  Assume that  $i\notin n$, then
${\sf c}_id={\sf c}_i\sum X=\sum_{x\in X}{\sf c}_ix=\sum X=d,$
because the ${\sf c}_i$s are completely additive and ${\sf c}_ix=x,$
for all $i\notin n$, since $x\in \Nr_n\D$.
We conclude that $d\in \Nr_n\D$,
and so $\Nr_n\D$ is complete as claimed. 
Now $\D={\bf P}_{a\in \At\B}\D_a$  and 
its $n$--neat reduct $\Nr_n\D$ are complete. 
Accordingly, we can make the identification  
$\Nr_n\D\subseteq_d \Cm\At\B$.  By density,
we get that  $\Nr_n\D=\Cm\At\B$ (since $\Nr_n\B$ is complete),  
hence $\Cm\At\B\in {\sf Nr}_n\CA_{\omega}$.  

Using only $\Cm\At\B\in {\sf Nr}_n\CA_{\omega}$, 
we get  that $\B\in \bold S_d{\sf Nr}_n\CA_{\omega}$, 
because $\B$ is dense in its \de\ completion. 
Hence we attain the second possibility.
But it will now readily follows that any class $\bold K$, such that $\bold S_d{\sf Nr}_n\CA_{\omega}\cap {\sf CRCA_n}\subseteq \bold K\subseteq \bold S_c{\sf Nr}_n\CA_{n+3}$
is not elementary, where $\bold S_d$ denotes the operation of forming dense subalgebras.
Indeed, we have $\B\subseteq_d \Cm\At\B\in \bold {\sf Nr}_n\CA_{\omega}\cap \sf CRCA_n\subseteq \bold K$, 
$\C\notin \bold S_c{\sf Nr}_n\CA_{n+3}\supseteq \bold K$, and $\C\equiv \B$.

Non--first order definability of $\bold N_k$ follows from $\C\equiv \B$, 
$\Cm\At\B\in {\sf Nr}_n{\sf CA}_{\omega}$ and $\C\notin {\sf Nr}_n\CA_{\omega}(\supseteq \bold S_c{\sf Nr}_n\CA_{n+3}$). 
For the last part, It suffices to consider classes between ${\sf Nr}_n\CA_{\omega}$ and $\bold S_d{\sf Nr}_n\CA_{\omega}$. 
One implication, namely $\Longleftarrow$  is trivial. For the other less trivial implication, assume for contradiction that there is such a class $\bold K$ that is elementary.
Then ${\bf El}{\sf Nr}_n\CA_{\omega}\subseteq \bold K$, because $\bold K$ is elementary.
It readily follows that  ${\sf Nr}_n\CA_{\omega}\subseteq {\bf El}{\sf Nr}_n\CA_{\omega}\subseteq \bold K\subseteq \bold S_d{\sf Nr}_n{\sf CA}_{\omega}$,
which is impossible by the 
given assumption that ${\bf El}{\sf Nr}_n\CA_{\omega}\subsetneq \bold S_d{\sf Nr}_n{\sf CA}_{\omega}$.

(3)  Fix an infinite cardinal $\kappa\geq |\alpha|$.  Assume that $\alpha>1$.  Let $\F$ be field of characteristic $0$ such that $|\F|=\kappa$,
$V=\{s\in {}^{\alpha}\F: |\{i\in \alpha: s_i\neq 0\}|<\omega\}$ and let
${\A}$ have universe $\wp(V)$ with the usual concrete operations.
Then clearly $\wp(V)\in {\sf Nr}_{\alpha}\sf CA_{\alpha+\omega}$.
Let $y$ denote the following $\alpha$--ary relation:
$y=\{s\in V: s_0+1=\sum_{i>0} s_i\}.$
Let $y_s$ be the singleton containing $s$, i.e. $y_s=\{s\}.$

Let ${\B}=\Sg^{\A}\{y,y_s:s\in y\}.$ Clearly $|\B|=\kappa$.
Now $\B$ and $\A$ having same top element $V$, share the same atom structure, namely, the singletons, so $\B\subseteq_ d \A$ 
and $\Cm\At\B=\A$. As  proved in \cite{SL}, we have
$\B\notin {\bf  El}{\sf Nr}_{\alpha}{\sf CA}_{\alpha+1}$, hence 
$\B\in \bold S_d\sf Nr_\alpha\CA_{\beta}\sim \sf Nr_\alpha\CA_\beta$.

(4)  The class ${\sf CRCA}_n$ is not elementary by the proof of the first item, cf. \cite{HH},  hence it is not pseudo--univeral. It is also not 
closed under $\bold S$: Take any representable algebra that is not completely representable,  
for example an infinite algebra that is not atomic. Other atomic examples is the term algebra $\Tm \bf At$ dealt with in the proof of theorem \ref{can} 
and $\CA_{\Z, \N}$ dealt with above. The former  is not completely representable because a complete representation of $\Tm\bf At$ induces a representation of $\Cm\bf At$ 
which we know is outside 
$\bold S{\sf Nr}_n\CA_{n+3}$. 
Call such an elgebra $\A$. 
Then $\A^+$  is completely representable, a classical result of Monk's \cite{HH} and $\A$ embeds into $\A^+$.
For pseudo--elementarity one proceeds like the relation algebra case \cite[pp. 279--280]{HHbook} 
defining complete representability 
in a two--sorted theory, undergoing the obvious modifications.
For pseudo--elementarity  for the class ${\sf Nr}_n\CA_{\beta}$ for any $2<n<\beta$  one easily adapts \cite[Theorem 21]{r} by defining  ${\sf Nr}_n\CA_\beta$ 
in a two--sorted theory, when $1<n<\beta<\omega$, and a three--sorted one, when
$\beta=\omega$. The first part is easy.  For the second part; one uses a sort for a $\CA_n$
$(c)$, the second sort is for the Boolean reduct of a $\CA_n$ $(b)$
and the third sort for a set of dimensions $(\delta).$

For any infinite ordinal $\mu$, the defining theory for ${\sf Nr}_n\CA_{\mu}={\sf Nr}_n{\sf CA}_{\omega}$,
includes sentences requiring that the constants $i^{\delta}$ for $i<\omega$
are distinct and that the last two sorts define
a $\CA_\omega$. There is a function $I^b$ from sort $c$ to sort $b$ and sentences forcing  that $I^b$ is injective and
respects the $\CA_n$ operations. For example, for all $x^c$ and $i<n$,
$I^b({\sf c}_i x^c)= {\sf c}_i^b(I^b(x^c)).$ The last requirement is that $I^b$ maps {\it onto} the set of $n$--dimensional elements. This can  be easily expressed
via (*)
$$\forall y^b((\forall z^{\delta}(z^{\delta}\neq 0^{\delta},\ldots (n-1)^{\delta}\implies  c^b(z^{\delta}, y^b)=y^b))\iff \exists x^c(y^b=I^b(x^c))).$$
In all cases, it is clear that any algebra of the right type is the first sort of a model of this theory.
Conversely, a model for this theory will consist of  $\A\in \CA_n$  (sort $c$),
and a $\B\in \CA_{\omega}$;  the dimension of the last is the cardinality of
the $\delta$--sorted elements which is $\omega$, such that by (*) $\A=\Nr_n\B$.
Thus this three--sorted theory defines the class of neat reducts;
furthermore, it is clearly recursive. Recursive enumerability follows from \cite[Theorem 9.37]{HHbook}.

For non--elementarity: The algebras $\A$ and $\B$ constructed in \cite[Theorem 5.1.4]{Sayedneat} satisfy that
$\A\in {\sf Nr}_n\CA_{\omega}$, $\B\notin {\sf Nr}_n\CA_{n+1}$ and $\A\equiv \B$. As they stand, $\A$ and $\B$ are not atomic, but they can be modified 
to be so giving the same result, by interpreting the uncountably many tenary relations in the signature of 
$\M$ defined in \cite[Lemma 5.1.4]{Sayedneat}, which is the base of $\A$ and $\B$ 
to be {\it disjoint} in $\M$, not just distinct. This can be fixed. 
For $u\in {}^nn$, we briefly write $\bold 1_u$ for $\chi_u^{\M}$, denoted by $1_u$ (for $n=3$) in \cite[Theorem 5.1.4]{Sayedneat}. 
We work with $2<n<\omega$ instead of only $n=3$. The proof presented in {\it op.cit} lifts verbatim to any such $n$.
Write $V$ for $^nn$ and recall that $Id:n\to n$ for the identity function on $n$.
For each $w\in V$
the component $\B_w=\{x\in \B: x\leq \bold 1_w\}(\subseteq \A_w=\{x\in \A: x\leq \bold 1_w\}$)
contains infinitely many atoms.

For any $w\in V\setminus \{Id\}$, $\At\B_w=\At\A_w$
and $|\At\A_w|=|\At\B_w|=\kappa$, where $\kappa$ is the (uncountable) cardinality of the $n$--ary relation symbols in the signature.
For $\B$, $|\At\B_{Id}|=\omega$, but it is still an infinite set.
We show that \pe\ has a \ws\ in an \ef-game over $(\A, \B)$ concluding that $\A\equiv_{\infty}\B$.
At any stage of the game,
if \pa\ places a pebble on one of
$\A$ or $\B$, \pe\ must place a matching pebble,  on the other
algebra.  Let $\b a = \la{a_0, a_1, \ldots, a_{n-1}}$ be the position
of the pebbles played so far (by either player) on $\A$ and let $\b
b = \la{b_0, \ldots, b_{n-1}}$ be the position of the pebbles played
on $\B$.  \pe\ maintains the following properties throughout the
game.
\begin{itemize}
\item For any atom $x$ (of either algebra) with
$x\cdot \bold 1_{Id}=0$ then $x \in a_i\iff x\in b_i$.
\item $\b a$ induces a finite partion of $\bold 1_{Id}$ in $\A$ of $2^n$
 (possibly empty) parts $p_i:i<2^n$ and $\b b$ induces a partion of
 $\bold 1_{Id}$ in $\B$ of parts $q_i:i<2^n$.  $p_i$ is finite $\iff$ $q_i$ is
finite and, in this case, $|p_i|=|q_i|$.
\end{itemize}
We have proved that (the stronger) $\A\equiv_{\infty} \B$. 
Though $L_{\infty, \omega}$ does  not see `this cardinality twist' implemented by forcing $\B_{Id}$ to be countable, 
a suitably chosen term will.
Such a term is {\it not term definable in the language of
$\CA_{n}$}. It is the substitutution operator $_{n}{\sf s}(0, 1)$ (using one spare dimension) as defined in the proof of \cite[Theorem 5.1.4]{Sayedneat}.
The term $_{n}{\sf s}(0, 1)$ witnesses that $\B$ is not a neat reduct in the following sense.
Assume for contradiction that 
$\B=\Nr_{n}\C$, with $\C\in \CA_{n+1}.$ Let $u=(1, 0, 2,\ldots n-1)$. Then $\A_u=\B_u$
and so $|\B_u|>\omega$. The term  $_{n}{\sf s}(0, 1)$ acts like a substitution operator corresponding
to the transposition $[0, 1]$; it `swaps' the first two co--ordinates.
Now one can show that $_{n}{\sf s(0,1)}^{\C}\B_u\subseteq \B_{[0,1]\circ u}=\B_{Id},$ 
so $|_{n}{\sf s}(0,1)^{\C}\B_u|$ is countable because $\B_{Id}$ was forced by construction to be 
countable. But $_{n}{\sf s}(0,1)$ is a Boolean automorpism with inverse
$_{n}{\sf s}(1,0)$, so $|\B_u|=|_{n}{\sf s(0,1)}^{\C}\B_u|>\omega$, contradiction.

(5) For the last required fix finite $n>2$.
In \cite[Remark 31]{r} a relation atomic algebra $\R$ having uncountably many atoms 
is constructed such that $\R$ has an $\omega$--dimensional cylindric basis $H$ and $\R$ is not completely representable. If one takes $\C=\Ca(H)$, then $\C\in \CA_{\omega}$,
$\C$ is atomless, and $\R=\Ra\C$. The required $\CA_n$ is $\B=\Nr_n\C$; $\B$ is atomic and has uncountably many atoms. 
Furthermore, $\B$ has no complete representation for a complete representation of $\B$ induces one of $\R$.
We show that $\B$ is in  ${\bf El}\sf CRCA_n$. Since $\B\in {\sf Nr}_n\CA_{\omega}$, 
then  by lemma \ref{n}, \pe\ has a \ws\ in $G_{\omega}(\At\B)$, hence  \pe\ has a \ws\ in $G_k(\At\B)$ for all $k<\omega$.
Using ultrapowers and an elementary chain argument \cite[Corollary 3.3.5]{HHbook2}, we get
that $\B\equiv \D$, for some countable atomic $\D$, and \pe\ has a \ws\ in $G_{\omega}(\At\D)$. Since $\D$ is countable
then by \cite[Theorem 3.3.3]{HHbook2} it is completely representable. We have
proved that 
$\B\in {\bf El}{\sf CRK}_n$. Since $\B\notin {\sf CRCA}_n$, then ${\sf CRCA}_n$ is not elementary. 

For relation algebras we have $\R\in \sf Ra\CA_{\omega}$ and $\R$ has no complete representation. The rest is like the $\CA$ case, 
using \cite[Theorem 33]{r},  when the dilation is $\omega$--dimensional, 
namely, $\R\in \bold S_c\sf Ra\CA_{\omega}\implies$ \pe\ has a \ws\ in $F^{\omega}$  equivalently in $G_{\omega}$ 
(the last two games formulated for 
$\RA$s the former as in \cite[Definition 28]{r}).
\end{proof}

Next we show that a \ws\ in $H_{\omega}$
is `not enough' to deduce (a) as in the second item of the previous proof 
in the sense that, for $2<n<\omega$, if $\D\in \CA_n$ is countable and atomic and \pe\ has a \ws\ in $H_{\omega}(\At\D)$
then $\At\D\in \At{\sf Nr}_n\CA_{\omega}$ but  this {\it does not necessarily imply} 
that $\D\in {\sf Nr}_n\CA_{\omega}$ (as shown in item (3) in the previous theorem).

\begin{theorem}\label{stronger}Let $2<n<\omega$. Let $\B$ be the algebra in item (3) of theorem \ref{rainbow} taking the field 
$\F$ to be $\mathbb{Q}$. Then  \pe\ has a \ws\ in $H_{\omega}(\At\B)$, $\At\B\in \At{\sf Nr}_n\CA_{\omega}$ but 
$\B\notin {\bf El}{\sf Nr}_n\CA_{n+1}\supsetneq {\sf Nr}_n\CA_{\omega}$.
\end{theorem}
\begin{proof} Fix $2<n<\omega$. 
As in the proof of the referred to theorem, let $y=\{s\in {}^n\mathbb{Q}: s_0+1=\sum_{i>0} s_i\}$, $y_s$ be the singleton containing $s$, i.e. $y_s=\{s\}$
and ${\B}=\Sg^{\A}\{y,y_s:s\in y\}$
where ${\A}=\wp(^n\mathbb{Q})$.   Then as shown in {\it op.cit}, $\B\notin {\sf Nr}_n\CA_{n+1}$, $\At\B\in {\sf Nr}_n\CA_{\omega}$, because 
$\At\B=\{\{s\}: s\in {}^n\Q\}=\At\A$, and $\A\in {\sf Nr}_n\CA_{\omega}$.

We refer the reader the second item of theorem \ref{rainbow} for
the notions of {\it long} and {\it short} hyperedges. 
Now we describe the \ws\ of \pe\ in $H_{\omega}(\At\B)$.
We start by describing \pe's strategy dealing with $\lambda$--neat hypernetworks, where $\lambda$ is a constant label kept on short hyperedges.
In a play, \pe\ is required to play $\lambda$--neat hypernetworks, so she has no choice about the
the short edges, these are labelled by $\lambda$. In response to a cylindrifier move by \pa\
extending the current hypernetwork providing a new node $k$,
and a previously played coloured hypernetwork $M$
all long hyperedges not incident with $k$ necessarily keep the hyperlabel they had in $M$.
All long hyperedges incident with $k$ in $M$
are given unique hyperlabels not occurring as the hyperlabel of any other hyperedge in $M$.
In response to an amalgamation move, which involves two hypernetworks required to be amalgamated, say $(M,N)$
all long hyperedges whose range is contained in $\nodes(M)$
have hyperlabel determined by $M$, and those whose range is contained in $\nodes(N)$ have hyperlabels determined
by $N$. If $\bar{x}$ is a long hyperedge of \pe\ s response $L$ where
$\rng(\bar{x})\nsubseteq \nodes(M)$, $\nodes(N)$ then $\bar{x}$
is given
a new hyperlabel, not used in 
any previously played hypernetwork and not used within $L$ as the label of any hyperedge other than $\bar{x}$.
This completes her strategy for labelling hyperedges.

The \ws\ for  \pe\
is to play $\lambda$--neat hypernetworks $(N^a, N^h)$ with $\nodes(N_a)\subseteq \omega$ such that
$(N^a)^+\neq 0$ (recall that $(N^a)^+$ is as defined in the proof of lemma \ref{n}).
In the initial round, let \pa\ play $a\in \At$.
\pe\ plays a network $N$ with $N^a(0, 1, \ldots n-1)=a$. Then $(N^a)^+=a\neq 0$.
The response to the cylindrifier move is exactly like in the first part of lemma \ref{n} because $\B$ is completely representable
so $\B\in \bold S_c{\sf Nr}_n\CA_{\omega}$ \cite[Theorem 5.3.6]{Sayedneat}.
For transformation moves: if \pa\ plays
$(M, \theta),$ then it is easy to see that we have 
${(M^a\theta)}^+\neq 0$,
so this response is maintained in the next round.
For the amalgamation (new) move, as far as the proof of lemma \ref{n}
is concerned, we need some preparing to do. 
We use the argument in \cite[Lemma 34]{r}. 
For each $J\subseteq \omega$, $|J|=n$ say, let $\Nr_J\D=\{x\in \D: {\sf c}_{l}x=x, \forall l\in \omega\setminus  J\}.$ 
Then it can be shown, using that $\At\B\in \At{\sf Nr}_n\CA_{\omega}$, 
that (*):  for all $y\in \Nr_J\D$,  where $J=\{i_0, i_1, \ldots, i_{n-1}\}$,  
the following holds for $a\in \alpha$:  
${\sf s}_{i_0i_1\ldots i_{n-1}}a\cdot y\neq 0\implies  {\sf s}_{i_0 i_1\ldots i_{n-1}}a \leq y$.

Now we are ready to describle \pe's strategy in response to 
amalgamation moves. For better readability, we write $\bar{i}$ for $\{i_0, i_1, \ldots i_{n-1}\}$, 
if it occurs as a set, and we write ${\sf s}_{\bar{i}}$ short for
${\sf s}_{i_0}{\sf s}_{i_1}\ldots {\sf s}_{i_{n-1}}$. 
Also we only deal with the network part of the game.
Now suppose that  \pa\ plays the amalgamation move $(M,N)$ where $\nodes(M)\cap \nodes(N)=\{\bar{i}\}$,
then $M(\bar{i})=N(\bar{i})$.  Let $\mu=\nodes(M)\setminus \bar{i}$ and $v=\nodes(N)\setminus \bar{i}.$
Then ${\sf c}_{(v)}M^+=M^+$ and
${\sf c}_{(u)}{N^+}={M}^+$.
Hence using (*), we have;
${\sf c}_{(u)}{M}^+={\sf s}_{\bar{i}}M(\bar{i})={\sf s}_{\bar{i}}N(\bar{i})={\sf c}_{(v)}{N}^+$
so ${\sf c}_{(v)}{M}^+={M}^+\leq {\sf c}_{(u)}{M}^+={\sf c}_{(v)}N^+$
and ${M}^+\cdot {N}^+\neq 0.$  So there is $L$ with $\nodes(L)=\nodes(M)\cup \nodes(N)\neq 0$,
and ${L}^+\cdot x\neq 0$, where ${M}^+\cdot {N}^+=x$,
thus ${L}^+\cdot {M}^+\neq 0$
and consequently ${L}\restr {\nodes(M)}={M}\restr {\nodes(M)}$,
hence $M\subseteq L$ and similarly $N\subseteq L$, so that $L$ is the required amalgam.
\end{proof}
From the above proof it is not hard to discern below its surface that if $\D$ is an atomic algebra having countably many atoms 
and \pe\ has a \ws\ in $H_{\omega}(\At\D)$, then $\At\D\in \At{\sf Nr}_n\CA_{\omega}$; in fact a weaker game defined in theorem \ref{finitepa2} 
forces this. For relation algebras there is an entirely analogous situation. 
In \cite{r2} the result alleged in \cite{r} was accordingly  weakened by replacing 
$\sf Ra\CA_{\omega}$ by $\bold S_c\sf Ra\CA_{\omega}$ and $\R\in \sf Ra\CA_{\omega}$ by 
$\At\R\in \At\sf Ra\CA_{\omega}$. 
Here we also know that the end point ${\sf Ra}\CA_{\omega}$ is not elementary \cite{bsl}.
Like the reasoning used in item (2) of theorem \ref{rainbow}, by forming the $\omega$--dilation in $L_{\infty, \omega}$, together with the arguments in \cite{r} we get
the following improvement of the result in \cite{r2}, cf. \cite[Theorem 36]{r}: 
\begin{theorem}\label{ra} Any class $\bold K$ of relation algebras, such that $\bold S_d\sf Ra\CA_{\omega}\cap {\sf CRRA}\subseteq \bold K \subseteq
\bold S_c{\sf Ra}\CA_5$, is not elementary.  Furthermore, if  ${\bf El}\sf Ra\CA_{\omega}\nsubseteq \bold S_d{\sf Ra}\CA_\omega$, then 
any  $\bold K$ such that $\sf Ra\CA_{\omega}\cap {\sf CRRA}\subseteq \bold K \subseteq
\bold S_c{\sf Ra}\CA_5$, $\bold K$  is not elementary.  
\end{theorem}
\section{Notions of representability and neat embeddings}

Let $2<n<\omega$. We know by item (3) of theorem \ref{rainbow} 
that ${\sf Nr}_n\CA_{\omega}\subsetneq \bold S_d{\sf Nr}_n\CA_{\omega}$ (the algebra denoted by $\B$ in {\it op.cit} is in the latter class 
but not in the former one).
But unlike the case with relation algebras, we do not know whether the inclusion $\bold S_d{\sf Nr}_n\CA_{\omega}\subseteq \bold S_c{\sf Nr}_n\CA_{\omega}$ is proper or not. 
Follows is an attempt to show that it is.
\begin{definition} Let $M$ be the base of a representation of $\A\in \CA_n$. Then $M$ is {\it $n$--homogeneous} if for any partial isomorphism $\theta$ having size $n$ or less
and any finite subset $X$ of $M$, there is a partial isomorphism $\psi$ extending $\theta$ with $X$ contained within $\rng(\psi)$.
\end{definition}
\begin{theorem} \label{finitepa2}Assume that $2<n<\omega$ and that $\A\in \RCA_n$ is complete and atomic having 
countably many atoms. If $\A$  has no $n$--homogeneous representation, then $\At\A\notin \At{\sf Nr}_n\CA_{\omega}$ and $\A\notin \bold S_d{\sf Nr}_n\CA_{\omega}$.
\end{theorem}
\begin{demo}{Sketch of proof} Fix $2<n<\omega$. 
We first show that $\alpha=\At\A\notin \At{\sf Nr}_n\CA_{\omega}$.
Assume for contradiction that $\alpha\in \At{\sf Nr}_n\CA_{\omega}$. Let $G_{ca}$ be a game that is  like $H$
having a cylindrifier move and two amalgamation moves, but it is 
played on {\it networks} not hypernetworks. Also, in amalgamation moves \pa's choice is restricted by choosing networks that overlap only on at most $n$ 
nodes. Having at hand the assumption that $\At\A\in \At{\sf Nr}_n\CA_{\omega}$,  it can be proved that 
\pe\ has a  \ws\  in $G_{ca}(\alpha)$. 
Next one builds a sequence
of networks $N_0\subseteq \ldots N_r\subseteq \omega$, 
such  that $N_0$ is \pe's response  to \pa's move choosing $a$ in the initial round.
By construction this sequence of networks satisfies: 

(a) if $N_r(\bar{x})\leq {\sf c}_ib$ for $\bar{x}\in \nodes(N_r)$, then there
exists $N_s\supseteq N_r$ and a node $k\in \omega\setminus N_r$ such that $N_s(\bar{y})=b$; where
$\bar{y}\equiv_i \bar{x}$ and $\bar{y}_i=k$,

(b) if $\bar{x}, \bar{y}\in \nodes(N_r)$ such that $N_r(\bar{x})=N_r(\bar{y})$,
then there is a finite surjective map $\theta$ extending $\{(x_i, y_i): i<n\}$ mapping onto $\nodes(N_r)$
such that $\dom(\theta)\cap \nodes(N_r)=\bar{y}$,

(c) if $N_r$ is in the sequence and $\theta$ is any partial
isomorphism of $N_r$, then there is $s\geq r$ and a
partial isomorphism $\theta^+$ of $N_s$ extending $\theta$ such that
$\rng(\theta^+)\supseteq\nodes(N_r)$.

Let $N_a$ be the limit of such networks (defined like in the proof of the second item of theorem \ref{rainbow}).
Define a representation $\cal N$ of $\A$ having
domain  $\bigcup_{a\in A}\nodes(N_a)$, by
$S^{\cal N}=\{\bar{x}: \exists a\in A, \exists s\in S, N_a(\bar{x})=s\},$
for any subset $S$ of $\alpha$. Then this can be checked to be by construction (using (a) and (b) and (c)) to be 
a  complete $n$--homogeneous  representation of $\A$, contradiction, so $\alpha\notin \At{\sf Nr}_n\CA_{\omega}$. 
{\it A fortiori} $\A\notin {\sf Nr}_n\CA_{\omega}$, but $\A$ is complete and atomic so 
$\A\notin \bold S_d{\sf Nr}_n\CA_{\omega}$. 
\end{demo}
\begin{corollary}\label{finitepa3} Let $\R$ be an integral finite non--permutational $\sf RRA$ known to exist \cite{r}.  
Let $2<n<\omega$. If $\C_n(\R)(\in \RCA_n$) as constructed in \cite{AU} 
has no $n$--homogeneous representation, then $\bold S_d{\sf Nr}_n\CA_{\omega}\subsetneq \bold S_c{\sf Nr}_n\CA_{\omega}$ 
and  non of the two classes  $\sf CRCA_n$ and $\bold S_d{\sf Nr}_n\CA_{\omega}\cap \bf At$ is contained in the other. In particular,  the last two statements are true 
for $n=3$.
\end{corollary}
\begin{proof} From the immediately preceding (sketch of) proof, which shows that under the given hypothesis  
$\C_n(\R)\in {\sf CRCA}_n\subseteq \bold S_c{\sf Nr}_n\CA_{\omega}$  and $\C_n(\R)\notin \bold S_d{\sf Nr}_n\CA_{\omega}$. 
 Also $\bold S_d{\sf Nr}_n\CA_{\omega}\cap {\bf At}\supseteq {\sf Nr}_n\CA_{\omega}\cap {\bf At}\nsubseteq \sf CRCA_n$ (by the proof of 
the last item of theorem \ref{rainbow}). 
The last part follow from the fact that if $\R$ is finite and representable having no homogeneous representation,  
then $\C_3(\R)\in \RCA_3$ is finite 
and it also has  no $3$--homogeneous representation. 
\end{proof}
Fix  finite $n>2$.  The chapter \cite{HHbook2} is devoted to studying the following inclusions between various types of atom structures:
$${\sf CRAS}_n\subseteq {\sf LCAS}_n\subseteq {\sf SRAS_n}\subseteq {\sf WRAS_n}.$$
The first is the class of {\it completely representable} atom structures, the second is the class of atom structures satisfying the {\it Lyndon conditions}, the 
third is the class of {\it strongly representable} atom structures, and the last is the class of {\it weakly representable} atom structures, 
all of dimension $n$. 
It  is shown in \cite{HHbook2} that  all inclusions are proper.

Now one can lift such notions from working on {\it atom structures (the frame level)} to working 
on the {\it (complex) algebra level} restricting his attension to atomic ones. and study
them in connection to neat embedding properties, baring in mind
that Henkin's neat embedding theorem characterizes the class of {\it all} representable algebras and that ${\sf CRCA}_n$ and 
$\bold S_c{\sf Nr}_n\CA_{\omega}$ coincide on atomic algebras 
with countably many atoms, theorem \ref{ch}.

We denote the (elementary) class of $\CA_n$s satsfying the Lyndon conditions by ${\sf LCA}_n$, the (non--elementary) class 
of strongly representable $\CA_n$s by ${\sf SRCA}_n$; $\A\in \CA_n$ is {\it strongly representable} $\iff$ $\A$ is atomic and $\Cm\At\A$ is representable. 
Finally, the (elementary) class of weakly representable $\CA_n$s
by $\sf WRCA_n$, which is just the class $\RCA_n\cap \bf At$. All such classes, by definition, consist of atomic algebras.
In the following theorem ${\bf Up}$ denotes the operation of forming ultraproducts,
and ${\bf Ur}$  denotes the operation of forming {\it ultraroots}.

\begin{theorem}\label{main2} Let $2<n<\omega$. Then the following inclusions hold:
$$\ \ {\sf Nr}_n\CA_{\omega}\cap {\bf At}\subsetneq {\bf El}{\sf Nr}_n\CA_{\omega}\cap {\bf At}
\subsetneq {\bf El}\bold S_d{\sf Nr}_n\CA_{\omega}\cap {\bf At}\subseteq {\bf El}\bold S_c{\sf Nr}_n\CA_{\omega}\cap {\bf At}={\sf LCA}_n$$
$$={\bf El}{\sf CRCA}_n \subsetneq {\sf SRCA_n} \subsetneq
{\bf Up}{\sf SRCA_n}={\bf Ur}{\sf  SRCA_n}={\bf El}{\sf SRCA}_n$$
$$\subset {\bf S}{\sf Nr}_n\CA_{\omega}\cap {\bf At}={\sf WRCA}_n.$$
Furthermore, ${\bf El}\bold L$ for any $\bold L$ of the above classes  is not finitely axiomatizable.
\end{theorem}

\begin{proof}
It is known \cite[Proposition 2.90]{HHbook} that
${\bf Up}{\sf SRCA_n}={\bf Ur}{\sf SRCA_n}={\bf El}{\sf SRCA}_n.$
The strictness of the first  inclusion follows from the fourth item of theorem \ref{rainbow}
and the strictness of the second inclusion follows from the third item 
of theorem \ref{rainbow},  respectively.
We have $\sf LCA_n\subsetneq\sf SRCA_n$, because the first is elementary by definition, the second is not \cite{strongca, HHbook2}. 

For (the remaining) equalities, we show that ${\bf El}(\bold S_c{\sf Nr}_n\CA_{\omega}\cap {\bf At})={\sf LCA}_n ={\bf El}{\sf CRCA}_n$.
Plainly ${\sf CRCA}_n\subseteq {\sf LCA}_n$ \cite{HHbook2}. Conversely, let $\A\in {\sf LCA}_n$. 
We proceed like in the proof of the last item of theorem \ref{rainbow}. We have \pe\ has a \ws\ in $G_k(\At\A)$ for all $k\in \omega$. Using ultrapowers and an elementary chain argument,
we get that $\A\equiv \B$, where $\B$ is countable and \pe\ has a \ws\ in $G_{\omega}(\At\B)$. Hence $\B$ is completely representable 
so $\A\in {\bf El}{\sf CRCA}_n$. Now if $\A\in \bold S_c{\sf Nr}_n\CA_{\omega}\cap {\bf At}$, then by lemma \ref{n}, 
\pe\ has a \ws\ in $F^{\omega}(\At\A)$, hence in $G_{\omega}(\At\A)$, {\it a fortiori}, in  $G_k$ for all $k<\omega$, so $\A$ satsfies the Lyndon conditions.
Since $\sf LCA_n$ is elementary,  we get that ${\bf ElS}_c{\sf Nr}_n\CA_{\omega}\cap {\bf At}\subseteq {\sf LCA}_n$.
But ${\sf CRCA}_n\subseteq \bold S_c{\sf Nr}_n\CA_{\omega}\cap {\bf At}$ by theorem \ref{ch}, hence 
${\sf LCA}_n={\bf El}{\sf CRCA}_n\subseteq {\bf El}(\bold S_c{\sf Nr}_n\CA_{\omega}\cap {\bf At})$, proving the remaining equality
and we are done.

For the second part on non--finite axiomatizability: In \cite[Construction 3.2.76, pp.94]{HMT2}
the non--represenatble Monk algebras are finite,  hence they atomic and
are outside ${\sf RCA}_n\supseteq {\bf El}{\sf SRCA_n}\supseteq {\sf LCA}_n$.  Furthermore, any non--trivial ultraproduct of such algebras
is also atomic and is in  
${\sf Nr}_n\CA_{\omega}\subseteq {\bf El}{\sf Nr}_n\CA_{\omega}\cap {\bf At}\subseteq {\bf El}\bold S_c{\sf Nr}_n\CA_{\omega}\cap {\bf At}\subseteq  
{\sf LCA}_n\subseteq {\bf El} {\sf  SRCA}_n.$ (Witness too example \ref{Monk}, and the last paragraph in the paper for two other 
different non--finite axiomatizability proofs).
 \end{proof}
Fix $2<n<\omega$. In the last item of theorem \ref{rainbow}, 
we showed that there is an atomic $\A\in {\sf Nr}_n\CA_{\omega}$ with uncountably many atoms such that $\A$ is not completely representable.
But the $\omega$--dilation $\C$ for which $\A=\Nr_n\C$ is atomless. So can $\C$ be atomic?  
For an ordinal $\alpha$, let $\PEA_\alpha$ stand for the class of polyadic algebras of dimension $\alpha$ \cite[\S 5.4] {HMT1}.
In the next theorem we show that if $\C\in {\sf PEA}_{\omega}\cap \bf At$,
and $\A=\Nr_n\C$, then $\A\in \PEA_n$ is completely representable.  This gives a plethora of completely representable $\PEA_n$s whose $\CA$ reducts are (of course) 
also completely representable. Recall that we write $\A\subseteq_c \B$ 
to denote that $\A$ is a complete subalgebra of $\B$. We use that if $\A\subseteq_c \B$ and $\B$ is atomic, 
then $\A$ is atomic \cite[Lemma 2.16]{HHbook}.

\begin{theorem}\label{pa} If $2<n<\omega$ and $\D\in \PEA_{\omega}$ is atomic, then
any complete subalgebra of $\Nr_n\D$ is completely representable.
\end{theorem}\begin{proof}
We often identify notationally set algebras with their domain. 
Assume that  $\A\subseteq_c{\Nr}_n\D$, where $\D\in \PEA_{\omega}$ is atomic. 
We want to completely represent $\A$.  It suffics to show that for any non--zero $a\in \A$, there is a homomorphism $f:\A\to \wp(^nU)$ 
such that $f(a)\neq 0$, and $\bigcup_{y\in Y}f(y)={}^nU$, whenever $Y\subseteq \A$ satisfies $\sum^{\A}Y=1$.
Assume for the moment  (to be proved in a while) that $\A\subseteq_c \D$. Then $\A$ is atomic, because $\D$ is. For brevity, let $X=\At\A$. 
Let $\mathfrak{m}$ be the {\it local degree} of $\D$, $\mathfrak{c}$ its {\it effective cardinality} 
and let $\mathfrak{n}$ be any cardinal such that $\mathfrak{n}\geq \mathfrak{c}$
and $\sum_{s<\mathfrak{m}}\mathfrak{n}^s=\mathfrak{n}$; such notions are defined in \cite{au}.

Assume that $\D=\Nr_{\omega}\B$, with $\B\in \PEA_{\mathfrak{n}}$ \cite[Theorem 5.4.17]{HMT2}.  
For $\tau\in {}^{\omega}\mathfrak{n}$, we write $\tau^+$ for $\tau\cup Id_{\mathfrak{n}\setminus \omega}(\in {}^\mathfrak{n}\mathfrak{n}$).
Consider the following family of joins evaluated in $\B$,
where $p\in \D$, $\Gamma\subseteq \mathfrak{n}$ and
$\tau\in {}^{\omega}\mathfrak{n}$:
(*) $ {\sf c}_{(\Gamma)}p=\sum^{\B}\{{\sf s}_{{\tau^+}}p: \tau\in {}^{\omega}\mathfrak{n},\ \  \tau\upharpoonright \omega\setminus\Gamma=Id\},$ and (**):
${\sf s}_{{\tau^+}}^{\B}X=1.$
The first family of joins exists \cite{au}, and the second exists, 
because $\sum ^{\A}X=\sum ^{\D}X=\sum ^{\B}X=1$ and $\tau^+$ is completely additive, since
$\B\in \PEA_{\mathfrak{n}}$. 
The last equality of suprema follows from the fact that $\D=\Nr_{\omega}\B\subseteq_c \B$ and the first
from the fact that $\A\subseteq_c \D$. We prove the former, the latter is exactly the same replacing
$\omega$ and $\mathfrak{n}$, by $n$ and $\omega$, respectivey, proving that $\Nr_n\D\subseteq_c \D$, hence $\A\subseteq_c \D$.  

We prove that $\Nr_{\omega}\B\subseteq_c \B$. Assume that $S\subseteq \D$ and $\sum ^{\D}S=1$, and for contradiction, that there exists $d\in \B$ such that
$s\leq d< 1$ for all $s\in S$. Let  $J=\Delta d\setminus \omega$ and take  $t=-{\sf c}_{(J)}(-d)\in {\D}$.
Then  ${\sf c}_{(\mathfrak{n}\setminus \omega)}t={\sf c}_{(\mathfrak{n}\setminus \omega)}(-{\sf c}_{(J)} (-d))
=  {\sf c}_{(\mathfrak{n}\setminus \omega)}-{\sf c}_{(J)} (-d)
=  {\sf c}_{(\mathfrak{n}\setminus \omega)} -{\sf c}_{(\mathfrak{n}\setminus \omega)}{\sf c}_{(J)}( -d)
= -{\sf c}_{(\mathfrak{n}\setminus \omega)}{\sf c}_{(J)}( -d)
=-{\sf c}_{(J)}( -d)
=t.$
We have proved that $t\in \D$.
We now show that $s\leq t<1$ for all $s\in S$, which contradicts $\sum^{\D}S=1$.
If $s\in S$, we show that $s\leq t$. By $s\leq d$, we have  $s\cdot -d=0$.
Hence by ${\sf c}_{(J)}s=s$, we get $0={\sf c}_{(J)}(s\cdot -d)=s\cdot {\sf c}_{(J)}(-d)$, so
$s\leq -{\sf c}_{(J)}(-d)$.  It follows that $s\leq t$ as required. Assume for contradiction that 
$1=-{\sf c}_{(J)}(-d)$. Then ${\sf c}_{(J)}(-d)=0$, so $-d =0$ which contradicts that $d<1$. We have proved that $\sum^{\B}S=1$,
so $\D\subseteq_c \B$.

Let $F$ be any Boolean ultrafilter of $\B$ generated by an atom below $a$. We show that $F$
will preserve the family of joins in (*) and (**).
We use a simple topological argument  used by the author in \cite{au}. 
One forms nowhere dense sets in the Stone space of $\B$ 
corresponding to the aforementioned family of joins 
as follows. The Stone space of (the Boolean reduct of) $\B$ has underlying set,  the set of all Boolean ultrafilters
of $\B$. For $b\in \B$, let $N_b$ be the clopen set $\{F\in S: b\in F\}$.
The required nowhere dense sets are defined for $\Gamma\subseteq \mathfrak{n}$, $p\in \D$ and $\tau\in {}^{\omega}\mathfrak{n}$ via:
$A_{\Gamma,p}=N_{{\sf c}_{(\Gamma)}p}\setminus N_{{\sf s}_{\tau^+}p}$; here we require that $\tau\upharpoonright (\omega\setminus \Gamma)=Id$, 
and $A_{\tau}=S\setminus \bigcup_{x\in X}N_{{\sf s}_{\tau^+}x}.$
The principal ultrafilters are isolated points in the Stone topology, so they lie outside the nowhere dense sets defined above.
Hence any such ultrafilter preserve the joins in (*) and (**). 
Fix a principal ultrafilter $F$ preserving (*) and (**) with $a\in F$. 
For $i, j\in \mathfrak{n}$, set $iEj\iff {\sf d}_{ij}^{\B}\in F$.
Then by the equational properties of diagonal elements and properties of filters, it is easy to show that $E$ is an equivalence relation on $\mathfrak{n}$.
Define $f: \A\to \wp({}^n(\mathfrak{n}/E))$, via $x\mapsto \{\bar{t}\in {}^n(\mathfrak{n}/E): {\sf s}_{t\cup Id}^{\B}x\in F\},$
where $\bar{t}(i/E)=t(i)$ and $t\in {}^n\mathfrak{n}$. 
It is not hard to show that $f$ is well--defined, a homomorphism (from (*)) and atomic (from (**)), such that $f(a)\neq 0$ $({\bar{Id}}\in f(a)$).
\end{proof}

If the dilation is in $\QEA_{\omega}$ (an $\omega$ dimensional quasi--polyadic equality algebra) we have a weaker result. We do not know whether the result proved for $\PEA_{\omega}$ holds
when the $\omega$--dilation is an atomic $\QEA_{\omega}$. Let $n<\omega$. Let $\D\in \QEA_{\omega}\cap \bf At$. 
Assume that  for all $x\in \D$ for all $k<\beta$, 
${\sf c}_kx=\sum_{l\in \beta}{\sf s}_l^kx$. (Such joins exist for example if $\D$ is {\it dimension complemented} in the sense 
that $\omega\setminus \Delta x$ is infinite
for all $x\in D$, where $\Delta x=\{i\in \omega: \c_ix\neq x\}$.)   
If $\A\subseteq \Nr_{n}\D$  such that $\A\subseteq_c \D$ (this is stronger than $\A\subseteq_c \Nr_n\D$), 
then $\A$ is completely representable. To see why, first observe that $\A$ is atomic, because $\D$ is atomic and $\A\subseteq_c \D$.
Accordingly, let $X=\At\A$. Let $a\in \A$ be non-zero. 
As  before, one finds a principal ultrafilter $F$ such that $a\in F$ and $F$ preserves the family of joins ${\sf c}_{i}x=\sum^{\D}_{j\in \beta} {\sf s}_j^i x$, 
and $\sum {\sf s}^{\D}_{\tau}X=1$, where $\tau:\omega\to \omega$ is a finite transformation; that is $|\{i\in \omega: \tau(i)\neq i\}| <\omega$. 
The first family of joins exists by assumption, the second exists, since $\sum^{\D}X=1$ by 
$\A\subseteq_c \D$ and the $\s_{\tau}$s are completely additive.   Any principal ultrafilter $F$ generated by an atom below $a$ will do, as shown in the previous 
proof.  Again as before, the selected $F$ gives the required complete representation of $\A$.

\section{Rainbows versus Monk--like algebras}

Let $2<n<\omega$. Rainbow algebras are similar to Monk--like algebras but {\it only superficially.}
Suppose that $\A\in \CA_n$ is {\it not representable}. The non--representability of $\A$ amounts  to that $\A\notin \bold S{\sf Nr}_n\CA_{n+k}$ for some finite $k$ because
$\RCA_n=\bigcap_{k<\omega}\bold S{\sf Nr}_n\CA_{n+k}$. Can we `pin down' the value of $k$? 
Roughly, the representability of an algebra can be tested by  an $\omega$--rounded game between the two players \pa\ and \pe.
In rainbow constructions the \ws's of the two players are independent, this is  reflected by the fact
that we have  two `independent parameters' $\sf G$ (the greens)  and $\sf R$ (the reds) which are usually surprisingly 
simple relational structures, like finite complete irreflexive 
graphs or ordered structures. (We encountered  these last two cases in theorems \ref{can} and \ref{rainbow}).

In  Monk--like algebras \ws's are interlinked, one operates through the other; hence only one parameter is the source of colours,
typically a certain graph $\G$, witness example \ref{Monk} to follow. Representability of the agebra in
this case depends only on the chromatic number of $\G$, via an application of Ramseys' theorem.
In both cases two players operate using `cardinality of a labelled graph'.
\pa\ trying to make this graph too large for \pe\ to cope, colouring some of
its edges suitably. For the rainbow case, as we have seen in the proofs of theorems \ref{can} and the first and second items of theorem 
\ref{rainbow}, it is a red clique formed during the play.

It might be clear in both cases (rainbow and Monk--like algebras), to see that \pe\ cannot win the infinite game, but what might not
be clear is {\it when does this happens; we know it eventually does but how 
many `pairs of pebbles' on the board, or/ and the number of rounds of the play  do \pa\ need to win?}.
In Monk algebras such numbers are  determined by a large {\it uncontrollable Ramsey number}.
In rainbow constructions, one has more control by varying the green parameter.
The structures $\sf G$ and $\sf R$, having any relative strength
gives flexibility and more control over the rainbow  game lifted
from an \ef\ forth--game on these structures.
The number of nodes  used by \pa\ in the graph game, dictated by the number of pebbles pairs in the aforementioned \ef\ (private) 
game, determines exactly when the algebra in question  
{\bf `stops to be representable'};  it has control over $k$ as specfied above. 

We have seen in this connection that {\it by adjusting the number of greens in the proof of theorem \ref{can} to be $n+1$ 
one gets a finer result than Hodkinson's \cite{Hodkinson} where there were infinitely many greens. 
By truncating  the greens to be $n+1$, we could tell when $\bold S{\sf Nr}_n\CA_{n+k}$, $3\leq k\leq \omega$ {\bf `stops to be atom--canonical'}}.
The dimension $n+k$ (of algebras in $\bold S{\sf Nr}_n\CA_{n+k}$) is controlled by the number of greens $\sf num(g)$ that we start off with.  
One takes ${\sf num({\sf g})}=n+1$, so that $n+3=2+{\sf num(\sf g)}$. The number $2$ is the {\it increase in the number} 
from passing from {\it the number of `pairs of pebbles'} 
used in the private \ef\ forth game ${\sf EF}_{n+1}^{n+1}(n+1, n)$ (as defined on \cite[p. 493]{HHbook}),  
to the {\it number of nodes} used in 
coloured graphs during the play lifted to the rainbow algebra $\D=\CA_{n+1, n}$. 
The last game is the graph version of $F_{r}^{n+3}(\At\D)$ some finite $r\geq n+3$.
We have seen that \pa\  has a \ws\ in both games; the private \ef\ forth `pebble game' using $n+1$ pebble pairs, 
and (hence) the rainbow game, where the number of nodes used is $(n+1)+2$ {\it excluding the existence of an $n+3$--dilation of $\D$}.

In the next example we show that one can prove the weaker result that $\RCA_n$ is not atom--canonical using Monk algebras based on (Monk) relation algebras
and that from such a construction one recovers the non--finite axiomatizability 
results proved by Monk for (representable) relation and cylindric algebras of finite dimension $>2$. But first a definition.  

Let $\R$ be an atomic  relation algebra. Let $2<m<\omega$. An {$m$--dimensional basic matrix}, or simply a matrix  
on $\R$, is a map $f: {}^2m\to \At\R$ satsfying the 
following two consistency 
conditions $f(x, x)\leq \Id$ and $f(x, y)\leq f(x, z); f(z, y)$ for all $x, y, z<m$. For any $f, g$ basic matrices
and $x, y<m$ we write $f\equiv_{xy}g$ if for all $w, z\in m\setminus\set {x, y}$ we have $f(w, z)=g(w, z)$.
We may write $f\equiv_x g$ instead of $f\equiv_{xx}g$.  

\begin{definition}\label{b}
An {\it $m$--dimensional cylindric basis} for an atomic relaton algebra 
$\R$ is a set $\cal M$ of $m$--dimensional matrices on $\R$ with the following properties:
\begin{itemize}
\item If $a, b, c\in \At\R$ and $a\leq b;c$, then there is an $f\in {\cal M}$ with $f(0, 1)=a, f(0, 2)=b$ and $f(2, 1)=c$
\item For all $f,g\in {\cal M}$ and $x,y<m$, with $f\equiv_{xy}g$, there is $h\in {\cal M}$ such that
$f\equiv_xh\equiv_yg$. 
\end{itemize}
\end{definition}
One can construct a $\CA_l$ in a natural way from an $l$--dimensional cylindric basis \cite{HHbook}.
For an atomic  relation algebra $\R$ and $l>3$, we denote by $\sf Mat_l(\At\R)$ the set of all $l$--dimensional basic matrices on $\R$.
${\sf Mat}_l(\At\R)$ is not always an $l$--dimensional cylindric basis.

\begin{example}\label{Monk}{\it Let $\G$ be a graph.
Let $\rho$ be a `shade of red'; we assume that $\rho\notin \G$. Let $L^+$ be the signature consisting of the binary
relation symbols $(a, i)$, for each $a \in \G \cup \{ \rho\}$ and
$i < n$.  Let $T$ denote the following (Monk) theory in this signature:
$M\models T$ iff
for all $a,b\in M$, there is a unique $p\in (\G\cup \{\rho\})\times n$, such that
$(a,b)\in p$ and if  $M\models (a,i)(x,y)\land (b,j)(y,z)\land (c,k)(x,z)$, $x,y, z\in M$, then $| \{ i, j, k \}> 1 $, or
$ a, b, c \in \G$ and $\{ a, b, c\} $ has at least one edge
of $\G$, or exactly one of $a, b, c$ -- say, $a$ -- is $\rho$, and $bc$ is
an edge of $\G$, or two or more of $a, b, c$ are $\rho$.

We denote the class of models of $T$ which can be seen as coloured undirected
graphs (not necessarily complete) with labels coming from
$(\G\cup \{\rho\})\times n$  by $\GG$.
Now specify $\G$ to be either:
\begin{itemize}
\item the graph with nodes $\N$ and edge relation $E$
defined by $(i,j)\in E$ if $0<|i-j|<N$, where $N\geq n(n-1)/2$ is a postive number.

\item or the $\omega$ disjoint union of $N$ cliques, same $N$.
\end{itemize}

In both cases the countably infinite graphs 
contain infinitey many $N$ cliques.  In the first they overlap, in the second they do not.
One shows that there is a countable ($n$--homogeneous) coloured graph  (model) $M\in \GG$,
with the following property \cite[Proposition 2.6]{Hodkinson}:
If $\triangle \subseteq \triangle' \in \GG$, $|\triangle'|
\leq n$, and $\theta : \triangle \rightarrow M$ is an embedding,
then $\theta$ extends to an embedding $\theta' : \triangle'
\rightarrow M$.

Here the choice of $N\geq n(n-)/n$ is not 
haphazard; it bounds the number of edges of any graph $\Delta$ of size $\leq n$.
This is crucial to show that for any permutation $\chi$ of $\omega \cup \{\rho\}$, $\Theta^\chi$
is an $n$-back-and-forth system on $M$ \cite{weak}.  Like in the proof of theorem \ref{can} and its preceding model--theoretic outline, 
the countable atomic   set algebra $\A$ based on $M$ whose top element $W$ is
obtained from $^nM$ by discarding assignments whose edges are labelled by one of $n$--shades of reds ($(\rho, i): i<n)$), 
is (classically) representable.  The classical semantics of  $L_{\omega, \omega}$ formulas 
and relativized semantics (restricting assignments to $W$), coincide, so that 
$\A$ is isomorphic to a set algebra with top element $^nM$. 

So for $l>2$, $l$ finite, let $\A_l$ be the atomic $\sf RCA_n$ 
constructed from $\G_l$, $l\in \omega$ where $\G_l$ has nodes $\N$ and edge relation $E_l$ defined by
$(i,j)\in E_l\iff 0<|i-j|<N_l$, or a disjoint countable union of $N_l$ cliques, such that for $i<j\in \omega$, $n(n-1)/n\leq N_i<N_j.$
Then  $\Cm\At\A_l$ with $\A_l$ based on $\G_l$, as constructed in \cite{weak} is not representable.
So $(\Cm\At\A_l: l\in \omega)$ is a sequence of non--representable algebras,
whose ultraproduct $\B$, being based on the ultraproduct of graphs having arbitrarily large chromatic number, 
will have an infinite clique, and so $\B$ will be completely representable \cite[Theorem 3.6.11]{HHbook2}.
The sequence $(\Tm\At\A_l: l\in \omega)$ is a sequence of representable, 
but not strongly representable, {\it least completely representable} algebras, whose ultraproduct is completey representable.
The same holds for the sequence of 
relation algebras $(\R_l:l\in \omega)$ constructed as in \cite{weak}  for 
which $\Tm\At\A_l\cong \Mat_n\At\R_l$.
Using a standard Los argument, this recovers Monk's classical result \cite{Monk} 
on non--finite axiomatizability of $\sf RRA$s and $\RCA_n$s. 
Also from the second part it follows that the elementary closure of the class of completely representable relation algebras 
and $\CA_n$s, namely, the class of algebras satisfying the Lyndon conditions is not finitely axiomatizable

The relation algebra $\R_l$ in the above sequence is defined as follows. We fix $l$ and we denote $\G_l$ by $\G$ and $\R_l$ by $\R$. Consider the following relation algebra atom structure 
$\alpha(\G)=(\{{\sf Id}\}\cup (\G\times n), R_{\sf Id}, \breve{R}, R_;)$, where:
The only identity atom is $\sf Id$. All atoms are self converse,
so $\breve{R}=\{(a, a): a \text { an atom }\}.$
The colour of an atom $(a,i)\in \G\times n$ is $i$. The identity $\sf Id$ has no colour. A triple $(a,b,c)$
of atoms in $\alpha(\G)$ is consistent if
$R;(a,b,c)$ holds $(R;$ is the accessibility relation corresponding to composition). Then the consistent triples 
are $(a,b,c)$ where:
One of $a,b,c$ is $\sf Id$ and the other two are equal, or
none of $a,b,c$ is $\sf Id$ and they do not all have the same colour, or
$a=(a', i), b=(b', i)$ and $c=(c', i)$ for some $i<n$ and
$a',b',c'\in \G$, and there exists at least one graph edge
of $G$ in $\{a', b', c'\}$.

$\C$ is not representable because $\Cm(\alpha(\G))$ is not representable  and 
${\sf Mat}_n(\alpha(\G))\cong \At\A$, where $\A$ is the atomic $\CA_n$ based on $\G$ defined above. Indeed, for 
$m  \in {\Mat}_n(\alpha(\G)), \,\ \textrm{let} \,\ \alpha_m
= \bigwedge_{i,j<n}  \alpha_{ij}. $ Here $ \alpha_{ij}$ is $x_i =
x_j$ if $ m_{ij} = \Id$ and $R(x_i, x_j)$ otherwise, where $R =
m_{ij} \in L$. Then the map $(m \mapsto
\alpha^W_m)_{m \in {\Mat}_n(\alpha(\G))}$ is a well - defined isomorphism of
$n$-dimensional cylindric algebra atom structures.
Non-representability follows from the fact that  $\G$ is a `bad' graph, that is, 
$\chi(\G)=N<\infty$ \cite[Definition 14.10, Theorem 14.11]{HHbook}.
The relation algebra atom structure specified above is exactly like the one in Definition 14.10 in {\it op.cit}, except that we have $n$ colours 
rather than just three.}
\end{example}
\subsection{Monk algebras and the neat embedding problem}

Sometimes Monk--like algebra are more handy and efficient when games used involve so--called {\it amalgamation moves} 
\cite[The proof of Theorem 15.1]{HHbook}, giving sharper non--finite axiomatizability results. 
The efficiently of the use of Monk--like algebras is witnessed in Hirsch and Hodkinson's 
construction used to solve \cite[Problem 2.12]{HMT1}. In this context, Hirsch and Hodkinson do not require an (uncontrollable) Ramsey number of extra variables (dimensions) needed in proofs,
which was what Monk did in his original proof of the non--finitely axiomatizability result for $\sf RCA_n$ $(2<n<\omega)$ showing that 
$\bold S{\sf Nr}_n\CA_{n+k}\neq {\sf RCA}_n$ for $2<n<\omega$ and $k\in \omega$, 
but they rather {\it require only {\bf one more}  than the  number of colours used}. 
By doing so the proof gives a substantially finer result, namely, that for $2<n<\omega$ and any positive $k\geq 1$, 
$\bold S{\sf Nr}_n\CA_{n+k+1}$ is not even finitely axiomatizable over $\bold S{\sf Nr}_n\CA_{n+k}$. 

In \cite{t}, the famous {\it Neat embedding Problem}, posed as \cite[Problem 2.12]{HMT1} for $\CA$s,
Pinter's substitution algebras ($\Sc$s), polyadic algebras ($\PA$s), quasi--polyadic algebras ($\QA$s),
$\PEA$s, and $\QA$s with equality ($\QEA$s). For all such classes of cylindric--like algebras, the notion of  neat reducts can be defined analogously to the $\CA$ 
case. 
Existing in a somewhat scattered form in the literature, equations defining
$\sf Sc_\alpha, \sf QA_\alpha$ and $\QEA_\alpha$ are given in the appendix of \cite{t} for any ordinal $\alpha$.
It is proved in {\it op. cit} that for any class $\K$ between $\Sc$ and $\QEA$, for any positive $k$,  and for any ordinal $\alpha>2$, 
the class $\bold S{\sf Nr}_\alpha\K_{\alpha+k+1}$ is not axiomatizable by a finite schema over 
$\bold S{\sf Nr}_n\K_{\alpha+k}$.  We strengthen this result when $\alpha\geq \omega$ and when we have diagonal
elements, namely, for any class $\K$ between $\CA_{\alpha}$ and $\QEA_{\alpha}$.
\begin{theorem} \label{2.12a} Let $\alpha$ be any ordinal $>2$ possibly infinite.  Then for any $r\in \omega$, and $k\geq 1$, there exists $\A_r\in \bold S{\sf Nr}_{\alpha}\QEA_{\alpha+k}$
such that $\Rd_{ca}\A_r\notin \bold S{\sf Nr}_{\alpha}\CA_{\alpha+k+1}$
and $\Pi_{r/U}\A_r\in \sf RQEA_{\alpha}$ for any non--principal ultrafilter
$U$ on $\omega$. 
\end{theorem}
\begin{proof} 
The idea used here is the same idea used in \cite[Theorem 3.1]{t}. We use the same notation in {\it op.cit}.
But here the result that we lift from  the finite dimensional case is stronger than
that obtained for finite dimensions in \cite[Theorem 3.1]{t}, hence
we are rewarded by a result stronger than that obtained in \cite{t} for infinite dimensions when restricted to any
$\K$ between $\CA$ and $\QEA$.

Fix $2<m<n<\omega$. Let $\mathfrak{C}(m,n,r)$ be the algebra $\Ca(\bold H)$ where $\bold H=H_m^{n+1}(\A(n,r), \omega)),$
is the $\CA_m$ atom structure consisting of all $n+1$--wide $m$--dimensional
wide $\omega$ hypernetworks \cite[Definition 12.21]{HHbook}
on $\A(n,r)$  as defined in \cite[Definition 15.2]{HHbook}.   Then $\mathfrak{C}(m, n, r)\in \CA_m$, and it can be easily expanded
to a $\QEA_m$, since $\C(m, n, r)$ is `symmetric', in the sense  that it allows a polyadic equality 
expansion by defining substitution operations corresponding to transpositions.\footnote{
This follows by observing that $\bold H$ is obviously symmetric in the following exact sense:  
For $\theta:m\to m$ and $N\in \bold H$, $N\theta\in \bold H,$
where $N\theta$ is defined by $(N\theta)(x, y)=N(\theta(x), \theta(y)).$ 
 Hence, the binary polyadic operations defined on the atom structure 
$\bold H$ the obvious way (by swapping co--ordinates) 
lifts to polyadic operations of its complex algebra $\mathfrak{C}(m, n, r)$. In more detail, for a transposition 
$\tau:m\to m$,  and $X\subseteq \bold H$,  define $\s_\tau(X)=\set{N\in \bold H: N\tau \in X}$.}

Furthermore, for any $r\in \omega$ and $3\leq m\leq n<\omega$, $\C(m,n,r)\in {\sf Nr}_m{\sf QEA}_n$, $\Rd_{ca}\C(m,n,r)\notin {\bold  S}{\sf Nr}_m{\sf CA_{n+1}}$
and $\Pi_{r/U}\C(m,n,r)\in {\sf RQEA}_m$ by easily
adapting \cite[Corollaries 15.7, 5.10, Exercise 2, pp. 484, Remark 15.13]{HHbook}
to the $\QEA$ context.

Take
$$x_n=\{f\in H_n^{n+k+1}(\A(n,r), \omega); m\leq j<n\to \exists i<m, f(i,j)=\Id\}.$$
Then $x_n\in \C(n,n+k,r)$ and ${\sf c}_ix_n\cdot {\sf c}_jx_n=x_n$ for distinct $i, j<m$.
Furthermore (*),
$I_n:\C(m,m+k,r)\cong \Rl_{x_n}\Rd_m {\C}(n,n+k, r)$
via the map, defined for $S\subseteq H_m^{m+k+1}(\A(m+k,r), \omega)),$ by
$$I_n(S)=\{f\in H_n^{n+k+1}(\A(n,r), \omega):  f\upharpoonright {}^{\leq m+k+1}m\in S,$$
$$\forall j(m\leq j<n\to  \exists i<m,  f(i,j)=\Id)\}.$$
We have proved the (known) result for finite ordinals $>2$.

To lift the result to the transfinite,
we proceed like in \cite{t}, using a lifting argument due to Monk.
Let $\alpha$ be an infinite ordinal. Let $I=\{\Gamma: \Gamma\subseteq \alpha,  |\Gamma|<\omega\}$.
For each $\Gamma\in I$, let $M_{\Gamma}=\{\Delta\in I: \Gamma\subseteq \Delta\}$,
and let $F$ be an ultrafilter on $I$ such that $\forall\Gamma\in I,\; M_{\Gamma}\in F$.
For each $\Gamma\in I$, let $\rho_{\Gamma}$
be an injective function from $|\Gamma|$ onto $\Gamma.$
Let ${\C}_{\Gamma}^r$ be an algebra similar to $\QEA_{\alpha}$ such that
$\Rd^{\rho_\Gamma}{\C}_{\Gamma}^r={\C}(|\Gamma|, |\Gamma|+k,r)$
and let
$\B^r=\Pi_{\Gamma/F\in I}\C_{\Gamma}^r.$
Then we have $\B^r\in \bold {\sf Nr}_\alpha\QEA_{\alpha+k}$ and
$\Rd_{ca}\B^r\not\in \bold S{\sf Nr}_\alpha\CA_{\alpha+k+1}$.
These can be proved exactly like the proof of the first two items in \cite[Theorem 3.1]{t}. 
We know
from the finite dimensional case that $\Pi_{r/U}\Rd^{\rho_\Gamma}\C^r_\Gamma=\Pi_{r/U}\C(|\Gamma|, |\Gamma|+k, r) \subseteq \Nr_{|\Gamma|}\A_\Gamma$,
for some $\A_\Gamma\in\QEA_{|\Gamma|+\omega}=\QEA_{\omega}$.
Let $\lambda_\Gamma:\omega\rightarrow\alpha+\omega$
extend $\rho_\Gamma:|\Gamma|\rightarrow \Gamma \; (\subseteq\alpha)$ and satisfy
$\lambda_\Gamma(|\Gamma|+i)=\alpha+i$
for $i<\omega$.  Let $\F_\Gamma$ be a $\QEA_{\alpha+\omega}$ type algebra such that $\Rd^{\lambda_\Gamma}\F_\Gamma=\A_\Gamma$.
Then $\Pi_{\Gamma/F}\F_\Gamma\in\QEA_{\alpha+\omega}$, and we have proceeding like in the proof of item 3 in \cite[Theorem 3.1]{t}:

$\Pi_{r/U}\B^r=\Pi_{r/U}\Pi_{\Gamma/F}\C^r_\Gamma
\cong \Pi_{\Gamma/F}\Pi_{r/U}\C^r_\Gamma
\subseteq \Pi_{\Gamma/F}\Nr_{|\Gamma|}\A_\Gamma
=\Pi_{\Gamma/F}\Nr_{|\Gamma|}\Rd^{\lambda_\Gamma}\F_\Gamma
=\Nr_\alpha\Pi_{\Gamma/F}\F_\Gamma$.

But $\B=\Pi_{r/U}\B^r\in \bold S{\sf Nr}_{\alpha}\QEA_{\alpha+\omega}$
because $\F=\Pi_{\Gamma/F}\F_{\Gamma}\in \QEA_{\alpha+\omega}$
and $\B\subseteq \Nr_{\alpha}\F$, hence it is representable (here we use the neat embedding theorem).
Now it can be easily shown that for any
$\K$ between $\CA$ and $\QEA$, and positive $k$,
$\bold S{\sf Nr}_{\alpha}\K_{\alpha+k+ l}$ is not axiomatizable
by a finite schema over  $\bold S{\sf Nr}_{\alpha}\K_{\alpha+k}$ in the sense of \cite[Definition 5.4.12]{HMT2} for any $l\geq 1$.
In \cite[Theorem 3.1]{t}, the ultraproduct was proved to be in
$\bold S{\sf Nr}_{\alpha}\K_{\alpha+k+1}$ for $\K$ between $\Sc$ and $\QEA$,
a strict superset of $\sf RK_{\alpha}$. In fact, the result here is `infinitely stronger'. Using a L\'os argument, we have $\sf RK_{\alpha}$
cannot be axiomatized by a finite schema over
$\bold S{\sf Nr}_{\alpha}\K_{\alpha+m}$
for any finite $m\geq 0$
\end{proof}
In \cite[Theorem 3.1]{t} the following is proved. Let $\alpha >2$.  Then for any $r\in \omega$, for any
finite $k\geq 1$, there exist $\B^{r}\in \bold S\sf Nr_{\alpha}\QEA_{\alpha+k}$, and
$\Rd_{\Sc}\B^r\notin \bold S\sf Nr_{\alpha}\Sc_{\alpha+k+1}$ such
$\Pi_{r/U}\B^r\in \bold S\sf Nr_{\alpha}\QEA_{\alpha+k+1}.$
We do not know whether we can replace $\bold S{\sf Nr}_{\alpha}\QEA_{\alpha+k+1}$ in the conclusion by
$\RQEA_{\alpha}$, like we did in theorem \ref{2.12a} when dealing only with $\CA$s and $\QEA$s. 

Now we review the main result in \cite{t} for finite dimensions. The infinite dimensional case is obtained from
the finite dimensional one using the same  lifting argument 
used in the proof of theorem \ref{2.12a}, cf. \cite[Theorem 3.1]{t}.
The third item in our coming theorem \ref{thm:cmnr},
which is (the main theorem) \cite [Theorem 1.1]{t} is {\it strictly weaker} than the result 
(for finite dimensions) used in proof of theorem \ref{2.12a}, namely  (using the notation {\it op. cit}), 
that $\Pi_{r/U}\C(m, n, r)\in \RQEA_m$ (upon replacing $\C(m, n, r)$ by $\D(m, n, r)$.)
The theorem is due to Robin Hirsch.
\begin{theorem}\label{thm:cmnr} Let $3\leq m\leq n$ and $r<\omega$.
\begin{enumerate}
\item $\D(m, n, r)\in {\sf Nr}_m\sf QEA_n$,\label{en:one}
\item $\Rd_{\Sc}\D(m, n, r)\not\in \bold S{\sf Nr}_m\Sc_{n+1}$, \label{en:two}
\item $\Pi_{r/U} \D(m, n, r)\in {\bf El}{\sf Nr}_m\sf QEA_{n+1}$.  
\end{enumerate}
\end{theorem}
We define the algebras $\D(m,n,r)$ for $3\leq m\leq n<\omega$ and $r<\omega$.
The hardest part is proving \eqref{en:two}. This is given in detail in \cite[p. 211--215]{t}. 
We start with.
\begin{definition}\label{def:cmnr}
Define a function $\kappa:\omega\times\omega\rightarrow\omega$ by $\kappa(x, 0)=0$
(all $x<\omega$) and $\kappa(x, y+1)=1+x\times\kappa(x, y))$ (all $x, y<\omega$).
For $n, r<\omega$ let
\[\psi(n, r)=
\kappa((n-1)r, (n-1)r)+1.\]
This is to ensure that $\psi(n, r)$ is sufficiently big compared to $n, r$ for the proof 
of non-embeddability to work.
The second parameter $r<\omega$ may be considered as a finite linear order of length $r$.
For any  $n<\omega$ and any linear order $r$, let
${\sf Bin}(n, r)=\{{\sf Id}\}\cup\{a^k(i, j):i< n-1,\;j\in r,\;k<\psi(n, r)\}$
where ${\sf Id}, a^k(i, j)$ are distinct objects indexed by $k, i, j$.
Let $3\leq m\leq n<\omega$ and let $r$ be any linear order.
Here ${\sf Bin}(n,r)$ is an atom structure of a finite relation relation $\R$
and ${\sf Forb}$ specifies its operations by the standard procedure of specifying forbidden triples \cite{HHbook}.
The relation algebra $\R$;
is similar (but not identical) to $\A(n,r)$ used in the first part of the 
proof of theorem \ref{2.12a} and $\D(m,n,r)$ is defined to be $\Cm{\sf Mat}_m(\At\R)(\in \QEA_m)$. 
\end{definition}
Unlike the algebras $\C(m,n,r)$ in the proof of theorem \ref{2.12a}, 
the algebras $\D(m, n, r)$ are now finite.
It is not hard to see
that  $3\leq m,\; 2\leq n$ and $r<\omega$
the algebra $\D(m, n, r)$ satisfies all of the axioms defining $\QEA_m$.
Furthermore, if  $3\leq m\leq m'$ then $\C(m, n, r)\cong\Nr_m\C(m', n, r)$
via $X\mapsto \set{f\in F(m', n, r): f\restr{m\times m}\in X}.$

Recall that in the first part of the proof of theorem \ref{2.12a}, we had $\Pi_{r/U}\C(m, n, r)\in \RQEA_m$.
Here we do not guarantee that the ultraproduct on $r$
of $\D(m,n,r)$ ($2<m<n<\omega)$ is representable.
A standard L\"os argument shows that
$\Pi_{r/U}\C(m, n, r) \cong\C(m, n, \Pi_{r/U} r)$ and $\Pi_{r/U}r$
contains an infinite ascending sequence.
(Here  one has to extend the definition of $\psi$
by letting $\psi(n, r)=\omega,$ for any infinite linear order $r$.)
The infinite algebra
$\D(m,n, J)\in {\bf El}{\sf Nr}_n\QEA_{n+1}$
when $J$ is an infinite linear order as above.
Since $\Pi_{r/U} r$ is such, then we get $\Pi_{r/U}\D(m, n, r)\in 
{\bf El}{\sf Nr}_m\QEA_{n+1}(\subseteq \bold S{\sf Nr}_m\QEA_{n+1}$), cf. \cite[pp.216--217]{t}.
This suffices to show that for any $\K$ having signature between $\Sc$ and $\QEA$, for any $2<m<\omega$, and for any positive $k$,
the variety $\bold S{\sf Nr}_m\K_{m+k+1}$ is not finitely axiomatizable over the variety  ${\bold S}{\sf Nr}_m{\sf K}_{m+k}$.

{\bf Remark:} Let ${\sf Cs}_n$ denote the class of cylindric set algebras of dimension $n$. 
In theorem \ref{can} for each $2<n<\omega$,
an atomic, countable $\A_n\in {\sf Cs}_n$ was constructed such that $\B_n=\Cm\At\A_n\notin \bold S{\sf Nr}_n\CA_{n+3}$. 
{\it Using this notation, if for $2<k<m<\omega$, $\B_k\subseteq \Rd_m\B_m,$
then for any ordinal $k\geq 3$,  
the variety $\bold S{\sf Nr}_{\omega}\CA_{\omega+k}$ would not be atom--canonical.}
To see why, for each finite $n\geq 3$,  let $\A_n^+$ be an algebra having the signature of $\CA_{\omega}$
such that $\Rd_n\A_n^+=\A_n$. Analogously, let $\B_n^+$ be an algebra having the signature
of $\CA_{\omega}$ such that $\Rd_n\B_n^+=\B_n$, and we require in addition that $\B_n^+=\Cm(\At\A_n^+)$.
As in the proof of theorem \ref{2.12a}, $\A=\Pi_{i\in \omega}\A_i^+/F\in \RCA_{\omega}$ and $\B=\Pi_{i\in \omega}\B_i^+/F\in \CA_{\omega}$.
Also, $\Cm\At\A =\Cm(\At[\Pi_{i\in \omega}\A_n^+/F])
\cong\Cm[\Pi_{i\in \omega}(\At\A_n^+)/F)]
\cong \Pi_{i\in \omega}(\Cm(\At\A_n^+)/F)
=\Pi_{i\in \omega}\B_n^+/F
=\B.$ We now show that $\B$ is outside $\bold S{\sf Nr}_{\omega}\CA_{\omega+3}$ proving the required.
Assume for contradiction that $\B\in \bold S{\sf Nr}_{\omega}\CA_{\omega+3}$.
Then $\B\subseteq \Nr_{\omega}\C$ for some $\C\in \CA_{\omega+3}$.
Let $3\leq m<\omega$ and  $\lambda:m+3\rightarrow \omega+3$ be the function defined by $\lambda(i)=i$ for $i<m$
and $\lambda(m+i)=\omega+i$ for $i<3$.
Then $\Rd^\lambda\C\in \CA_{m+3}$ and $\Rd_m\B\subseteq \Rd_m\Rd^\lambda\C$.
By assumption, we have $\B_m$ embeds into $\Rd_m\B_{t}$
whenever $3\leq m<t<\omega$, via $I_t$ say.
Let $\iota( b)=(I_{t}b: t\geq m )/F$ for  $b\in \B_m$.
Then $\iota$ is an injective homomorphism that embeds $\B_m$ into
$\Rd_m\B$.  By the above, we have $\Rd_{m}\B\in {\bf S}{\sf Nr}_m\CA_{m+3}$, hence  $\B_m\in \bold S{\sf Nr}_{m}\CA_{m+3}$, too,
which is a contradiction and we are done.

{\bf Rainbows versus splitting:}
Rainbows also offer solace, when splitting techniques as in \cite{Andreka} (that depend essentially on diagonal elements) do not work, to show that for $2<n<\omega$, 
the variety ${\sf RDf}_n$
cannot be axiomatized by a set of universal formulas having finitely many variable. 
We give an outline of the idea which is inspired by the rainbow construction  for relation algebras in \cite[\S 17.3]{HHbook}. 
Fix $2<n<\omega$ and finite $m>1$ and let $\sf K$ be a class whose signature is between that of ${\sf Df}_n$ and ${\sf CA}_n$.  
One can construct two  finite simple rainbow 
algebras $\A,\B\in \CA_n$ that satisfy the following. The $n$--coloured graphs (atoms) 
in both algebras are the same, except for the red atoms ($n$--coloured graphs having at least one edge labelled by a red). 
The rainbow signature of $\A$ has more red colours, thus $\A$ has more red atoms than $\B$, enabling  \pe\ to win the game $G_{\omega}(\At\A)$
implying that $\A\in \RCA_n$.  In $\B$ there are fewer red atoms; the greens outfit the reds, so \pa\ can win $G_{\omega}(\At\B)$
so that $\B\notin \RCA_n$.  Since $\B$ is generated by the set $\{b\in \B: \Delta b\neq n\}$, then by \cite[Theorem 5.4.26]{HMT2}, 
it will follow that 
the diagonal free reduct $\Rd_{df}\B$ is not in ${\sf RDf_n}$. Inspite of the discrepancy in the number of red atoms in $\A$ and $\B$, 
to the extent that $\A$ is representable, while the (diagonal free reduct of) $\B$ is not representable, this discrepancy will not be witnessed by $m$--variable equations. 
For any assignment $s:m\to \A$, for any equation $e$ in the signature of $\K$ using $m$ variables, one construct an assignment $s':m\to \B$ by adjusting the (fewer) red atoms of $\B$ 
below $s'(i)$ for each $i<m$, by putting  in there `enough' red atoms   
in such a way  that  $\Rd_{\K}\B, s'\models e\iff \Rd_{\K}\A, s\models e$.  
This can   be done for each $m>1$. 
Now if $\Sigma$ is any $m$--variable equational theory then the $\K$  reduct of $\A$ and $\B$ either both validate
$\Sigma$ or neither do. Since one algebra is in ${\sf RK}_n$ (representable $\K_n$s) 
while the other is not, it follows that $\Sigma$ does not axiomatize 
${\sf RK}_n$.   But $\Rd_{\K}\A$ and $\Rd_{K}\B$ are simple algebras,
so ${\sf RK}_n$ has no finite variable universal prenex axiomatization because 
in a discriminator variety every universal prenex formula  is equivalent in subdirectly irreducible (hence simple) members to an
equation using the same number of variables.  From this it follows that, for $2<n<\omega$, there 
is no universal axiomatization containing finitely many variables for ${\sf RDf}_n$. 

Let us see how to get the right (finite) number of reds and greens so that the above idea works. Remember that $n$ and $m$ are fixed; $2<n<\omega$ and $1<m<\omega$.  
First one takes $\lambda=(n\times 2^m)^3$. 
Then the greens for both $\A$ and $\B$ will be  ${\sf G}=\lambda+2$. For $\A$, the reds are 
${\sf R}_{\A}=[(\lambda+1)\times (\lambda+2)]/2$
and for $\B$ the reds are ${\sf R}_{\B}=\lambda<[(\lambda+1)(\lambda+2)]/2$. 
Let $\R$ be the red atoms of $\A$ and $\R'$ be the red atoms of $\B$. 
Then $|\R'|\geq n\times 2^m$. Let $e$ an equation using $m$ variables. 
Assume that $s:m\to \A$ falsifies $e.$
We define $s':m\to \B$ that falsifies $e$, too. (The converse is entirely analogous).
We denote a red atom 
by $\r$. Define a partition $(\R_S: S\subseteq m)$ of
$\R$ where $\R_S=\{\r\in \R: \r\leq s(i)\in S,  \r\cdot s(i): i\in m\sim S\}$. (Here $\R_S$ can be empty).
Because $|\R'|\geq n\times 2^m$,  one can define another  partition $\mathfrak{P}=(\R'_S:S\subseteq m)$ of $\R'$
such that $\R_S\subseteq s(i)\iff \R'_S\subseteq s'(i)$ for all $S\subseteq m$ and $i<m$, 
and using $\mathfrak{P}$ it can be arranged that for each $i<m$, the number of red atoms of $\A$ under $s(i)$ 
is the same as the number of atoms of 
$\B$ under $s'(i)$ if this number is $<n$ and both are greater than $n$ otherwise, by putting the red atom $r'$ below $s'(i)$ if $\r'\in \R'_S$ for some
$S$ with $i\in S$. The other atoms are the same for both algebras. Then $s'(i)$ is defined to be the sum of the atoms below it, the non--red ones being the same as the ones below $s(i).$
Then it can be shown $s'$ is as required.  

Unlike `usual rainbow constructions' where reds have double indices,
here reds have single indices, so for example in $\B$ the reds are
$\{\r_i: i\in \lambda\}$.  By definition a triangle of reds $(\r_i, \r_j, \r_k)$ is consistent  $\iff |\{i, j, k\}|=3$. The number of all other colours (uniquely determined by the red and green colours) 
is the same for both algebras.  
In the game played on coloured graphs,  \pa\ can force a red clique of size $\lambda+2$ and not more, by playing $\lambda+2$ cones having the same base and distinct 
green tints. Not to lose, 
\pe\ has to choose a label for each edge between two successive appexes of cones, using
a red colour, and she has to ensure that each edge within the clique has a red label unique to this edge, so that no
triangle (within this clique) has two  reds with same index.
She will not succeed if the number of reds is $\leq \lambda$, there will have to come a 
a round in the play where she will be forced to play an inconsistent triple of reds $(\r, \r, \r')$
in which case \pa\ wins after a finite $l(\geq \lambda+2)$ many rounds.  
He plays as follows: In his zeroth move, \pa\ plays a graph $\Gamma$ with
nodes $0, 1,\ldots, n-1$ and such that $\Gamma(i, j) = \w_0 (i < j <
n-1), \Gamma(i, n-1) = \g_i ( i = 1,\ldots, n-2), \Gamma(0, n-1) =
\g^0_0$, and $ \Gamma(0, 1,\ldots, n-2) = \y_{\lambda+2}$. This is a $0$-cone
with base $\{0,\ldots, n-2\}$. In the following moves, \pa\
repeatedly chooses the face $(0, 1,\ldots, n-2)$ and demands a node
$t$ with $\Phi(i,t) = \g_i$, $(i=1,\ldots, n-2)$ and $\Phi(0, t) = \g^t_0$,
in the graph notation -- i.e., an $t$-cone, $t\leq \lambda+2$,  on the same base.
\pe\ among other things, has to colour all the edges
connecting $\lambda+2$ nodes $n_0, n_1, \ldots n_{\lambda+1}$ created by \pa\ as apexes of cones based on the face $(0,1,\ldots, n-2)$ by red labels. 
But there are only $\lambda$ red labels, so there must be $0<i\neq j<\lambda+2$, such that in the last coloured graph 
$\Lambda$,  $\Lambda(n_0, n_i)=\Lambda(n_0, n_j)$. But
$(n_0, n_i, n_j)$ is inconsistent, so \pa\ wins.
The conclusion now follows.
Thus $\Rd_{df}\B\notin {\sf RDf}_n$.
On the other hand, the number of reds in $\A$ is 
$\geq  [(\lambda+1)\times (\lambda+2)]/2$, so  
\pe\ can always choose a red label avoiding inconsistent red triangles. In this case she can cope with 
such red cliques  `indexed' by the green tints, hence $\A\in \RCA_n$.

\section{Metalogical appplications}

Throughout this section, unless otherwise indicated $n$ is a finite ordinal. 
Results in algebraic logic are most attractive when they lend themselves to (non--trivial) applications in (first order) logic.
In this section, we apply the hitherto obtained results on (non) atom--canonicity to (failure of) omitting types theorems for 
the so--called {\it clique guarded} $n$--variable fragments
of first order logic. We start with defining the notion of {\it clique guarded semantics}.
\begin{definition}
Assume that $1<n<m<\omega$. Let $M$ be a {\it relativized representation} of $\A\in \CA_n$, that is, there exists an injective
homomorphism $f:\A\to \wp(V)$, where $V\subseteq {}^nM$ and $\bigcup_{s\in V} \rng(s)=M$. Here we identify notationally the set algebra with universe $\wp(V)$ with
its universe $\wp(V)$.
We write $M\models a(s)$ for $s\in f(a)$. Let  $\L(\A)^m$ be the first order signature using $m$ variables
and one $n$--ary relation symbol for each element in $A$. Then {\it an $n$--clique} is a set $C\subseteq M$ such
$(a_1,\ldots, a_{n-1})\in V=1^M$
for distinct $a_1, \ldots, a_{n}\in C.$
Let
${\sf C}^m(M)=\{s\in {}^mM :\rng(s) \text { is an $n$--clique}\}.$
Then ${\sf C}^m(M)$ is called the {\it $n$--Gaifman hypergraph of $M$}, with the $n$--hyperedge relation $1^M$.

The {\it clique guarded semantics $\models_c$} are defined inductively. For atomic formulas and Boolean connectives they are defined
like the classical case and for existential quantifiers
(cylindrifiers) they are defined as follows:
for $\bar{s}\in {}^mM$, $i<m$, $M, \bar{s}\models_c \exists x_i\phi$ $\iff$ there is a $\bar{t}\in {\sf C}^m(M)$, $\bar{t}\equiv_i \bar{s}$ such that
$M, \bar{t}\models \phi$.

We say that $M$  is  {\it $m$--square},
if  $\bar{s}\in {\sf C}^m(M), a\in \A$, $i<n$,
and   $l:n\to m$ is an injective map, $M\models {\sf c}_ia(s_{l(0)},\ldots, s_{l(n-1)})$,
$\implies$ there is a $\bar{t}\in {\sf C}^m(M)$ with $\bar{t}\equiv _i \bar{s}$,
and $M\models a(t_{l(0)}, \ldots, t_{l(n-1)})$. $M$ is said to be {\it $m$--flat} if  it is $m$--square and
for all $\phi\in \L(\A)^m$, for all $\bar{s}\in {\sf C}^m(M)$, for all distinct $i,j<m$,
$M\models_c [\exists x_i\exists x_j\phi\longleftrightarrow \exists x_j\exists x_i\phi](\bar{s}).$
\end{definition}

\begin{definition}\label{strongblur}\cite[Definition 3.1]{ANT}
Let $\R$ be a relation algebra, with non--identity atoms $I$ and $2<n<\omega$. Assume that  
$J\subseteq \wp(I)$ and $E\subseteq {}^3\omega$.
We say that $(J, E)$  is an {\it $n$--blur} for $\R$, if $J$ is a {\it complex $n$--blur}
and the tenary relation $E$ is an {\it index blur} defined  as 
in item (ii) of \cite[Definition 3.1]{ANT}.
We say that $(J, E)$ is a {\it strong $n$--blur}, if it $(J, E)$ is an $n$--blur,  such that the complex 
$n$--blur  satisfies the condition $(J5)_n$ on \cite[pp.79]{ANT}.
\end{definition}
We need the following lemma proved in \cite{mlq}.
\begin{lemma}\label{square}
Let $2<n<m<\omega$ and $\A\in \CA_n$.
Then $\A$ has an $m$--flat representation $\iff$ $\A\in \bold S{\sf Nr}_n\CA_m$ and $\A$ has a complete  $m$--square representation $\iff$  
\pe\ has a \ws\ in $G_{\omega}^m(\At\A).$ 
\end{lemma}
Now fix $2<n<\omega$. Let $n\leq l<m\leq \omega$. Consider results of the following form, $\Psi(l, m) (\Psi(l, m)_f)$ for short: 
{\it There is an atomic, countable and complete $L_n$ theory $T$, such that the type $\Gamma$ consisting of co--atoms is realizable 
in every $m$-- square (flat) model, but any formula isolating this type has
to contain more than $l$ variables.}

By an $m$--flat model of $T$ we understand an $m$--flat representation of the Tarski--Lindenbaum quotent algebra $\Fm_T$.
By $T$ atomic, we mean 
that the Boolean reduct of $\Fm_T$ is atomic. The type $\Gamma$ is defined to be 
$\{\phi: (\neg \phi)_T\in \At\Fm_T\}$. The theory $T$ is complete means that for any $L_n$ sentence $\phi$ in the signature of $T$, 
$T\models \phi$ or $T\models \neg \phi$.
This, in turn,  is equivalent to that $\Fm_T(\in \RCA_n$) 
is simple (has no proper congruences), hence $\Fm_T\in {\sf Cs}_n$. 
\begin{theorem}\label{2.12} Fix $2<n\leq l<m< \omega$. Then the following hold:
\begin{enumarab}
\item $\Psi(l, \omega)$ and $\Psi(n, n+k)$ are true for $k\geq 3$.
\item If there exists a finite relation algebra 
$\R_m$ that has  a strong $m$--blur, but does not have an 
infinite $m+1$--dimensional hyperbases, 
then $\Psi(l, m+1)_f$ is true. If $\R_m$ has no $m+1$--dimensional {\it relational basis}, then $\Psi(l, m+1)$ is true.
\end{enumarab}
\end{theorem}
\begin{proof}
$\Psi(l, \omega)$ follows easily from the construction in \cite{ANT}, cf. \cite[Theorem 3.1.1]{Sayed}. 
We prove $\Psi(n, n+3)$.  Let $\A$ be the infinitely countable, atomic and simple algebra  
obtained by blowing up and blurring $\CA_{n+1, n}$ as in theorem \ref{can}. 
We can identify $\A$ with $\Fm_{T}$ for some countable, consistent
and complete theory $T$ (since $\A$ is simple) 
using $n$ variables, and because $\A$ is atomic, $T$ is an atomic theory, as well.
Let $\Gamma=\{\neg\phi: \phi_T\in \At\Fm_{m,T}\}$ be the type consisting of co--atoms.
Then $\Gamma$ is
a non--principal type. But  $\Gamma$ cannot be omitted
in an {\it $n+3$--square} model for such a  model necessarily gives  a {\it complete} $n+3$--square representation
of $\A$, which gives an (ordinary) $n+3$--square representation of $\Cm\At\A$. This in turn induces an $n+3$--square representation of $\CA_{n+1, n}$,
because $\CA_{n+1, n}$ embeds into $\Cm\At\A$. Then  by lemma \ref{square} \pe\ has a \ws\ in $G_{\omega}^{n+3}(\At\CA_{n+1, n}),$
{\it a fortiori} in the game $F^{\omega}(\At\CA_{n+1, n})$ (in only finitely many rounds) where he is allowed to use the $n+3$ nodes in play.
This contradicts the proof of theorem \ref {can}.  
We have shown that $\Gamma$ is realized in every $n+3$ model, but it cannot be isolated by a formula using only
$n$-variables, lest $\Gamma$ will be a principal type, which is not the case. 
We conclude, as required, that $\Psi(n, n+3)$ is true.

Now we prove the second item. The idea is that one blows up and blur $\R_m$ in place of the Maddux algebra 
$\mathfrak{E}_k(2, 3)$ having $k$ non identity atoms dealt with in \cite[Lemma 5.1]{ANT}.   
We use the notation in \cite{ANT}.
Fix $2<n\leq l\leq m<\omega$.  Denote $\R_m$ by $\R$. Let $(J, E)$ be the strong $m$--blur of $\R$. Since $l\leq m$, 
$(J, E)$ is a strong $l$--blur of $\R$. Let ${\cal R}={\sf Bb}(\R, J, E)$ with atom structure  
whose underlying set is denoted by $At$ on \cite[p.73]{ANT}. 
 Let $\C_n={\sf Bl}_n(\R, J, E)\in \CA_n$ defined as in \cite[Top of p. 78]{ANT}. By  $l$--blurness we 
get that for  $\cal R$, the set ${\sf Mat}_l(\At\cal R)$ of $l$ by $l$ dimensional matrices on $\cal R$
is an $l$--dimensional cylindric basis \cite[Theorem 3.2]{ANT} which is 
the atom structure of $\C_l={\sf Bl}_l(\R, J, E)$, and by strong blurness we get by \cite[item (3) p.80]{ANT} that  $\C_n=\Nr_n\C_l$.  
The algebra $\Cm\At\C_n$ does not have an $m+1$-- flat representation because $\R$ embeds into 
$\Cm\At(\cal R)$ which embeds into $\sf Ra\Cm\At\C_n$. So an $m+1$--dimensional flat representation of $\Cm\At\C_n$ induces an infinite $m+1$--
flat representation of $\R$. But this is impossible because by hypothesis 
$\R$ does not have an infinite $m+1$--dimensional hyperbases. Hence $\C_n$ does not have a complete $m+1$--flat representation.

Like before, we can assume that $\C_n= \Fm_T$ for a countable, atomic theory $L_n$ theory $T$.  
Let $\Gamma$ be the type consisting of co--atoms of $T$.
Then $\Gamma$ is realizable in every $m+1$--flat model, for if $M$ is an $m+1$--flat model omitting 
$\Gamma$, then it induces a complete  $m+1$--flat  representation  of $\Fm_T=\C_n.$ 
Since $\C_n\in {\sf Nr}_n\CA_l$,  then, using the argument in \cite[Theorem 3.1]{ANT}, 
any witness isolating $\Gamma$  needs more than $l$--variables. For squareness, the reasoning is the same by observing 
that here we have to exclude the existence of {\it any $m+1$--dimensional relational  basis of $\R$ even a finite one} for 
and any finite relation algebra having an (infinite)  $m+1$--dimensional relational basis, 
has a finite one \cite[Theorem 19.18]{HHbook}.  This, in turn, 
excludes the existence of an $m+1$--dimensional square representation 
of $\Cm\At\C_n$, since, like before $\R$ embeds into 
$\Ra\Cm\At\C_n$, and we are done. 
\end{proof}
Now we prove differently the results on non--finite axiomatizability given in example \ref{Monk} above. In fact, we give a finer result.
Fix $2<n<\omega$. For each $2<n\leq l<\omega$, 
let $\R_l$ be the finite Maddux algebra $\mathfrak{E}_{f(l)}(2, 3)$ with strong $l$--blur
$(J_l,E_l)$ and $f(l)\geq l$ as specified in theorem \ref{2.12}.  
Let ${\cal R}_l={\sf Bb}(\R_l, J_l, E_l)\in \sf RRA$ and 
let $\A_l=\Nr_n{\sf Bb}_l(\R_l, J_l, E_l)\in \RCA_n$, as defined on \cite[pp.78]{ANT} and the second part of 
theorem \ref{2.12}.  Then  $(\At{\cal R}_l: l\in \omega\sim n)$, and $(\At\A_l: l\in \omega\setminus n)$ are sequences of weakly representable atom structures 
that are not strongly representable with a completely representable 
ultraproduct. 
The (complex algebra) sequences $(\Cm \At{\cal R}_l: l\in \omega\setminus n)$, 
$(\Cm\At\A_l: l\in \omega\setminus n$) are typical examples of `bad' Monk (non--representable) algebras 
converging to  a 'good' (representable) one, namely, their (non--trivial) ultraproduct. 
Here we also have that for $2<n\leq k <m<\omega$, $\A_k=\Nr_k\A_m$. 
Using a standard Los argument, this recovers Monk's and Maddux's stronger result \cite{Monk, Maddux} 
on non--finite axiomatizability of $\sf RRA$s and $\RCA_n$s
and the class of algebras satisfying the Lyndon conditions for both $\sf RRA$ and $\CA_n$ 
since algebras considered are generated by a 
single $2$--dimensional element \cite{ANT, Sayed}.

\section{Summary of results in tabular form}
Throughout this section fix $2<n<\omega$.

{\bf Atom--canonicity:} In the next table we address atom--canonicity for classes of relation and cylindric algebras. 
Such classes are defined via the operators ${\sf Ra}$  of taking relation algebras reducts, and $\sf Nr_n$  
of taking $n$--neat reducts, respectively. $\sf RRA$ denotes the class of representable $\RA$s.
In the table $k\geq 3$ and $m\geq 6$. It is known that $\bold S{\sf Nr}_n\CA_{n+1}$, $\bold S{\sf Ra}\CA_3$ and $\bold S{\sf Ra}\CA_4$ are 
atom--canonical; the first class of $\CA_n$s admits a finite Sahlqvist axiomatization, a result of Andr\'eka, and for the relation algebra cases, cf. \cite[p.531]{HHbook}.
\vskip2.2mm
\begin{tabular}{|l|c|c|c|c|c|c|}    \hline
					Algebras	                     &{\sf Atom--canonical}                         & {\sf Canonical}  \\

                                                               \hline
                                                                          $\sf RCA_n, \sf RRA$&no &yes\\

                                                               \hline
                                                                          $\bold S{\sf Nr}_n\CA_{n+2}$, $\bold S{\sf Ra}\CA_{5}$&?  &yes\\

                                                               \hline
                                                                          $\bold S{\sf Nr}_n\CA_{n+k}$, $\bold S{\sf Ra}\CA_{m}$&no, thm \ref{can}, \cite[Thm 17.37]{HHbook}&yes\\

\hline

\end{tabular}

\vskip2.3mm
The results in the second row are known \cite{HH, Hodkinson, HHbook, strongca}. 
The third row involves open questions. It can be shown that $\bold S{\sf Ra}\CA_{5}$ and $\bold S{\sf Nr}_n\CA_{n+2}$ are not atom--canonical if there exists a finite relation algebra 
having an $n$--blur (not neccessarily strong) 
but has no infinite  $n+2$--dimensional hyperbasis in the sense of \cite[Definition 12.11]{HHbook} 
by using the argument in  the proof of the second item of theorem \ref{2.12}. 

{\bf First order definability:} In the next table the answers in the third column are to the question as to whether $\bold K$ is elementary or not. In all cases considered it is not. 
 
\vskip2.3mm

\begin{tabular}{|l|c|c|c|c|c|c|}    \hline
					Relation algebras	& Cylindric algebras                     &$\equiv$                                      \\

                                                               \hline
                                                                          $\bold S_d{\sf Ra}\CA_{\omega}\subseteq \bold K\subseteq \bold S_c{\sf Ra}\CA_5$  &$\bold S_d{\sf Nr}_n\CA_{\omega}\subseteq \bold K\subseteq 
\bold S_c{\sf Nr}_n\CA_{n+3}$&no, thm. \ref{rainbow}(2), \cite{r}\\

                                                                           \hline
                                                                          ${\sf Ra}\CA_{\omega}\subseteq \bold K\subseteq \bold {\sf Ra}\CA_5$   &${\sf Nr}_n\CA_{\omega}\subseteq \bold K\subseteq \bold 
{\sf Nr}_n\CA_{n+1}$& no, \cite{bsl}, (4) thm \ref{rainbow}\\

\hline

\end{tabular}

\vskip2.3mm
    
We still do not know whether {\it there  is an elementary class} between
${\sf Nr}_n\CA_{\omega}$ and $\bold S_d{\sf Nr}_n\CA_{\omega}$, witness item(2) of theorem \ref{rainbow} and to the best of our knowledge 
whether $\bold S_c{\sf Nr}_n\CA_{n+k}$ for $k=1, 2$ is first order definable or not, remains so far unsettled.\footnote{ Fix $2<n<\omega$. We devise an $\omega$--rounded {\it non--atomic} game 
$\bold G$ such that if $\B$ is an atomic algebra having countably many atoms, then a \ws\ for \pe\ in 
$\bold G(\B)\implies
\B\in {\sf Nr}_n\CA_{\omega}$. 
This game $\bold G$ is strictly stronger than $H$ used in the second item
of theorem \ref{rainbow}, by theorem \ref{stronger}. The game is played on both $\lambda$--neat hypernetworks as defined in the proof of the second 
item of theorem \ref{rainbow}, and complete labelled graphs (possibly by non--atoms)
with no consistency conditions.  The play at a certain point, like in $H$ as in the second item of theorem \ref{rainbow},  will be a
$\lambda$--neat hypernetwork, call its
network part $X$, and we write
$X(\bar{x})$ for the atom the edge $\bar{x}$. By the network part we mean forgetting hyperedges getting 
non--atomic labels.
{\it An $n$-- matrix} is a finite complete graph with nodes including $0, \ldots, n-1$
with all edges labelled by {\it arbitrary elements} of $\B$. No consistency properties are assumed.
\pa\ can play an arbitrary $n$--matrix  $N$, \pe\ must replace $N(0, \ldots, n-1),$  by
some element $a\in \B$; this is a non-atomic move.
The final move is that \pa\ can pick a previously played $n$--matrix $N$, and pick any  tuple $\bar{x}=(x_0,\ldots, x_{n-1})$
whose atomic label is below $N(0, \ldots, n-1)$.
\pe\ must respond by extending  $X$ to $X'$ such that there is an embedding $\theta$ of $N$ into $X'$
 such that $\theta(0)=x_0\ldots , \theta(n-1)=x_{n-1}$ and for all $i_0, \ldots i_{n-1} \in N,$ we have
$$X(\theta(i_0)\ldots, \theta(i_{n-1}))\leq N(i_0,\ldots, i_{n-1}).$$
This ensures that in the limit, the constraints in
$N$ really define the element $a$.
Assume that $\B\in \CA_n$ is atomic and has countably many atoms. Assume that  \pe\ has a \ws\ in $\bold G(\B).$ 
Then the extra move involving non--atoms labelling matrices, ensures that  that every $n$--dimensional element generated by
$\B$ in a  dilation $\D\in \RCA_\omega$ having base $M$  
constructed from a \ws\ in $\bold G$ as the 
limit of the  $\lambda$--neat hypernetworks played during the game 
(and further assuming without loss that \pa\ plays every possible move)  
is already an element of $\B$.  Here we can work in $L_{\omega, \omega}$ as a vehicle for constructing $\D$; we do not need infinite conjunctions. 
In this case,  we have $\Sg^{\D}X=\Sg^{\Nr_n\D}X\cong \Nr_n\Sg^{\D}X$ for any $X\subseteq \B$; in particular, $\B=\Nr_n\D$ 
(which is stonger than $\At\B\cong \At\Nr_n\D$ as shown in item(3) of theorem \ref{rainbow} and theorem \ref{stronger}).
 For $k<\omega$, let $\bold G_k$ be the game $\bold G$ truncated to $k$ rounds.
Using the argument in the proof of the second item of theorem \ref{rainbow} replacing $H$ by $\bold G$ one shows that for $2<n<m<\omega$, if  
there exists a countable atom structure $\alpha$ such that \pe\ has a \ws\ in $\bold G_k(\Cm\alpha)$
for all $k\in \omega$ and \pa\ has a \ws\ in $F^m$, then any class $\bold K$, 
such that ${\sf Nr}_n\CA_{\omega}\subseteq \bold K\subseteq \bold S_c{\sf Nr}_n\CA_m$,
is not elementary. For $\alpha=\At\CA_{\Z, \N}$ and $m=n+3$ it is not hard to show that \pa\ has a \ws\ in $\bold G_k$ for some finite $k>2$ so this atom structure does not work here.}

{\bf Omitting types:} In the next table the status of $\Psi(l, m)$ and $\Psi(l, m)_f$ (as defined after lemma \ref{square}) is given for various values 
of $l$ and $m$ where $2<n\leq l<m\leq \omega$. The formula $\Psi(\omega, \omega)$ is the limiting case when models are 
ordinary and the number of variables used are $\omega$. In the table $m$--hyp is short hand for {\it infinite} $m$--dimensional
hyperbasis and $m$ basis is short for $m$--dimensional relational 
basis  ($3<m\leq \omega$). $\sf VT$ is short for Vaught's theorem: Any countable complete atomic $L_{\omega, \omega}$ 
theory has an atomic countable model. 
\vskip2mm
\begin{tabular}{|l|c|c|c|c|c|c|}    \hline
					
                                                                          $\Psi(n, \omega)$&yes, \cite{ANT}\\

                                                               \hline
                                                                          $\Psi(n, n+3)$& yes, thm \ref{2.12} \\

                                                               \hline
                                                                          $\Psi(n, n+2)_f$&yes, if there is $\R$ with $n$--blur and no  $n+2$-hyp, thm \ref{2.12} \\

                                                               \hline
                                                                          $\Psi(l, \omega)$&yes, $\mathfrak{E}_{k}(2,3)$ has strong $l$-blur, and no $\omega$-hyp, thm \ref{2.12}.\\

                                                                           \hline
                                                                          $\Psi(l, m)_f, l\leq m-1$ &yes, if there exists $\R$ with strong $l$-blur, and no $m$-hyp, thm \ref{2.12} \\

                                                                           \hline
                                                                          $\Psi(l, m), l\leq m-1$ &yes, if there exists $\R$ with strong $l$-blur, and no $m$-bases, \ref{2.12}\\

                                                                           \hline
                                                                          $\Psi(\omega, \omega)$&no, by {\sf VT}\\

                                                                          \hline

\end{tabular}


\begin{thebibliography}{}

\bibitem{Andreka} Andr\'eka, H., {\it Complexity of equations valid in algebras of relations}. 
Annals of Pure and Applied logic, {\bf 89} (1997),  pp.149--209.


\bibitem{1} H. Andr\'eka, M. Ferenczi and  I. N\'emeti, (Editors), {\bf Cylindric-like Algebras and Algebraic Logic},
Bolyai Society Mathematical Studies and Springer-Verlag, {\bf 22} (2012).

\bibitem{ANT}  H. Andr\'eka,  I. N\'emeti and T. Sayed Ahmed,  {\it Omitting types for finite variable fragments and complete representations.}
Journal of Symbolic Logic. {\bf 73} (2008) pp. 65--89.

\bibitem{HMT1}  L. Henkin, J.D. Monk and  A. Tarski {\it Cylindric Algebras Part I}.
North Holland, 1971.

\bibitem{HMT2}  L. Henkin, J.D. Monk and  A. Tarski {\it Cylindric Algebras Part I}.
North Holland, 1985.

\bibitem{t} R. Hirsch and T. Sayed Ahmed, {\it The neat embedding problem for algebras other than cylindric algebras
and for infinite dimensions.} Journal of Symbolic Logic {\bf 79}(1) (2014) pp.208--222.

\bibitem{r} R. Hirsch, {\it Relation algebra reducts of cylindric algebras and complete representations},
Journal of Symbolic Logic, {\bf 72}(2) (2007), pp.673--703.

\bibitem{r2} R. Hirsch {\it Corrigendum to `Relation algebra reducts of cylindric algebras and complete representations'} Journal of Symbolic Logic,
{\bf 78}(4) (2013), pp. 1345--1348.

\bibitem{HH} R. Hirsch and I. Hodkinson {\it Complete representations in algebraic logic},
Journal of Symbolic Logic, {\bf 62}(3)(1997) pp. 816--847.

\bibitem{HHbook}  R. Hirsch and I. Hodkinson,  {\it Relation algebras by games.}
Studies in Logic and the Foundations of Mathematics, {\bf 147} (2002).


\bibitem{strongca} R. Hirsch and I. Hodkinson {\it Strongly  representable atom structures  of cylindic  algebras}, Journal of Symbolic Logic 
{\bf 74}(2009), pp. 811--828

\bibitem{HHbook2} R. Hirsch and I. Hodkinson {\it  Completions and complete representations}, in \cite{1} pp. 61--90.


\bibitem{Hodkinson} I. Hodkinson  I., {\it Atom structures of relation and cylindric algebras.} Annals of pure and applied logic,
{\bf  89} (1997), pp. 117--148.

\bibitem {AU} I. Hodkinson, \emph{A construction of cylindric and polyadic algebras from atomic relation algebras.}
Algebra Universalis, {\bf 68} (2012), p. 257--285.


\bibitem{Maddux}  R. Maddux  {\it Non finite axiomatizability results for cylindric and relation algebras},
Journal of Symbolic Logic (1989) {\bf 54}, pp. 951--974.


\bibitem{Monk}  J.D. Monk  {\it Non finitizability of
classes of representable cylindric algebras}, Journal of Symbolic Logic {\bf 34}(1969) pp. 331--343.

\bibitem{bsl} T. Sayed Ahmed, {\it ${\sf Ra}\CA_n$ is not elementary for $n\geq 5$},
Bulletin Section of Logic, {\bf 37}(2)(2008), pp. 123--136.

\bibitem{weak} T. Sayed Ahmed, {\it Weakly representable atom structures that are not strongly representable,
with an application to first order logic,} Mathematical Logic Quarterly, {\bf 54}(3)(2008) pp. 294--306



\bibitem{Sayedneat}  T. Sayed Ahmed, {\it Neat reducts and neat embeddings in cylindric algebras}, in \cite{1}, pp. 105--134.

\bibitem{Sayed}  T. Sayed Ahmed  {\it Completions, Complete representations and Omitting types}, in \cite{1}, pp. 186--205.

\bibitem{au} T. Sayed Ahmed, {\it The class of completely representable polyadic algebras is elementary}. 
Algebra Universalis, {\bf 72} (2014), pp. 371--380. 

\bibitem{mlq} T. Sayed Ahmed {\it On notions of representability for cylindric--polyadic algebras and a solution to the finitizability problem for first order logic with equality}. 
Mathematical Logic Quarterly, in press.

\bibitem{SL}  T. Sayed Ahmed  and I. N\'emeti,  {\it On neat reducts of algebras of logic}, Studia Logica. {\bf 68(2)} (2001), pp. 229--262.

\end{thebibliography}
\end{document}